\documentclass{amsart}
\pdfoutput=1

\usepackage{fullpage}
\usepackage{amssymb}   
\usepackage{amsmath}
\usepackage{amscd}
\usepackage{mathrsfs}
\usepackage{stmaryrd}
\usepackage[numbers, sort&compress]{natbib}

\usepackage{xr} 
\usepackage{tikz}
\usetikzlibrary{matrix,arrows,patterns,calc,decorations.markings}
\usepackage{paralist}
\usepackage{CommonMath}
\usepackage{xargs}

\newcommand{\comment}[1]{}

\newcommand {\Z}{{\bf Z}}
\newcommand {\Q}{{\bf Q}}

\newcommand {\C}{{\bf C}}
\renewcommand {\P}{{\bf P}}
\newcommand {\from}{{\colon}}
\newcommand{\into}{{\hookrightarrow}}
\newcommand{\onto}{\twoheadrightarrow}
\newcommand{\st}[1]{\mathscr{#1}}      
\newcommand{\sh}[1]{{\mathscr #1}}        
\newcommand{\orb}[1]{{\mathcal #1}}     
\newcommand{\cat}[1]{\mathbf{#1}}       
\renewcommand{\sp}[1]{#1}               
\renewcommand {\o}[1]{\overline{#1}}    
\newcommand{\ideal}[1]{\langle #1 \rangle}     
\newcommand{\dual}[1]{#1^{\vee}}
\newcommand{\isom}{\stackrel\sim\longrightarrow}
\DeclareMathOperator{\Crimp}{Crimp}
\DeclareMathOperator{\Mor}{Mor}
\DeclareMathOperator{\Def}{Def}
\DeclareMathOperator{\Gl}{Gl}

\DeclareMathOperator{\supp}{supp}
\DeclareMathOperator{\spec}{Spec}
\DeclareMathOperator{\proj}{Proj}
\DeclareMathOperator{\mult}{mult}
\DeclareMathOperator{\br}{br}
\DeclareMathOperator{\Sym}{Sym}
\DeclareMathOperator{\id}{id}
\DeclareMathOperator{\tr}{tr}
\DeclareMathOperator{\length}{length}
\DeclareMathOperator{\G}{{\mathbf G}}
\DeclareMathOperator{\A}{{\mathbf A}}
\DeclareMathOperator{\Bl}{Bl}
\DeclareMathOperator{\Ext}{Ext}

\DeclareMathOperator{\cok}{cok}
\DeclareMathOperator{\Isom}{Isom}
\DeclareMathOperator{\Aut}{Aut}

\DeclareMathOperator{\F}{\mathbf F}
\DeclareMathOperator{\Exc}{Exc}

\DeclareMathOperator{\Pic}{Pic}
\DeclareMathOperator{\ch}{ch}
\DeclareMathOperator{\td}{td}
\DeclareMathOperator{\Hom}{Hom}

\DeclareMathOperator{\Hilb}{Hilb}

\DeclareMathOperator{\im}{im}

\newcommand{\f}[1]{\llbracket #1 \rrbracket}

\newcommand{\compl}[1]{\widehat{#1}}
\newcommand{\tw}[1]{\widetilde{#1}}

\newcommand{\sub}{{\rm sub}}
\newcommand{\Br}{{\rm Br}}
\newcommand{\s}{\mathbf S}

\tikzset{math/.style = {execute at begin node=$, execute at end node=$}}
\tikzset{map/.style = {font=\scriptsize}}
\tikzset{equal/.style = {double distance=.15em}}
\newcommandx{\cartsquare}[3][3=.35]{
    \draw ($#1!#3!#2$) rectangle ($#2!#3!#1$);
}

\numberwithin{equation}{section}

\newcounter{thmalph}

\newtheorem{introtheorem}[thmalph]{Theorem}

\title{Modular compactifications of the space of marked trigonal curves} 
\author{Anand Deopurkar}
\email{anandrd@math.harvard.edu}
\address{Harvard University\\
  One Oxford Street, Cambridge, MA, 02139}

\begin{document}
\begin{abstract}
  We construct a sequence of modular compactifications of the space of marked trigonal curves by allowing the branch points to coincide to a given extent. Beginning with the standard admissible cover compactification, the sequence first proceeds through contractions of the boundary divisors and then through flips of the so-called Maroni strata, culminating in a Fano model for even genera and a Fano fibration for odd genera. While the sequence of divisorial contractions arises from a more general construction, the sequence of flips uses the particular geometry of triple covers. We explicitly describe the Mori chamber decomposition given by this sequence of flips.
\end{abstract}

\maketitle
\section{Introduction}\label{sec:intro}
Moduli spaces of geometrically interesting objects are usually not compact. They need to be compactified by adding to them a carefully chosen class of degenerate objects. Often, this can be done in several ways, leading to birationally equivalent models exhibiting highly interesting instances of the Minimal Model Program (MMP). Such a program has been the topic of intense current research most importantly for the moduli space of curves but also for related moduli spaces such as the spaces of pointed curves \citep{hassett03:_modul,smyth11:_modul_i,smyth11:_modul_ii}, the Kontsevich spaces \citep{chen08:_moris_konts_m_bbb_p}, the space of hyperelliptic curves \citep{fedorchuk10:_modul_hyp}, and so on. In \citep{deopurkar12:_compac_hurwit}, we took up such a study of the Hurwitz spaces of branched covers by constructing compactifications where the branch points are allowed to coincide to various extents. In this paper, we describe a beautiful picture that emerges from this theme for the space of marked triple covers. The picture consists of two parts. The first is a sequence of contractions of the boundary divisors obtained using the results of \citep{deopurkar12:_compac_hurwit}. The second is a continuation of the first sequence by flips of the Maroni strata and uses the particular geometry of triple covers. The second sequence culminates in a model of the kind 
predicted by the MMP: a Fano variety for even genera and a Fano fibration over $\P^1$ for odd genera.

Denote by $T_{g;1}$ the moduli space of simply branched trigonal curves of genus $g$ along with a marked unramified fiber of the trigonal map. In symbols, $T_{g;1} = \{(\phi \from C \to P; \sigma)\}$, where $P$ is a copy of $\P^1$, $C$ a smooth curve of genus $g$, $\phi$ a simply branched triple cover and $\sigma \in P$ a point away from $\br\phi$. The space $T_{g;1}$ is a quasi-projective, irreducible variety of dimension $2g+2$.

How can we compactify $T_{g;1}$? A reasonable approach is to compactify the space of branch divisors and then use it to compactify the space of branched covers. In this case, the relevant space is $M_{0;b,1}$, the moduli space of $(P; \Sigma; \sigma)$, where $P$ is a copy of $\P^1$, $\Sigma \subset P$ a reduced divisor of degree $b = 2g+4$ and $\sigma \in P$ a point away from $\Sigma$. We have a branch morphism $\br \from T_{g;1} \to M_{0;b,1}$ given by
\[ \br \from (\phi \from C \to P; \sigma) \mapsto (P; \br\phi; \sigma).\]
By the main result of \citep{deopurkar12:_compac_hurwit}, every suitable compactification of $M_{0;b,1}$ yields a corresponding compactification of $T_{g;1}$. Our first sequence of compactifications of $T_{g;1}$ is obtained in this way using the compactifications of $M_{0;b,1}$ by weighted pointed stable curves of \citep{hassett03:_modul}. Let $\epsilon$ be such that $1/b < \epsilon \leq 1$ and let $\o M_{0;b,1}(\epsilon)$ be the space of $(\epsilon;1)$-stable marked curves as in \citep{hassett03:_modul}. Roughly speaking, it is the compactification of $M_{0;b,1}$ in which $\lfloor 1/\epsilon \rfloor$ of the $b$ points of $\Sigma$ are allowed to collide.
\begin{introtheorem}\label{intro:first_seq}
  Let $\o{\orb T}_{g;1}(\epsilon)$ be the moduli stack of $(\phi \from C \to P; \sigma)$, where $\phi$ is an $\epsilon$-admissible triple cover and $\sigma \in P$ a point away from $\br\phi$. Then $\o {\orb T}_{g;1}(\epsilon)$ is a smooth proper Deligne--Mumford stack with a projective coarse space containing $T_{g;1}$ as a dense open subspace and admitting an extension of the branch morphism $\br \from \o {T}_{g;1}(\epsilon) \to \o M_{0;b,1}(\epsilon)$. 
\end{introtheorem}
\begin{proof}
  See [\autoref{thm:first_seq}] in the main text.
\end{proof}
The reader may think of an $\epsilon$-admissible cover $\phi \from C \to P$ to be a generically \'etale cover, admissible over the nodes of $P$, in which at most $\lfloor 1/\epsilon \rfloor$ of the branch points are allowed to coincide; the precise definition of $\o{\orb T}_{g;1}(\epsilon)$ is in \autoref{sec:hd}. Clearly, we may restrict the values of $\epsilon$ to reciprocals of integers. We thus arrive at the sequence
\begin{equation}\label{eqn:Tg1_contractions}
  \o T_{g;1}(1) \dashrightarrow \cdots \dashrightarrow \o T_{g;1}(1/j)\dashrightarrow \o T_{g;1}(1/(j+1))\dashrightarrow \cdots \dashrightarrow \o T_{g;1}(1/(b-1)).
\end{equation}

What is the nature of these modifications? For $\epsilon \geq \epsilon'$, the birational map $\o T_{g;1}(\epsilon) \dashrightarrow \o T_{g;1}(\epsilon')$, lying over the divisorial contraction $\o M_{0;b,1}(\epsilon) \to \o M_{0;b,1}(\epsilon')$, is regular in codimension one, but not regular in general. It contracts the boundary divisors in $\o T_{g;1}(\epsilon)$ lying over the contracted boundary divisors of $\o M_{0;b,1}(\epsilon)$. In the final model, $(b-1)$ of the $b$ branch points are allowed to coincide.

To what extent does the sequence \eqref{eqn:Tg1_contractions} realize the MMP for $\o T_{g;1}(1)$? In this sense, it is incomplete in two ways. Firstly, the  maps $\o{\sp T}_{g;1}(1/j) \dashrightarrow \o{\sp T}_{g;1}(1/(j+1))$ are not everywhere regular. Secondly, the final model $\o T_{g;1}(1/(b-1))$ is not what is expected to be an ultimate model. It is easy to see that $T_{g;1}$ is unirational (in particular, uniruled). In this case, the MMP is expected to end with a Fano fibration. The final model of \eqref{eqn:Tg1_contractions} is not of this form.

How can we extend \eqref{eqn:Tg1_contractions} to reach a Fano fibration? The bulk of this paper is devoted to answering this question. In the final model of \eqref{eqn:Tg1_contractions}, $(b-1)$ of the $b$ branch points were allowed to coincide. We could now allow all $b$ to coincide. Note, however, that there is no compactification of $M_{0;b,1}$ that allows all $b$ points of $\Sigma$ to coincide. Therefore, such a compactification of $T_{g;1}$ (if one exists) must be constructed `by hand'; we cannot rely on the machinery of \citep{deopurkar12:_compac_hurwit}. Consider for a moment a marked cover $(\phi \from C \to \P^1; \sigma)$ with concentrated branching; say its branch divisor is concentrated at a point $p \in \P^1$. The moduli of such a cover is completely captured by the moduli of the cover around $p$, in other words, by the moduli of the `trigonal singularity.' Which trigonal singularities should  we allow in our moduli problem? Correspondingly, which triple covers must we disallow? It turns out that there is not one, but many ways of arranging this trade-off, leading to a sequence of models of $T_{g;1}$ in which covers with concentrated branching are allowed. As expected, the trade-off involves a local and a global restriction.

The global restriction is in terms of the classical Maroni invariant of a trigonal curve. Let $C$ be a reduced, connected curve of genus $g$ and $\phi \from C \to \P^1$ a triple cover. We have
\[ \phi_* O_C/O_{\P^1} \cong O_{\P^1}(-m) \oplus O_{\P^1}(-n),\]
(as $\P^1$-modules) for some positive integers $m$, $n$ with $m+n = g+2$. The \emph{Maroni invariant} $M(\phi)$ is defined to be the difference $|n-m|$. 

The local restriction is in terms of a similarly defined `$\mu$ invariant' of a trigonal singularity. Let $C$ be the germ of a reduced, connected curve singularity expressed as a triple cover $\phi \from C \to \Delta$, where $\Delta$ is the germ of a smooth curve. Let $\tw C \to C$ be the normalization, and assume that $\tw C \to \Delta$ is \'etale. We have
\[ O_{\tw C}/O_C \cong k[t]/t^m \oplus k[t]/t^n,\]
(as $\Delta$-modules) for some positive integers $m$, $n$ with $m + n = \delta$, the delta invariant of $C$. The \emph{$\mu$ invariant} $\mu(\phi)$ is defined to be the difference $|n-m|$. Note that if $C$ is a piece of a global curve of arithmetic genus $g$, then $\delta = g+2$.

The local-global trade-off is encoded in the notion of an `$l$-balanced triple cover.'
\begin{definition}\label{def:l-balanced}
  Let $C$ be a reduced, connected curve and $\phi \from C \to \P^1$ a triple cover. We say that $\phi$ is \emph{$l$-balanced} if
  \begin{compactenum}
  \item the Maroni invariant of $\phi$ is at most $l$;
  \item if $\br\phi$ is concentrated at one point, then the $\mu$ invariant of $\phi$ is greater than $l$.
  \end{compactenum}
\end{definition}
\begin{introtheorem}\label{intro:second_seq}
  Let $\o {\orb T}_{g;1}^l$ be the moduli stack of $(\phi \from C \to P; \sigma)$, where $P$ is a copy of $\P^1$, $C$ a curve of genus $g$, $\phi$  an $l$-balanced triple cover, and $\sigma \in P$ a point away from $\br\phi$. Then $\o {\orb T}_{g;1}^l$ is a smooth proper Deligne--Mumford stack admitting a projective coarse space $\o T_{g;1}^l$ birational to $T_{g;1}$.
\end{introtheorem}
\begin{proof}
See \autoref{thm:tg3d} and \autoref{thm:flip_projectivity} in the main text.
\end{proof}
Since both the Maroni invariant and the $\mu$ invariant lie between $0$ and $g$ and are congruent to $g$ modulo $2$, we may assume that $l$ is also of this form. We thus get a sequence of yet more compactifications
\begin{equation}\label{eqn:flips}
 \o{\sp T}_{g;1}^g \dashrightarrow
 \dots \dashrightarrow
 \o{\sp T}_{g;1}^l \dashrightarrow \o{\sp T}_{g;1}^{l-2}
 \dashrightarrow \dots
 \dashrightarrow \o{\sp T}_{g;1}^{0 \text{ or } 1}.
\end{equation}
It is immediate that for $l = g$, the space $\o T_{g;1}^l$ is nothing but $\o T_{g;1}(1/(b-1))$. As $l$ decreases, the loci of Maroni unbalanced curves are replaced one-by-one by loci of covers with concentrated branching, culminating in a space where only the most balanced covers are kept.

We must again ask the same questions as before: what is the nature of the modifications in the sequence \eqref{eqn:flips}, and does it reach the model predicted by the MMP? The answers are a part of the following theorem that offers a detailed analysis of these transformations.
\begin{introtheorem}\label{intro:flips}
  Denote by $\beta_l$ the rational map $\o T_{g;1}^l \dashrightarrow \o T_{g;1}^{l-2}$.
  \begin{compactenum}
  \item The map $\beta_g$ extends to a morphism that contracts the hyperelliptic divisor to a point. For even $g$, the map $\beta_2$ extends to a morphism that contracts the Maroni divisor to a $\P^1$. All the other $\beta_l$ are flips.
  \item For even $g$, the final space $\o T_{g;1}^0$ is Fano of Picard rank one. It is the quotient of a weighted projective space by $\s_3$. For odd $g$, the final space $\o T_{g;1}^1$ admits a morphism to $\P^1$ whose fibers are Fano of Picard rank one.
  \item For $0 < l < g$, the Picard group of $\o T_{g;1}^l$ has rank two. For $g \neq 3$, it is generated by the Hodge class $\lambda$ and the class of the boundary divisor $\delta$. The canonical divisor is given by
  \[
  K = \frac{2}{(g+2)(g-3)}\left(3(2g+3)(g-1) \lambda - (g^2-3)\delta\right).
  \]
  The nef cone of $\o T_{g;1}^l$ in the $\ideal{\lambda, \delta}$ plane is spanned by $D_l$ and $D_{l+2}$, where 
  \[ 
  \left(\frac{g-3}{2}\right)D_l = \left\{(7g+6)\lambda -g \delta\right\} + \frac{l^2}{g+2}\cdot \left\{9\lambda - \delta\right\}.
  \] 
  (\autoref{fig:nef_cones} depicts this Mori chamber decomposition.)
\end{compactenum}
\end{introtheorem}
\begin{proof}
See \autoref{thm:hyperelliptic}, \autoref{thm:maroni}, \autoref{thm:final_even}, \autoref{thm:final_odd}, and \autoref{thm:mori_cones} in the main text.
\end{proof}
\begin{figure}[h]
  \centering
  \begin{tikzpicture}[math,scale=1,map]
    \draw[->] (0,0) -- (0:5) node [below] {-\delta};
    \draw[->] (0,0) -- (90:5) node [left] {\lambda};
    \path[dotted]
    (0,0) edge node [below, sloped, near end] {7+\frac6g \sim \langle D_0 \rangle} (20:5)
    (0,0) edge node [above, sloped, near end] {8+\frac{4}{g+1} \sim \langle D_{g+2}\rangle } (75:5)
    (0,0) edge node [below, sloped, near end] {6 \approx \langle K \rangle} (10:5)
    ;
    \fill[black!10!white] (0,0) -- (40:5) -- (55:5) -- (0,0);
    \fill[black!3!white] (0,0) -- (55:5) -- (70:5) -- (0,0);
    \fill[black!3!white] (0,0) -- (25:5) -- (40:5) -- (0,0);
    
    \draw
    (0,0) -- (25:5) node [right] {D_{l-2}}
    (0,0) -- (40:5) node [right] {D_{l}}
    (0,0) -- (55:5) node [above] {D_{l+2}}
    (0,0) -- (70:5) node [above] {D_{l+4}};
    
    \draw 
    (48:4.25) node {\o{\sp T}_{g;1}^l}
    (62:4.25) node {\o{\sp T}_{g;1}^{l+2}}
    (32:4.25) node {\o{\sp T}_{g;1}^{l-2}};
  \end{tikzpicture}
  \caption{The Mori chamber decomposition of $\Pic_\Q$ given by the models $\o T_{g;1}^l$ and an approximate location of the ray spanned by the canonical class $K$.}
  \label{fig:nef_cones}
\end{figure}
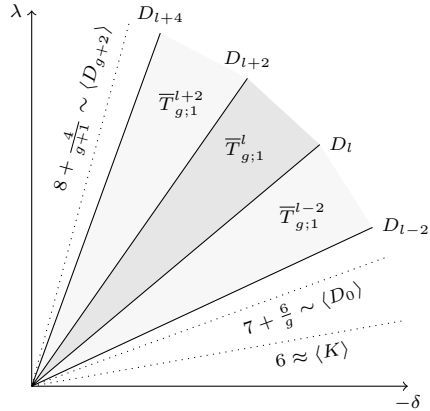

Finally, let us highlight a few key points. Firstly, the singularities that appear on the curves parametrized by our compactifications of $T_{g;1}$ are in a sense much nastier than the usually considered candidates for alternate compactifications of $M_g$ or other related moduli spaces of curves. In particular, our singularities are not even necessarily Gorenstein! Secondly, our approach is purely functorial---there is no GIT. We first construct the moduli stack, use the Keel--Mori theorem to get a coarse algebraic space and  prove projectivity by exhibiting ample line bundles. 

\subsubsection*{Organization}
In \autoref{sec:hd}, we recall the results of \citep{deopurkar12:_compac_hurwit}. They not only yield \autoref{intro:first_seq} at once, but also form the technical basis of all further constructions. In \autoref{sec:lbalanced}, we show that the stack $\o{\orb T}_{g;1}^l$ of marked $l$-balanced covers is a proper smooth Deligne--Mumford stack, yielding \autoref{intro:second_seq} except the projectivity of the coarse spaces. In \autoref{sec:prelim}, we take up the study of the coarse spaces $\o T_{g;1}^l$, starting with general results about triple covers and triple point singularities. In \autoref{sec:hyperelliptic} and \autoref{sec:maroni}, we treat the hyperelliptic and the Maroni contractions, respectively. The next major goal is to obtain ample line bundles and the Mori chamber decomposition. To that end, we study the Picard group in \autoref{sec:picard} and compute the ample cones in \autoref{sec:flip_ample}. We describe the final models in \autoref{sec:flip_final}.

So far, the marked fiber has been unramified, that is, of type $(1,1,1)$. The entire picture can be carried over to the space of trigonal curves marked with a fiber of any given ramification type: $(1,1,1)$, $(1,2)$ or $(3)$. This generalization is the content of the final section, \autoref{sec:ramified_fiber}.

\subsubsection*{Conventions}
We work over an algebraically closed field $k$ of characteristic zero. All schemes are understood to be locally Noetherian schemes over $k$. Unless specified otherwise, ``point'' means  ``$k$-point.'' An \emph{algebraic stack} or an \emph{algebraic space} is in the sense of \citep{laumon00:_champ}. 

The projectivization of a vector bundle $E$ is denoted by $\P E$; this is the space of one-dimensional \emph{quotients} of $E$. The space of one-dimensional sub-bundles of $E$ is denoted by $\P_\sub E$. A morphism $X \to Y$ is \emph{projective} if it factors as a closed embedding $X \into \P E$ followed by $\P E \to Y$ for some vector bundle $E$ on $Y$.

A \emph{curve} over a scheme $S$ is a flat, proper morphism whose geometric fibers are connected, reduced and purely one-dimensional. The source of this morphism could be a scheme, an algebraic space or a Deligne--Mumford stack; in the last case it is usually denoted by a curly letter. \emph{Genus} always means arithmetic genus. By the genus of a stacky curve, we mean the genus of its coarse space. A \emph{cover} is a representable, flat, surjective morphism. The symbol $\mu_n$ denotes the group of $n$th roots of unity; its elements are usually denoted by $\zeta$.

\section{The big Hurwitz stack and the spaces of weighted admissible covers}\label{sec:hd}
We recall the big Hurwitz stack $\st H^d$ from \citep{deopurkar12:_compac_hurwit}, beginning with the notion of pointed orbinodal curves. Let $S$ be a scheme. A \emph{pointed orbinodal curve} over $S$ is the data $(\orb P \to P \to S; \sigma_1, \dots, \sigma_n)$, where $\orb P \to S$ is a proper Deligne--Mumford stack which is an at worst nodal curve, $\orb P \to P$  the coarse space and $\sigma_1, \dots, \sigma_n: S \to P$ sections such that
\begin{itemize}
\item \'etale locally near a node, $\orb P \to P$ has the form
  \[ [\spec O_S[u,v]/(uv-t) / \mu_r] \to \spec k[x,y]/(xy-t^r),\]
  for some $t \in O_S$ with $\mu_r$ acting by $u \mapsto \zeta u, v \mapsto \zeta^{-1}v$;
\item \'etale locally near a section, $\orb P \to P$ has the form
  \[ [\spec O_S[u] / \mu_r] \to \spec O_S[x],\]
  with $\mu_r$ acting by $u \mapsto \zeta u$.
\end{itemize}

The big Hurwitz stack $\st H^d$ is defined by
\[ \st H^d(S) = \{(\phi \from \orb C \to \orb P; \orb P \to P \to S; \sigma_1, \dots, \sigma_n)\},\]
where $(\orb P \to P \to S; \sigma_1, \dots, \sigma_n)$ is a pointed orbinodal curve and $\phi$ a degree $d$ cover such that for every geometric point $s \to S$, the following are satisfied:
\begin{itemize}
\item $\phi$ is \'etale over the generic points of the components, the nodes, and the marked points of $\orb P_s$;
\item the stack structure is minimal in the following sense: the classifying map $\orb P_s \setminus \br\phi \to B\s_d$ corresponding to the \'etale cover obtained by restricting $\phi$ is representable.
\end{itemize}
\begin{remark}
What is the purpose of the orbi-structure? The orbi-structure on the nodes and on the marked points serve two different purposes. The orbi-structure on the nodes of $\orb P$ implies that the induced cover of the coarse spaces $C \to P$ is an admissible cover over the nodes of $P$ in the sense of \citep{Harris82:_Kodair_Dimen_Of_Modul_Space_Of_Curves}. Thus, the  orbi-structure on the nodes serves to impose the admissibility condition. On the other hand, an orbi-structure along $\sigma_i$ with automorphism group $\mu_r$ implies that the monodromy of the coarse space cover $C \to P$ around $\sigma_i$ is a permutation of order $r$ in $\s_d$. Thus, the orbi-structure along the marked points serves to (partially) encode the ramification type of $C \to P$ over those points.
\end{remark}

The stack $\st H^d$ should be seen as the stack of `all branched covers of curves.' It admits a branch morphism to a stack $\st M$, which should be seen as the stack of `all branching data.' The stack $\st M$ is defined by
\[ \st M(S) = \{(P \to S; \Sigma; \sigma_1, \dots, \sigma_n)\},\]
where $P \to S$ is an at worst nodal curve, $\Sigma \subset P$ an $S$-flat Cartier divisor supported in the smooth locus of $P \to S$ and $\sigma_1, \dots, \sigma_n \from S \to P$ disjoint sections lying in the smooth locus of $P \to S$ and away from $\Sigma$. The branch morphism $\br \from \st H^d \to \st M$ is simply
\[ \br \from (\phi \from \orb C \to \orb P; \orb P \to P \to S; \sigma_1, \dots, \sigma_n) \mapsto (P \to S; \br\phi; \sigma_1, \dots, \sigma_n).\]
Here the branch divisor $\br\phi$ is defined by the discriminant ideal (see \citep[\autoref*{p1:sec:Ad}]{deopurkar12:_compac_hurwit}). 
\begin{theorem}\label{thm:big_hurwitz}
  $\st H^d$ and $\st M$ are algebraic stacks, locally of finite type. The branch morphism $\br \from \st H^d \to \st M$ is proper and of Deligne--Mumford type.
\end{theorem}
\begin{proof}[Sketch of the valuative criterion]
  We refer the reader to \citep[\autoref*{p1:sec:proofs}]{deopurkar12:_compac_hurwit} for the full technical details, sketching here only the proof of the valuative criterion of properness. Let $\Delta$ be the spectrum of a DVR with special point $0$, generic point $\eta$ and uniformizer $t$. Let $(P \to \Delta; \Sigma; \sigma_1, \dots, \sigma_n)$ be an object of $\st M(\Delta)$ and $(\phi_\eta \from \orb C_\eta \to \orb P_\eta; \orb P_\eta \to P_\eta)$ an object of $\st H^d(\eta)$ with $\br\phi_\eta = \Sigma_\eta$. We must extend it to an object of $\st H^d(\Delta)$, possibly after a base change.

  It is not to hard to see that after a finite base change, the finite cover $\phi_\eta$ extends to a finite \'etale cover over the generic points of the components of $P_0$. We thus have an extension of $\phi$ except along finitely many points of $\orb P_0$. If $p \in \orb P_0$ is away from $\Sigma$, then $\phi$ is \'etale in the neighborhood of $p$ and hence is equivalent to a map to the classifying space $B\s_d$. By putting the appropriate orbi-structure around $p$, we can extend this classifying map by precisely following \citep{acv:03}.

  The final obstacle is extending $\phi$ over $p \in \Sigma|_0$. Since $\Sigma \subset P$ lies in the smooth locus of $P \to \Delta$, the surface $P$ is regular at $p$. Now the finite cover $\phi$ extends by \autoref{lem:horrocks}.
\end{proof}
\begin{lemma}\label{lem:horrocks}
  Let $S$ be a smooth surface and $s \in S$ a point. Set $S^\circ = S \setminus \{s\}$ and let $\phi^\circ \from T^\circ \to S^\circ$ be a finite flat morphism. Then $\phi^\circ$ extends to a finite flat morphism $\phi \from T \to S$.
\end{lemma}
\begin{proof}
  Consider the vector bundle $\phi^\circ_* O_{T^\circ}$ on $S^\circ$. By a theorem of Horrocks \citep[Corollary~4.1.1]{horrocks64:_vector}, it extends to a vector bundle on $S$. Next, the $O_T$-algebra structure on this vector bundle is specified by a structure map from $O_T$ and a multiplication map. These  maps of vector bundles are already specified on $S^\circ$; they extend to $S$  by Hartog's theorem.
\end{proof}

Denote by $\st M_{0;b,1} \subset \st M$ the open and closed substack where the genus of the curve is $0$, the degree of the marked divisor is $b$ and the number of marked points is $1$. Let $\orb M_{0;b,1} \subset \st M_{0;b,1}$ be the open substack parametrizing $(P; \Sigma; \sigma)$ with $P$ smooth and $\Sigma$ reduced. Denote by $\st T_{g;1} \subset \st H^3 \times_{\st M} \st M_{0;b,1}$ the open and closed substack parametrizing $(\phi \from \orb C \to \orb P; \orb P \to P; \sigma)$ where $\orb C$ is connected and $\Aut_\sigma\orb P = \{1\}$; that is, the stack structure of $\orb P$ over $\sigma$ is trivial. Let $\orb T_{g;1} \subset \st T_{g;1}$ be the preimage in $\st T_{g;1}$ of $\orb M_{0;b,1}$. 
\begin{proposition}\label{thm:irred_and_dim}
  $\st M_{0;b,1}$ and $\st T_{g;1}$ are smooth and irreducible of dimension $2g+2$. 
\end{proposition}
\begin{proof}
  The statement is clear for $\st M_{0;b,1}$.  For $\st T_{g;1}$, we appeal to \citep[\autoref*{p1:thm:smoothness_for_d23}]{deopurkar12:_compac_hurwit}, which implies that it is smooth and contains $\orb T_{g;1}$ as a dense open subset. In turn, it is clear that $\orb T_{g;1}$ is irreducible of dimension $2g+2$.

  For the convenience of the reader, we recall the main idea behind the proof of \citep[\autoref*{p1:thm:smoothness_for_d23}]{deopurkar12:_compac_hurwit}. The deformation space of a point $\xi = (\phi \from \orb C \to \orb P; \orb P \to P; \sigma)$ on $\st T_{g;1}$ is the product of the deformation spaces of the singularities on $\orb C$, up to smooth parameters. Since $\orb C \to \orb P$ is a triple cover \'etale over the nodes of $\orb P$, the singularities of $\orb P$ have embedding dimension at most three. Since spatial singularities are smoothable and have smooth deformation spaces, $\st T_{g;1}$ is smooth and contains $\orb T_{g;1}$ as a dense open.
\end{proof}

We now have the tools to obtain the compactifications of $T_{g;1}$ described in \autoref{intro:first_seq}. Let $\epsilon > 1/b$ be a rational number. Denote by $\o{\orb M}_{0;b,1}(\epsilon)$ the open substack of $\st M_{0;b,1}$ consisting of $(P; \Sigma; \sigma)$ that are $(\epsilon;1)$ stable in the sense of \citep{hassett03:_modul}; that is, they satisfy
\begin{compactenum}
\item $\sigma \not \in \Sigma$ and $\epsilon \cdot \mult_p(\Sigma) \leq 1$ for all $p \in P$;
\item $\omega_P(\epsilon \Sigma + \sigma)$ is ample.
\end{compactenum}
Then $\o{\orb M}_{0;b,1}(\epsilon)$ is a smooth, proper, Deligne--Mumford stack with a projective coarse space $\o M_{0;b,1}(\epsilon)$. Set 
\[
\o{\orb T}_{g;1}(\epsilon) = \st T_{g;1} \times_ {\st M_{0;b,1}} \o{\orb M}_{0;b,1}(\epsilon).
\]
\begin{theorem}\label{thm:first_seq}
  $\o {\orb T}_{g;1}(\epsilon)$ is a Deligne--Mumford stack, smooth and proper over $k$, and irreducible of dimension $2g+2$. It contains $\orb T_{g;1}$ as a dense open substack, admits a branch morphism $\br \from \o {\orb T}_{g;1}(\epsilon) \to \o {\orb M}_{0;b,1}(\epsilon)$ and a projective coarse space $\o T_{g;1}(\epsilon)$.
\end{theorem}
\begin{proof}
  All the statement except the projectivity of $\o T_{g;1}(\epsilon)$ follow from \autoref{thm:big_hurwitz}, \autoref{thm:irred_and_dim} and the properness of $\o {\orb M}_{0;b,1}(\epsilon)$. For the projectivity, we cite \citep[\autoref*{p1:thm:projectivity}]{deopurkar12:_compac_hurwit}, which proves that the Hodge class $\lambda$ is relatively anti-ample along $\br \from \o T_{g;1}(\epsilon) \to \o M_{0;b,1}(\epsilon)$.
\end{proof}

\section{The moduli  of $l$-balanced covers }\label{sec:lbalanced}
Retain the notation introduced just before \autoref{thm:irred_and_dim}. Also recall the Maroni invariant, the $\mu$ invariant and the notion of $l$-balanced covers from \autoref{def:l-balanced}. It is easy to check that the Maroni invariant is upper semi-continuous and the $\mu$ invariant is lower semi-continuous in appropriate families. Let $\o{\orb T}_{g;1}^l \subset \st T_{g;1}$ be the open substack consisting of $(\phi \from \orb C \to \orb P; \orb P \to P; \sigma)$ where $P$ is smooth and $\phi$ an $l$-balanced triple cover. Since $P$ is smooth and the orbi-structure of $\orb P$ along $\sigma$ is trivial, the orbi-structure of $\orb P$ is in fact trivial; that is, $\orb P \to P$ is an isomorphism. Hence, $\orb C$ is also simply a schematic curve. Therefore, we write just $(\phi \from C \to P; \sigma)$ for the objects of $\o {\orb T}_{g;1}^l$.

\begin{restatable}{theorem}{thmtgd}
  \label{thm:tg3d}
  $\o{\orb T}_{g;1}^l$ is a Deligne--Mumford stack, smooth and proper over $k$, and irreducible of dimension $2g+2$.
\end{restatable}

For $l \geq g$, we have an equality as substacks of $\st T_{g;1}$:
  \[ \o{\orb T}_{g;1}(1/(b-1)) = \o{\orb T}_{g;1}^l.\]
  Indeed, since the Maroni and the $\mu$ invariants lie between $0$ and $g$, both sides are essentially the open substack of $\st T_{g;1}$ parametrizing $(\phi \from C \to P; \sigma)$ where $P$ is smooth and $\br\phi$ is not concentrated at one point.

  By the Keel--Mori theorem \citep{keel97:_quotien}, we conclude the following.
\begin{corollary}
  $\o{\orb T}_{g;1}^l$ admits a coarse space $\o{\sp T}_{g;1}^l$ proper over $k$. It is an irreducible algebraic space of dimension $2g+2$ with at worst quotient singularities. In particular, it is normal and $\Q$-factorial.
\end{corollary}

We now turn to the proof of \autoref{thm:tg3d}. The main ingredient is the behavior of vector bundles under pull back and push forward along blow ups of smooth surfaces.
\begin{lemma}\label{thm:blow_down_vb}
  Let $X$ be a smooth surface, $s \in X$ a point, $\beta \from \Bl_sX \to X$ the blowup, and $E \subset \Bl_sX$ the exceptional divisor. Let $V$ be a locally free sheaf on $\Bl_sX$. Denote by $e$ the natural map
  \[ e \from \beta^*\beta_* V \to V.\]
  Then,
  \begin{compactenum}
  \item $\beta_*V$ is torsion free; that is, $\Hom(O_s, \beta_*V) = 0$.
  \item If $H^0(V|_E \otimes O_E(-1)) = 0$, then $\beta_*V$ is locally free and $e$ is injective.
  \item If $V|_E$ is globally generated, then $R^i\beta_*V = 0$ for all $i > 0$, and $e$ is surjective.
  \item More generally, if $l \geq 0$ is such that $V|_E \otimes O_E(l)$ is globally generated, then $\cok(e)$ is annihilated by $I_E^l$, where $I_E \subset O_{\Bl_s X}$ is the ideal sheaf of $E$.
  \end{compactenum}
\end{lemma}
\begin{proof}Without loss of generality, take $X$ to be affine.

  \begin{compactenum}
  \item We have $\Hom(O_s, \beta_*V) = \Hom(\beta^*O_s, V) = 0$, since $V$ is locally free.
  \item Set $X^\circ = X \setminus \{s\} = \Bl_sX\setminus E$ and let $i \from X^\circ \into X$ be the open inclusion. We have a natural map
    \begin{equation}\label{eqn:map_in_a_vb}
      \beta_*V \to i_*i^* \beta_*V,
    \end{equation}
    which is injective by (1). The target $i_*i^* \beta_*V$ is locally free---it is the unique locally free extension to $X$ of $i^*\beta_* V = V|_{X^\circ}$ on $X^\circ$. We prove that the map \eqref{eqn:map_in_a_vb} is surjective. Equivalently, we want to prove that every element of $H^0(X^\circ, V)$ extends to an element of $H^0(\Bl_sX, V)$. Take $f \in H^0(X^\circ, V)$. Let $n \geq 0$ be the smallest integer such that $f$ extends to a section $\tw f$ in $H^0(\Bl_sX, V(nE))$. The minimality of $n$ means that the restriction of $\tw f$ to $E$ is not identically zero. If $n = 0$, we are done. If $n \geq 1$, then the hypothesis $H^0(V|_E \otimes O_E(-1)) = 0$ implies that $\tw f$ restricted to $E$ is identically zero, contradicting the minimality of $n$. Finally, since $\beta_*V$ is locally free, so is $\beta^*\beta_*V$. The map $\beta^*\beta_*V \to V$ is injective away from $E$, and hence injective.
    
  \item Let $V|_E$ be globally generated. By the theorem on formal functions, we have
    \[ \compl{(R^i\beta_*V)}_s = \varprojlim_{m} H^i(V|_{mE}).\]
    To get a handle on $V|_{mE}$, we use the exact sequence
    \[ 0 \to V|_E \otimes I_E^{m-1} \to V|_{mE} \to V|_{(m-1)E} \to 0.\]
    Since $V|_E$ is globally generated, so is $V|_E\otimes I_E^{m-1} = V|_E \otimes O_E(m-1)$, for all $m \geq 1$. It follows by induction on $m$ that $H^i(V|_{mE}) = 0$ for all $m \geq 1$ and $i > 0$. Thus $R^i\beta_*V = 0$.

    We now prove that $e$ is surjective. It is an isomorphism away from $E$. We need to prove that it is surjective along $E$. Since $V|_E$ is globally generated, it suffices to prove that $\beta_*V \to \beta_*\left(V|_E\right)$ is surjective. From the exact sequence
    \[ 0 \to I_E \otimes V \to V \to V|_E \to 0,\]
    we get the exact sequence 
    \[ \beta_*V \to \beta_* \left(V|_E\right) \to R^1\beta_*(I_E \otimes V).\]
    Since $(I_E \otimes V)|_E = V|_E \otimes O_E(1)$ is globally generated, $R^1\beta_*(I_E \otimes V)$ vanishes and we conclude that $\beta_*V \to \beta_*(V|_E)$ is surjective.

  \item Consider the diagram
    \[
    \begin{CD}
      \beta^*\beta_* (I_E^l \otimes V) @>>> I_E^l \otimes V @>>> 0 \\
      @VVV @VVV @VVV \\
      \beta^*\beta_* V @>>> V @>>> Q @>>> 0\\
    \end{CD}.
    \]
    The first row is exact by (3), as $(I_E^l \otimes V)|_E = V|_E \otimes O_E(l)$ is globally generated. It follows that the multiplication map $I_E^l \otimes V \to Q$ is zero. Since $V \to Q$ is surjective, the multiplication $I_E^l \otimes Q \to Q$ is zero as well.
  \end{compactenum}
\end{proof}
\autoref{thm:blow_down_vb} allows us to analyze the ``blowing down'' of trigonal curves. This analysis is the content of the following lemma. Roughly, it says that blowing down a trigonal curve of Maroni invariant $M$ results in a singularity of $\mu$ invariant at most $M$.

\begin{lemma}\label{thm:blow_down_trigonal_curves}
  Let $X$, $s$, $\beta \from \Bl_s X \to X$, $E$ and $X^\circ$ be as in \autoref{thm:blow_down_vb}. Let $F \subset X$ be a smooth curve passing through $s$ and $\tw F \subset \Bl_sX$ its proper transform. Let $\tw f \from \tw C \to \Bl_sX$ be a triple cover, \'etale over $\tw F$. Assume that $\tw C|_E$ is a reduced curve of genus $g$ and $\tw{\phi}\from \tw C|_E \to E$ has Maroni invariant $M$. Denote by $\phi \from C \to X$ the unique extension to $X$ of $\tw{\phi}\from \tw C|_{X^\circ} \to X^\circ$ (\autoref{lem:horrocks}). Then $\phi\from C|_F \to F$ is \'etale except over $s$, and has a singularity of $\mu$-invariant at most $M$ over $s$:
  \[ \mu(\phi|_F) \leq M(\tw\phi|_E).\]
\end{lemma}
The setup is partially described in \autoref{fig:blow_down_trigonal_curves}.
\begin{figure}[h]
  \centering
  \begin{tikzpicture}[math,map,thick,scale=.8]
    \draw [dotted] (0,0) circle (4em);
    \draw (0,8em) node {\tw C};
    \path (0,7em) edge[->,thin] (0, 5em);
    \draw 
    (-.3,-.5) .. controls (.2,.5) .. (.3,1.3)
    node [below right] {\tw F}
    (-.3,.5) .. controls (.2,-.5) .. (.3,-1.3)
    node [above right] {E};
    \node (Xt) at (0,0) {} node [right] {\tw s};
    \node at (0,-2) [below] {\Bl_sX};

    \begin{scope}[xshift=15em]
      \draw [dotted] (0,0) circle (4em);
      \draw (0,8em) node {C};
      \path (0,7em) edge[->,thin] (0, 5em);
      \draw 
      (0,-1.3) -- (0,1.3)
      node [below right] {F};
      \draw[fill] (0,0) circle (0.1em) node (X) [right] {s};
      \node at (0,-2) [below] {X};
    \end{scope}

    \path[shorten >=4em, shorten <=4em,thin, ->] 
    (Xt) edge node [above] {\beta} (X);
  \end{tikzpicture}
  \caption{The setup of \autoref{thm:blow_down_trigonal_curves}.}
  \label{fig:blow_down_trigonal_curves}
\end{figure}
In this setup, we say that $C \to X$ is obtained by \emph{blowing down} $\tw C \to \Bl_sX$ along $E$.

\begin{proof}
  To simplify notation, we drop the $\tw \phi_*$ (resp. $\phi_*$) and simply write $O_{\tw C}$ (resp. $O_C$), considered as a sheaf of algebras on $\Bl_sX$ (resp. $X$).

We apply \autoref{thm:blow_down_vb} to the vector bundle $O_{\tw C}$ on $\Bl_s X$. The condition (2) in \autoref{thm:blow_down_vb} is clearly satisfied; therefore $\beta_*O_{\tw C}$ is a locally free sheaf of rank 3. Note that $\beta_*O_{\tw C}$ is naturally an $O_X$ algebra which agrees with $O_{C}$ on $X^\circ$. It follows that 
\[O_C = \beta_*O_{\tw C}.\]

Since $\tw C|_{\tw F} \to \tw F$ is \'etale,  $C|_F \to F$ is \'etale except possibly over $s$. Next, set $\tw s = \tw F \cap E$. Consider the map of $O_{\Bl_sX}$ algebras
  \[ \nu \from \beta^*\beta_*O_{\tw C} = \beta^*O_C \to O_{\tw C}.\]
  This is an isomorphism away from $E$, and hence, when restricted to $\tw F$, we have a sequence
  \begin{equation}\label{eqn:disguised_normalization_sequence}
    0 \to (\beta^*O_C)|_{\tw F} \stackrel{\nu|_F} \longrightarrow O_{\tw C}|_{\tw F} \to Q \to 0,
  \end{equation}
  where $Q$ is supported at $\tw s$.
  
  The sequence \eqref{eqn:disguised_normalization_sequence} exhibits $O_{\tw C}|_{\tw F}$ as the normalization of $\beta^*O_C|_{\tw F}$. Moreover, since $\beta \from \tw F \to F$ is an isomorphism, the algebra $\beta^*O_C|_{\tw F}$ on $\tw F$ can be identified with the algebra $O_C|_F$ on $F$ via $\beta$. Hence, the splitting type of the singularity of $C \to F$ over $s$ is simply the splitting type of the module $Q$. 

  Let $x$ be a uniformizer of $\tw F$ near $\tw s$. Suppose
  \[Q \cong k[x]/x^m \oplus k[x]/x^n,\]
  and
  \[O_{\tw C|_E}/O_E \cong O_E(-m') \oplus O_E(-n'),\]
  for some positive integers $m$, $n$, $m'$ and $n'$ with $m+n=m'+n'=g+2$.  By \autoref{thm:blow_down_vb} (4), the ideal $I_E^{\max\{m',n'\}}$ annihilates the cokernel of $\nu$. Restricting to $\tw F$, we see that $x^{\max\{m',n'\}}$ annihilates $Q$. In other words, 
  \[ \max\{m,n\} \leq \max\{m',n'\}.\]
    Since $m+n = m'+n'$, it follows that
    \[ \mu(\phi|_F) = |m-n| \leq |m'-n'| = M(\tw\phi|_E).\]
  \end{proof}

  Next, we prove a precise result about the behavior of rank two bundles under elementary transformations, especially about their splitting type. Let $R$ be a DVR with uniformizer $t$, residue field $k$, and fraction field $K$. Set $\Delta = \spec R$. Consider $P = \proj R[X,Y] = \P^1_\Delta$ with the two disjoint sections $s_0 \equiv [0:1]$ and $s_\infty \equiv [1:0]$. Denote by $F$ the central fiber of $P \to \Delta$, and by $0$ (resp. $\infty$) the point $F \cap s_0$ (resp. $F \cap s_\infty$). Consider the map 
\[ \beta \from P \setminus \{\infty\} \to P, \quad [X:Y] \mapsto [tX: Y].\]
Then $\beta$ has a resolution (see \autoref{fig:resolution_of_beta})
\[
\begin{tikzpicture}[math,node distance=5em]
  \draw 
  node (P) {\tw P}
  node (P1) [below left of=P] {P}
  node (P2) [below right of=P] {P};
  
  \path[->,map]
  (P) edge node [left] {\beta_1} (P1)
  (P) edge node [right] {\beta_2} (P2)
  (P1) edge [dashed] node [above] {\beta} (P2);
\end{tikzpicture}.
\]
Here $\beta_1 \from \tw P \to P$ is the blow up at $\infty$ and $\beta_2 \from \tw P \to P$ is the blow up at $0$. The central fiber of $\tw P \to \Delta$ is $F_1 \cup F_2$, where $F_i$ is the exceptional divisor of $\beta_i$. The $F_i$ meet transversely at a point, say $s$. \begin{figure}[h]
  \centering
  \begin{tikzpicture}[math,map,scale=.8]
    \foreach \i in {-5,5}{
      \begin{scope}[xshift=\i cm]
        \draw[dotted, thick] (-1,-2) rectangle (1,2);
        \draw
        (-1,1)  -- (1,1) node[right] {s_\infty}
        (-1,-1) -- (1,-1) node [right] {s_0}
        node at (0,1) [below left] {\infty}
        node at (0,-1) [below left] {0}
        node [font=] at (0,-2.5) {P}
        node (P\i) at (0,0) [right] {F};
        \draw [thick] (0,-2) --  (0,2);
      \end{scope}
    }
    
    \begin{scope}[yshift=2cm, thick]
      \draw[dotted]  (-1,-2) rectangle (1,2);
      \draw
      (-.2,.5) .. controls (0.2,-.8) .. (.2,-2) 
      (-.2,-.5) .. controls (0.2, .8) .. (.2,2);
      \draw
      node at (0,2) [below left] {F_1}
      node at (0,-2) [above left] {F_2}
      node (Pt) at (0,0) {}
      node (s) at (0,0) [right] {s}
      node [font=] at (0,-2.5) {\tw P};
    \end{scope}
    
    \path[->, shorten <=1cm, shorten >=1cm]
    (Pt) edge node [above] {\beta_1} 
    node [below, sloped] {\infty \mapsfrom F_1} (P-5) 
    (Pt) edge node [above] {\beta_2} 
    node [below, sloped] {F_2 \mapsto 0} (P5);
    \path [dashed, shorten <=2cm, shorten >=2cm, bend right=30,->]
    (P-5) edge node {\beta} (P5) ;
  \end{tikzpicture}
  \caption{Resolution of $\beta \from [X:Y] \mapsto [tX:Y]$.}
  \label{fig:resolution_of_beta}
\end{figure}
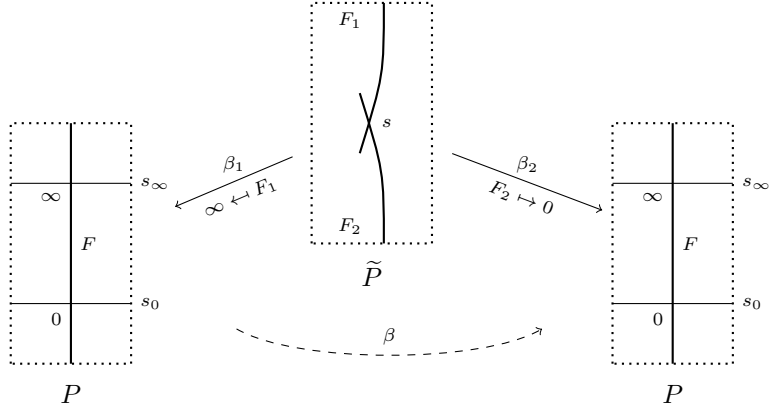

\begin{lemma}\label{eqn:pull_push_vb}
  Let $n > m$ be non-negative integers, and $O(1)$ the dual of the ideal sheaf of $s_0$. Identify
  \[ \Ext^1(O(-m), O(-n)) = R \langle X^{n-m-2}, \dots, X^iY^{n-m-2-i}, \dots , Y^{n-m-2} \rangle.\]
  Let $V$ be a vector bundle of rank two on $P$ given as an extension 
  \begin{equation}\label{eqn:ext_V}
    0 \to O(-n) \to V \to O(-m) \to 0,
  \end{equation}
  corresponding to the class $e(X,Y) \in \Ext^1(O(-m), O(-n))$. Denote by $W$ the unique vector bundle on $P$ obtained by extending $\beta^*(V)$. Assume that the class $t^{m-n+1}e(tX,Y)$, lying \emph{a priori} in $K \otimes_R \Ext^1(O(-m), O(-n))$, lies in $\Ext^1(O(-m), O(-n))$. Then $W$ can be expressed as an extension 
  \[ 0 \to O(-n) \to W \to O(-m) \to 0,\]
  with class $t^{m-n+1}e(tX,Y)$. Moreover, in this case, we have an exact sequence
  \begin{equation}\label{eqn:pull_push_split}
    0 \to (\beta_1^* W)|_{F_2} \to (\beta_2^*V)|_{F_2} \to k[u]/u^m \oplus k[u]/u^n \to 0, 
  \end{equation}
  where $u$ is a uniformizer of $F_2$ at $s$.
\end{lemma}
We say that a class in $K \otimes_R \Ext^1(O(-m),O(-n))$ is \emph{integral} if it belongs to $\Ext^1(O(-m),O(-n))$. The class $t^{m-n}e(tX,Y)$ is integral if $e(X,Y)$ is sufficiently divisible by $t$. 

\begin{proof}
  The proof is by a possibly tedious but straightforward local computation. Write $P = \spec R[x] \cup \spec R[y]$, where $x = X/Y$ and $y = Y/X$. To ease notation, write $e(x)$ for $e(x,1)$. Note that 
  \[\beta^{-1} \spec R[x] = \spec R[x]; \quad \beta^{-1} \spec R[y] = \spec K[y].\]
  The union $\beta^{-1}\spec R[x] \cup \beta^{-1}\spec K[y]$ is $P \setminus \{\infty\}$, as expected.

  We can choose local trivializations $\langle v^1_x, v^2_x \rangle$ and $\langle v^1_y, v^2_y \rangle$ for $V$ on $\spec R[x]$ and $\spec R[y]$ respectively, such that they are related on the intersection by
  \begin{equation}\label{eqn:transition_V}
    \begin{pmatrix}
      v^1_y \\ v^2_y
    \end{pmatrix}
    =
    \begin{pmatrix}
      x^{-n} & 0 \\
      x^{-n+1}e(x) & x^{-m}
    \end{pmatrix}
    \begin{pmatrix}
      v^1_x \\ v^2_x
    \end{pmatrix}.
  \end{equation}
  The bundle $\beta^*V$ is trivialized by $\langle \beta^*v^1_x, \beta^*v^2_x\rangle$ on $\beta^{-1}\spec R[x]$ and $\langle \beta^*v^1_y, \beta^*v^2_y \rangle$ on $\beta^{-1}\spec R[y]$. The transition matrix on the intersection is simply the pullback of the matrix in \eqref{eqn:transition_V}:
  \begin{equation}\label{eqn:transitiov_bV}
    \begin{pmatrix}
      \beta^*v^1_y \\ \beta^*v^2_y
    \end{pmatrix}
    =
    \begin{pmatrix}
      t^{-n}x^{-n} & 0 \\
      t^{-n+1}x^{-n+1}e(tx) & t^{-m}x^{-m}
    \end{pmatrix}
    \begin{pmatrix}
      \beta^*v^1_x \\ \beta^*v^2_x
    \end{pmatrix}.
  \end{equation}
  Construct $W$ by gluing trivializations $\langle w^1_x, w^2_x \rangle$ on $\spec R[x]$ and $\langle w^1_y,w^2_y \rangle$ on $\spec R[y]$ by
  \begin{equation}\label{eqn:transition_W}
    \begin{pmatrix}
      w^1_y \\ w^2_y
    \end{pmatrix}
    =
    \begin{pmatrix}
      x^{-n} & 0 \\
      t^{m-n+1}x^{-n+1}e(tx) & x^{-m}
    \end{pmatrix}
    \begin{pmatrix}
      w^1_x \\ w^2_x
    \end{pmatrix}.
  \end{equation}
  Construct an explicit isomorphism $\psi \from \beta^*V \stackrel{\sim}{\to} W$ on $P\setminus\{\infty\}$, as follows:
  \begin{equation}\label{eqn:phi_V_W}
    \begin{split}
      \psi \from
      \begin{pmatrix}
        \beta^*v^1_x \\
        \beta^*v^2_x
      \end{pmatrix}
      &\mapsto
      \begin{pmatrix}
        w^1_x \\
        w^2_x
      \end{pmatrix} \quad \text{on $\beta^{-1}\spec R[x] = \spec R[x]$,}\\
      \psi \from
      \begin{pmatrix}
        \beta^*v^1_y \\
        \beta^*v^2_y
      \end{pmatrix}
      &\mapsto
      \begin{pmatrix}
        t^{-n}w^1_y \\
        t^{-m}w^2_y
      \end{pmatrix} \quad \text{on $\beta^{-1}\spec R[y] = \spec K[y]$.}    
    \end{split}
  \end{equation}
  From the transition matrices \eqref{eqn:transitiov_bV} and \eqref{eqn:transition_W}, it is easy to check that this defines a map $\psi \from \beta^*V \to W$ on $P \setminus \{\infty\}$, which is clearly an isomorphism. From \eqref{eqn:transition_W}, we see that $W$ is an extension of $O(-m)$ by $O(-n)$ corresponding to the class $t^{m-n+1}e(tX,Y)$.

  Finally, we establish the exact sequence \eqref{eqn:pull_push_split}. By \autoref{thm:blow_down_vb}, ${\beta_1}_*\beta_2^* V$ is a vector bundle which is identical to $\beta^*V$ on $P \setminus \{\infty\}$. Therefore, we must have 
  \[ W \cong {\beta_1}_*\beta_2^* V.\]
   The map $\beta_1^*W \to \beta_2^*V$ in \eqref{eqn:pull_push_split} is simply the natural map
  \begin{equation}\label{eqn:WV}
    \beta_1^* W = \beta_1^*{\beta_1}_*(\beta_2^* V) \stackrel{\rm ev}\longrightarrow \beta_2^*V.
  \end{equation}
  To obtain the cokernel, we express ${\rm ev}$ in local coordinates around $s$. Set $u = \beta_1^*y$; this is a function on a neighborhood of $s$ in $\tw P$. A basis for $\beta_2^*V$ around $s$ is given by $\langle \beta_2^*v^1_x, \beta_2^*v^2_x\rangle$. A basis for $\beta_1^*W$ around $s$ is given by $\langle \beta_1^*w^1_y, \beta_1^*w^2_y\rangle$. From the description of $\psi$ in \eqref{eqn:phi_V_W}, it follows that ${\rm ev}$ is given by
  \[ 
  {\rm ev} \from 
  \begin{pmatrix}
    \beta_1^*w^1_y \\
    \beta_1^*w^2_y
  \end{pmatrix}
  \mapsto
  \begin{pmatrix}
    u^n & 0 \\
    t^{m-n+1}u^{n-1}e(t/u) & u^m
  \end{pmatrix}
  \begin{pmatrix}
    \beta_2^*v^1_x\\
    \beta_2^*v^2_x
  \end{pmatrix}.
  \]
  Note that $t^{m-n+1}u^{n-1}e(t/u)$ lies in $R\langle u^{m+1}, \dots, u^{n-1}\rangle$. Hence, we get
  \[\cok({\rm ev}|_{F_2}) \cong k[u]/u^m \oplus k[u]/u^n.\]
  Since $u|_{F_2}$ is a uniformizer for $F_2$ around $s$, the sequence \eqref{eqn:pull_push_split} is established.
\end{proof}

Lastly, we need a simple inequality between the Maroni and the $\mu$ invariants.
\begin{lemma}\label{thm:Maroni_less_than_mu}
  Let $C$ be a curve of genus $g$ and $\phi \from C \to \P^1$ a triple cover with concentrated branching. Then
  \[ M(\phi) \leq \mu(\phi).\]
\end{lemma}
\begin{proof}
  Let $\br(\phi) = b \cdot p$ for some $p \in \P^1$, where $b = 2g+4$. Let the splitting type of the singularity over $p$ be $(m,n)$, where $n \geq m$ and $n+m=g+2$. Set $F = \phi_* O_C/O_{\P^1}$. Then $F \cong O_{\P^1}(-m') \oplus O_{\P^1}(-n')$ for some $n' \geq m'$ with $m' + n' = g+2$. 

  Let $\tw C \to C$ be the normalization. We have the sequence
  \begin{equation}\label{eqn:FQ}
    0 \to F \to \tw\phi_* O_{\tw C}/O_{\P^1} \to k[x]/x^m \oplus k[x]/x^n \to 0.
  \end{equation}
  Since $\tw C \cong \P^1 \sqcup \P^1 \sqcup \P^1$, the middle term above is simply $O_{\P^1}^{\oplus 2}$. It is easy to see that we have a surjection 
\[ H^0(O_{\P^1}^{\oplus 2}(n-1)) \onto H^0(k[x]/x^m \oplus k[x]/x^n).\]
Using the sequence on cohomology induced by \eqref{eqn:FQ} twisted by $O_{\P^1}(n-1)$, we conclude that $H^1(F(n-1)) = 0$, or equivalently, that $n' \leq n$. It follows that 
 \[ M(\phi) = n' - m' = 2n' - (g+2) \leq 2n - (g+2) = n -m = \mu(\phi).\]
\end{proof}

We now have the tools to prove \autoref{thm:tg3d}.
\begin{proof}[Proof of \autoref{thm:tg3d}]
  Without loss of generality, assume that $l$ satisfies $0 \leq l \leq g$ and $l \equiv g \pmod 2$. We divide the proof into steps.
  
  \begin{asparadesc}
  \item[That $\o {\orb T}_{g;1}^l$ is smooth, of finite type, and irreducible of dimension $2g+2$:] 
    $\o {\orb T}_{g;1}^l$ is an open substack of $\st T_{g;1}$. Since $\st T_{g;1}$ is smooth and irreducible of dimension $2g+2$, so is $\o{\orb T}_{g;1}^l$.
    
    To see that it is of finite type, denote by $\st M_{0;b,1}^s \subset \st M_{0;b,1}$ the open substack parametrizing $(P;\Sigma;\sigma)$ with $P$ smooth. It is easy to see that $\st M_{0;b,1}^s$ is of finite type over $k$. By the definition of $\o{\orb T}_{g;1}^l$, the open immersion $\o{\orb T}_{g;1}^l \into \st T_{g;1}$ factors as 
      \[\o{\orb T}_{g;1}^l \into \st M_{0;b,1}^s \times_{\st M_{0;b,1}} \st T_{g;1} \subset \st T_{g;1}.\]
      Since $\st T_{g;1} \to \st M_{0;b,1}$ is of finite type, we conclude that
      $\st M_{0;b,1}^s \times_{\st M_{0;b,1}} \st T_{g;1}$ and hence $\o{\orb T}_{g;1}^l$ is of finite type over $k$.
  
    \item[\textbf{That $\o{\orb T}_{g;1}^l$ is separated}:]
      We use the valuative criterion. Let $\Delta = \spec R$ be a the spectrum of a DVR, with special point $0$, generic point $\eta$ and residue field $k$. Consider two morphisms $\Delta \to \o{\orb T}_{g;1}^l$ corresponding to $(P_i \to \Delta; \sigma_i; \phi_i \from C_i \to P_i)$ for $i=1,2$. Let $\psi_\eta$ be an isomorphism of this data over $\eta$, namely isomorphisms $\psi^P_\eta \from {P_1}|_\eta \to {P_2}|_\eta$ and $\psi^C_\eta \from {C_1}|_\eta \to {C_2}|_\eta$ over $\eta$ that commute with $\phi_i$ and  $\sigma_i$. We must show that $\psi_\eta$ extends to an isomorphism over all of $\Delta$. Suppose that $\psi^P_\eta$ extends to a morphism $\psi^P \from P_1 \to P_2$ over $\Delta$. Then $\psi^P$ must be an isomorphism, because the $P_i \to \Delta$ are $\P^1$ bundles and $\psi^P_\eta$ is an isomorphism. It also follows that $\psi^P$ must be an isomorphism of marked curves
     \[\psi^P \from (P_1; \br\phi_1; \sigma_1) \isom (P_1; \br\phi_2; \sigma_2).\]
     By the separatedness of $\st T_{g;1} \to \st M_{0;b,1}$, we conclude that we have an extension $\psi^C \from C_1 \to C_2$ over $\psi^P$. Therefore, it suffices to show that $\psi_\eta^P$ extends. Denote by $\psi^P$ the maximal extension of $\psi^P_\eta$. Since $P_i \to \Delta$ are $\P^1$ bundles, the rational map $\psi^P$ has a resolution of the form
      \begin{equation}\label{eqn:psi_resolution}
        \begin{tikzpicture}[math, node distance=1.5cm]
          \node (P1t) {\tw {P}};
          \node (P1) [below left of=P1t] {P_1};
          \node (P2) [below right of=P1t] {P_2};
          \path[map,->]
          (P1t) edge (P1) 
          (P1t) edge (P2);         
          \path[map,dashed,->]
          (P1) edge node [above] {\psi^P} (P2);
        \end{tikzpicture},
      \end{equation}
      where $\tw P$ is smooth and its (scheme theoretic) central fiber is a chain of smooth rational curves $E_0 \cup \dots \cup E_n$; the map $\tw P \to P_1$ blows down $E_n, \dots, E_1$ successively to a point $p_1 \in P_1|_0$; and the map $\tw P \to P_2$ blows down $E_0, \dots, E_{n-1}$ successively to a point $p_2 \in P_2|_0$ (see \autoref{fig:psi_resolution}). If $n = 0$, then $\psi^P$ is already a morphism, and we are done. Otherwise, we look for a contradiction.
\begin{figure}[ht]
  \centering
  \begin{tikzpicture}[math,map,thick,scale=1]
    \draw [dotted] (-1.2,-2.3) rectangle (1.2,2.3);
    \draw 
    (-.2,2.0) node [below left] {E_n}
    -- (.2,1.0)
    (.2,1.4) -- (-.2,.4)
    (-.2,.8) -- (.2,-.2)
    (.2,.2) --  (-.2,-.8)
    (-.2,-.4) -- (.2,-1.4)
    (.2,-1.0) -- (-.2,-2.0)
    node [right] {E_0}
    (0,0) node (Pt){}
    (0,-2.5) node [below] {\tw P}
    ;
    
    \begin{scope}[xshift=15em]
      \draw [dotted] (-1,-1.5) rectangle (1,1.5);
      \draw 
      (0,-1.3) -- (0,1.3)
      (0,0) node (P2) [right] {}
      (0,-2) node [below] {P_2};
      \fill (0,.7) circle (.05) node [right] {p_2};
    \end{scope}

    \begin{scope}[xshift=-15em]
      \draw [dotted] (-1,-1.5) rectangle (1,1.5);
      \draw 
      (0,-1.3) -- (0,1.3)
      (0,0) node (P1) [right] {}
      (0,-2) node [below] {P_1};
      \fill (0,-.7) circle (.05) node [right] {p_1};
    \end{scope}

    \path[shorten >=3em, shorten <=4em,thin, ->] 
    (Pt) edge  
    node [above] {E_n \isom P_2|_0} 
    node [below] {E_0, \dots, E_{n-1} \mapsto p_2} 
    (P2)
    
    (Pt) edge  
    node [below] {E_0 \isom P_1|_0} 
    node [above] {E_n, \dots, E_1 \mapsto p_1} 
    (P1);
  \end{tikzpicture}
  \caption{The resolution of $\psi^P \from P_1 \to P_2$}
  \label{fig:psi_resolution}
\end{figure}
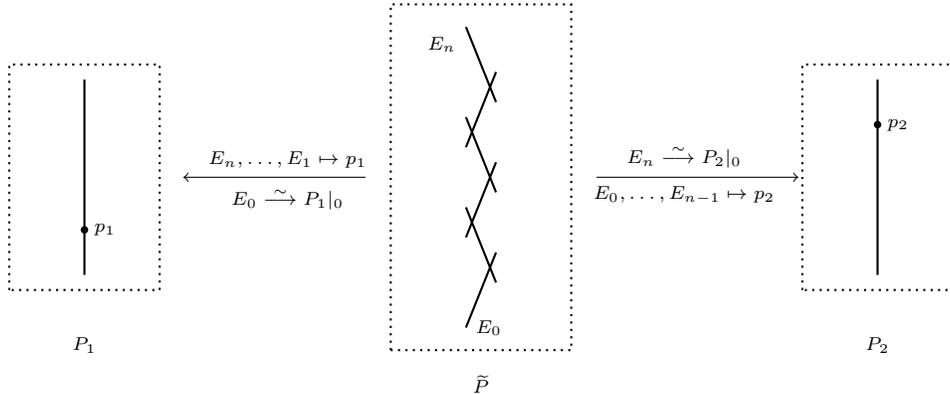

Since $\psi^P_\eta$ takes $\sigma_1(\eta)$ to $\sigma_2(\eta)$, either $p_1 = \sigma_1(0)$ or $p_2 = \sigma_2(0)$. By switching 1 and 2 if necessary, say $p_1 = \sigma_1(0)$. Let $\tw C \to \tw P$ be the pullback of $C_1 \to P_1$. Since $C_1 \to P_1$ is \'etale over $\sigma_1(0)$, the cover $\tw C \to \tw P$ is \'etale over $E_1, \dots, E_n$. Then $C_2 \to P_2$ is obtained by blowing down $\tw C \to \tw P$ successively along $E_0, \dots, E_{n-1}$. Thus, $C_2|_0 \to P_2|_0$ is has concentrated branching at $p_2$; let $\mu$ be its $\mu$ invariant. On the other hand, $\tw C|_{E_0} \to E_0$ is isomorphic to $C_1|_0 \to P_1|_0$; let $M$ be its Maroni invariant. Since both $(C_i \to P_i)_0$  are $l$-balanced, we have
 \[ \mu >  l \geq M.\]
 However, by repeated application of \autoref{thm:blow_down_trigonal_curves} and \autoref{thm:Maroni_less_than_mu}, we get
 \[ \mu \leq M.\]
 We have reached a contradiction.

\item[\textbf{That $\o{\orb T}_{g;1}^l$ is Deligne--Mumford:}]
  Since we are in characteristic zero, it suffices to prove that a $k$-point $(\phi \from C \stackrel{\phi}\to \P^1; \sigma)$ of $\o{\orb T}_{g;1}^l$ has finitely many automorphisms. We have a morphism of algebraic groups
\[ \tau \from \Aut(\phi \from C \to \P^1, \sigma) \to \Aut(\P^1).\]
The kernel of $\tau$ consists of automorphisms of $\phi$ over the identity of $\P^1$. Such an automorphism is determined by its action on a generic fiber of $\phi$. Hence $\ker\tau$ is finite. Since $\o{\orb T}_{g;1}^l$ is separated, $\Aut(\phi \from C \to \P^1, \sigma)$ is proper. On the other hand, $\Aut(\P^1)$ is affine. It follows that $\im \tau$ is finite. We conclude that $\Aut(\phi \from C \to \P^1, \sigma)$ is finite.

\item[\textbf{That $\o{\orb T}_{g;1}^l$ is proper}:]
  Let $\Delta = \spec R$ be as in the proof of separatedness.  Denote by $\o\eta$ a geometric generic point. Let $(\phi \from C_\eta \stackrel{\phi}\to P_\eta; \sigma)$ be an object of $\o{\orb T}_{g;1}^l$ over $\eta$. We need to show that, possibly after a finite base change, it extends to an object of $\o{\orb T}_{g;1}^l$ over $\Delta$. Without loss of generality, we may assume that the object over $\eta$ lies in a dense open substack of $\o{\orb T}_{g;1}^l$. Therefore, we may take $\phi \from C_{\o\eta} \to P_{\o \eta}$ to not have concentrated branching. Extend $(P; \br\phi; \sigma)$ to an object $(P; \Sigma; \sigma)$ of $\st M_{0;b,1}(\Delta)$. Since $\st T_{g;1}^l \to \st M_{0;b,1}$ is proper, we get an extension $(C \to P; \sigma)$ of $(C \to P; \sigma)_\eta$ over $(P; \Sigma; \sigma)$, possibly after a finite base change. Assume that $C|_0 \to P|_0$ satisfies the second condition of \autoref{def:l-balanced}. This can be achieved, for instance, by having $\Sigma|_0$ not supported at a point.

If the Maroni invariant of $C|_0 \to P|_0$ is at most $l$, we are done. Otherwise, we must modify $C \to P$ along the central fiber to make it more balanced. Fix an isomorphism of $P \to \Delta$ with $\proj R[X,Y]\to \Delta$ such that the section $\sigma \from \Delta \to P$ is the zero section $[0:1]$. Set $V = \phi_*O_C/O_{P}$; it is a vector bundle of rank 2 on $P$. Let
 \[
   V|_{P_0} \cong O_{\P^1}(-m) \oplus O_{\P^1}(-n),
   \]
 where $m < n$ are positive integers with $m+n=g+2$ and $n-m > l$. Then we can express $V$ as an extension
 \begin{equation}\label{eqn:V_ext}
   0 \to O_{P} (-n) \to V \to O_{P}(-m) \to 0.
 \end{equation}
 Denote the extension class by
\[ e(X,Y) \in \Ext^1(O_{P}(-m), O_{P}(-n)) = R \langle X^{n-m-2},\dots,Y^{n-m-2}\rangle.\]
Since $C_{\o\eta} \to P_{\o\eta}$ is $l$-balanced but $n-m > l$, the class $e(X,Y)$ is nonzero. However, as the restriction of \eqref{eqn:V_ext} to $P_0$ is split, $t$ divides $e(X,Y)$. By passing to a finite cover $\tw\Delta \to \Delta$, ensure that a sufficiently high power of $t$ divides $e(X,Y)$, so that $t^{m-n+1}e(tX,Y)$ is integral. Consider the rational map $\beta \from P \dashrightarrow P$, sending $[X:Y]$ to $[tX:Y]$. Then $\beta$ is defined away from $[1:0]$ on the central fiber. Let $\phi' \from C' \to P$ be the unique extension of $\beta^* C \to P$ (\autoref{lem:horrocks}). Then $C' \to P$ is isomorphic to $C \to P$ on the generic fiber, whereas the central fiber $C'|_0 \to P|_0$ is unramified except at $[1:0]$. The section $\sigma = [0:1]$ of $P \to \Delta$ serves as the required marking.

The cover $C' \to P$ may be thought of in terms of the resolution of $\beta$ (as in \autoref{fig:resolution_of_beta})
\[
\begin{tikzpicture}[math,node distance=1.5cm]
  \draw 
  node (P) {\tw P}
  node (P1) [below left of=P] {P}
  node (P2) [below right of=P] {P};
  
  \path[->,map]
  (P) edge node [left] {\beta_1} (P1)
  (P) edge node [right] {\beta_2} (P2)
  (P1) edge [dashed] node [above] {\beta} (P2);
\end{tikzpicture}.
\]
Here $\beta_1$ is the blowup at $[1:0]$ and $\beta_2$ at $[0:1]$ on the central fiber, with exceptional divisors $F_1$ and $F_2$ respectively. Set $\tw C = \beta_2^*C$. Then $C' \to P$ is the blowdown of $\tw C \to \tw P$ along $\beta_1$. Set $\tw V = {\tw \phi}_*O_{\tw C}/O_{\tw P}$ and $V' = \phi'_*O_{C'}/O_P$. Then $\tw V = \beta_2^*V$, and $V' = {\beta_1}_* \tw V$.

\begin{claim}\mbox{}
  
  \begin{compactenum}
    \item $V'$ is an extension of $O_P(-m)$ by $O_P(-n)$ given by $e'(X,Y) = t^{m-n+1}e(tX,Y)$. 
    \item The splitting type of the singularity of $C'|_0 \to P|_0$ over $[1:0]$ is $(m,n)$.
  \end{compactenum}
\end{claim}
\begin{proof}
The first claim is directly from \autoref{eqn:pull_push_vb}. For the second, consider the natural map 
  \[ \beta_1^* O_{C'} = \beta_1^* {\beta_1}_* O_{\tw C} \to O_{\tw C}.\]
  Its restriction $\beta_1^*O_{C'}|_{F_2} \to O_{\tw C}|_{F_2}$ expresses $O_{\tw C}|_{F_2}$ as the normalization of $\beta_1^*O_{C'}|_{F_2}$. On the other hand, $\beta_1$ gives an isomorphism between $\beta_1^*O_{C'}|_{F_2}$ and $O_{C'}|_{P_0}$. Hence, the splitting type of the singularity of $C'|_0$ over $[1:0]$ is the splitting type of the cokernel of $\beta_1^*O_{C'}|_{F_2} \to O_{\tw C}|_{F_2}$.  We can factor out the $O_{\tw P}|_{F_2}$ summands, by considering the diagram
    \[
    \begin{tikzpicture}
      \matrix (m) [matrix of math nodes, column sep=4em, row sep=2em]{
        & \beta_1^*O_P|_{F_2} & O_{\tw P}|_{F_2}\\
        0 & \beta_1^* O_{C'}|_{F_2} & O_{\tw C}|_{F_2} & Q & 0 \\
        0 & \beta_1^* V'|_{F_2} & \tw V|_{F_2} = \beta_2^*V|_{F_2} & Q & 0\\
    };
    \draw [equal, -] 
    (m-1-2) -- (m-1-3)
    (m-2-4) -- (m-3-4);
    \path [->]
    \foreach \j in {2,3}{
      (m-\j-1) edge (m-\j-2)
      (m-\j-2) edge (m-\j-3)
      (m-\j-3) edge (m-\j-4)
      (m-\j-4) edge (m-\j-5)
    }
    \foreach \i in {2,3}{
      (m-1-\i) edge (m-2-\i)
      (m-2-\i) edge (m-3-\i)
    };
  \end{tikzpicture}
  \]
  Thus, the second claim follows from the last exact sequence in \autoref{eqn:pull_push_vb}.
\end{proof}
Returning to the main proof, we consider the operation $e(X,Y) \mapsto t^{m-n+1}e(tX,Y)$, which explicitly looks as follows
  \[ X^iY^{n-m-2-i} \mapsto t^{m-n+1+i}X^iY^{n-m-2-i}, \text{ for $i = 0, \dots, n-m-2$.}\]
  This operation acts by purely ``negative weights'' $t^{m-n+1+i}$. It follows that after a base change $\tw\Delta \stackrel{t^N \mapsto t}\longrightarrow \Delta$ for a sufficiently divisible $N$ and a sequence of transformations $[X:Y] \mapsto [tX:Y]$ as above, we can arrange:
\begin{compactenum}[(I)]
\item The extension class $e'(X,Y) \in \Ext^1(O_P(-n), O_P(-m))$ of $V'$ is nonzero modulo $t$.
\item The splitting type of the singularity of $C'|_0 \to P|_0$ over $[1:0]$ is $(m,n)$.
\end{compactenum}

By (I), the new central fiber $\phi'_0 \from C'|_0 \to P|_0$ is more balanced than the original $C|_0 \to P|_0$. Since $l < n-m$, the new central fiber also has $\mu$ invariant greater than $l$. If $M(\phi'_0) \leq l$, then the new central fiber is $l$-balanced, and we are done. Otherwise, we repeat the entire procedure. After finitely many such iterations, we arrive at a central fiber of Maroni invariant at most $l$ and $\mu$ invariant greater than $l$. The proof of properness is thus complete.
\end{asparadesc}
\end{proof}


\section{The geometry of $\o T_{g;1}^l$: Preliminaries}\label{sec:prelim}

We collect results about triple covers and triple point singularities needed to analyze $\o T_{g;1}^l$.

\subsection{The structure of triple covers}\label{sec:triple_structure}
Let $Y$ be a scheme and $\phi \from X \to Y$ a triple cover. Assume that the fibers of $\phi$ are Gorenstein. Define $E$ by
\begin{equation}\label{eqn:std_sequence}
 0 \to O_Y \to \phi_*O_X \to \dual E \to 0.
\end{equation}
Then $E$ is a vector bundle of rank two on $Y$. By dualizing, we get a morphism $E \to \phi_* \omega_\phi$, where $\omega_\phi$ is the dualizing line bundle of $\phi$. Equivalently, we have a map $\phi^*E \to \omega_\phi.$ By explicit verification on the geometric fibers of $\phi$, it is easy to check that this map is surjective and gives an embedding $X \into \P E$ over $Y$. The ideal sheaf of $X$ in $\P E$ is canonically isomorphic to $O_{\P E}(-3) \otimes \pi^* \det E$.

Using the embedding above, one can deduce an explicit structure theorem for triple covers. This result is originally due to Miranda
\citep{miranda85:_tripl_cover_algeb_geomet}. Our exposition is based on the letters of Deligne \citep{deligne00:_letter_gan_gross_savin, deligne06:_letter_edixh} in response to the work of Gan, Gross and Savin \citep{gan02:_fourier_g}. Roughly speaking, it says that the data of a triple cover $X \to Y$ is equivalent to the data of a vector bundle $E$ of rank two on $Y$ along with a section of $\Sym^3(E) \otimes \det \dual E$. We begin by making this precise.

Let $\st B$ be the stack over $\cat{Schemes}$ given by $\st B(S) = \{(E, p)\}$, where $E$ is a vector bundle of rank two on $S$ and $p$ a global section of $\Sym^3(E) \otimes \det\dual E$. Then $\st B$ is an irreducible stack, smooth and of finite type.

Let $\st A_3$ be the stack over $\cat{Schemes}$ given by $\st A_3(S) = \{(\phi \from T \to S)\}$, where $\phi$ is a triple cover. This is an algebraic stack of finite type over $k$ (see \citep[\autoref*{p1:sec:Ad}]{deopurkar12:_compac_hurwit}). We define a morphism $\st B \to \st A_3$. Let $(\orb E, p)$ be the universal pair over $\st B$. Set $\P = \P\orb E$ and let $\pi \from \P \to \st B$ be the projection. The section $p$ gives $i \from O_{\P}(-3) \otimes \pi^* \det \orb E \to O_{\P}$. Let $W \subset \st B$ be the locus over which this map is not identically zero. It is easy to see that the complement of $W \subset \st B$ has codimension four. In particular, $i$ is injective because $\st B$ is smooth and $i$ is generically injective. Define $\orb Q$ as the quotient
\[ 0 \to O_{\P}(-3)\otimes \pi^*\det \orb E \to O_{\P} \to \orb Q \to 0.\]
Applying $\pi_*$, we get
\[ 0 \to O_{\st B} \to \pi_* \orb Q \to \dual{\orb E} \to 0.\]
Hence, $\pi_* \orb Q$ is an $O_{\st B}$ algebra which is locally free of rank three. We thus get a morphism
\begin{equation}\label{eqn:structure}
  f \from \st B \to \st A_3.
\end{equation}
\begin{theorem}\label{thm:structure}(\citep[Theorem~3.6]{miranda85:_tripl_cover_algeb_geomet}, \citep{deligne00:_letter_gan_gross_savin})
  With notation as above, the morphism $\st B \to \st A_3$ in \eqref{eqn:structure} is an isomorphism.
\end{theorem}
We follow Deligne's proof, which is based on the following observation.
\begin{proposition}\label{sec:smooth}\citep[Proposition~5.1]{poonen08}
  The stack $\st A_3$ is smooth. The complement of the Gorenstein locus $U$ has codimension four.
\end{proposition}
\begin{proof}
  We have a smooth and surjective morphism $\Hilb_3\A^2 \to \st A_3$, where $\Hilb_3\A^2$ is the Hilbert scheme of length three subschemes of $\A^2$. Since the Hilbert scheme of points on a smooth surface is smooth, we conclude that $\st A_3$ is smooth.

  The only non-Gorenstein subschemes of $\A^2_k$ of length three are $\spec k[x,y]/m^2$, where $m \subset k[x,y]$ is a maximal ideal. The locus of such has dimension two in the six dimensional space $\Hilb_3\A^2$. It follows that the complement of $U \subset \st A^3$ has codimension four.
\end{proof}
\begin{proof}[Proof of \autoref{thm:structure}]
  We construct an inverse $g \from \st A_3 \to \st B$. Let $\phi \from \st X \to \st A_3$ be the universal triple cover. Define $E$ by
  \[ 0 \to O_{\st A} \to \phi_*O_{\st X} \to \dual E \to 0.\]
  Then $E$ is a vector bundle of rank two on $\st A_3$.  Over the Gorenstein locus $U$, we have an embedding $\sh X \into \P = \P E$ giving the sequence
  \[ 0 \to O_{\P}(-3) \otimes \det E \to O_{\P} \to O_{\sh X} \to 0.\]
The map $O_\P(-3) \otimes \det E \to O_\P$ gives a section of $\Sym^3E \otimes \det\dual E$ over $U$. Since the complement of $U \subset \st A_3$ has codimension at least two and $\st A_3$ is smooth, this section extends to a section $p$ of $\Sym^3 E \otimes \det \dual E$ over all of $\st A_3$. The pair $(E,p)$ gives a morphism $g \from \st A_3 \to \st B$.
  
We must prove that $f \from \st B \to \st A_3$ and $g \from \st A_3 \to \st B$ are inverses.  Consider the composite $f \circ g \from \st A_3 \to \st A_3$. It corresponds to a triple cover of $\st A_3$. To check that $f \circ g$ is equivalent to the identity, we must check that this triple cover is isomorphic to the universal triple cover. By construction, such an isomorphism exists over the Gorenstein locus $U$. Since the complement of $U$ has codimension higher than two and $\st A_3$ is smooth, the isomorphism extends.

For the other direction, consider the composite $g \circ f \from \st B \to \st B$. It corresponds to a pair $(\orb E', p')$ on $\st B$, where $\orb E'$ is a vector bundle of rank two and $p'$ a section of $\Sym^3(\orb E') \otimes \dual{\det \orb E'}$. To check that $g \circ f$ is equivalent to the identity, we must check that this pair is isomorphic to the universal pair $(\orb E, p)$. By construction, such an isomorphism exists over $W$. Since the complement of $W$ has codimension higher than two and $\st B$ is smooth, the isomorphism extends.
\end{proof}

\subsection{Spaces of triple point singularities}\label{sec:triple_singularities}
An essential tool in our analysis of $\o T_{g;1}^l$ is an understanding of the stratification of the space of triple point singularities given by the $\mu$ invariant. We begin by recalling and generalizing the definition of the $\mu$ invariant.

Let $\Delta$ be the spectrum of a DVR and $\phi \from C \to \Delta$ a triple cover. Let $\tw C \to C$ be the normalization. Recall the case where $\tw C \to \Delta$ is \'etale. In this case, we look at the $\Delta$-module
\[ Q = O_{\tw C}/O_C = (O_{\tw C}/O_\Delta) / (O_C/O_\Delta),\]
which must be isomorphic to $k[t]/t^m \oplus k[t]/t^n$ for some $m$, $n$. We set $\mu(\phi) = |n-m|$.

If $\tw C \to \Delta$ is not \'etale, then consider a finite base change $\Delta' \to \Delta$, where $\Delta'$ is the spectrum of a DVR. Set $C' = C \times_\Delta \Delta'$ and $\phi' = \phi \times_{\Delta}\Delta'$. Let $\tw C' \to C'$ be the normalization. Assume that $\Delta' \to \Delta$ is such that $\tw C' \to \Delta'$ is \'etale. Define the $\mu$ invariant of $\phi$ as
\[ \mu(\phi) = \mu(\phi') / \deg(\Delta' \to \Delta).\]
It is easy to see that we can choose $\Delta' \to \Delta$ so that $\tw C' \to \Delta'$ is \'etale and this choice does not affect $\mu(\phi)$. 

Having defined the $\mu$ invariant for all triple point singularities, we analyze the loci of singularities of a fixed $\mu$ invariant. We treat `the space of triple point singularities' using the formalism of crimps of finite covers introduced in \citep[\autoref*{p1:sec:crimps}]{deopurkar12:_compac_hurwit}, which we first recall. Let $P$ be a smooth curve and $\tw\phi \from \tw C \to P$ a generically \'etale finite cover. In applications, $\tw C$ is typically smooth. A \emph{crimp} of $\tw\phi$ is the data $(C, \nu, \phi)$ forming
\[(\tw C \stackrel\nu\to C \stackrel\phi\to P),\]
where $\phi$ is a finite cover and $\nu$ an isomorphism generically over $P$. A convenient way to think of a crimp is as an $O_P$-subalgebra of $\tw\phi_* O_{\tw C}$, generically isomorphic to $\tw \phi_* O_{\tw C}$. 
\begin{proposition}\citep[\autoref*{p1:thm:crimps_projective}]{deopurkar12:_compac_hurwit}\label{thm:crimps}
  Let $\Sigma \subset P$ be a divisor. There exists a projective scheme $\Crimp(\tw\phi, \Sigma)$ parametrizing crimps $(\tw C \stackrel\nu\to C \stackrel\phi\to P)$ with $\br\phi = \Sigma$.
\end{proposition}

Take $P = \Delta$ and $\Sigma = b \cdot 0$. Then $\Crimp(\tw\phi, \Sigma)$ depends only on a formal neighborhood of $0$ in $\Delta$ (\citep[\autoref*{p1:thm:crimp_local}]{deopurkar12:_compac_hurwit}), and hence we may take $\Delta = \spec R$ with $R = k\f{t}$ without loss of generality.
\begin{proposition}\label{thm:delta_branch}
  Let $\tw C \to C \to \Delta$ be a crimp with branch divisor $b \cdot 0$ and $\tw C \to \Delta$ \'etale. Let $\delta = \length(O_{\tw C}/ O_C)$ be the $\delta$-invariant of $C$. Then $2\delta = b$.
\end{proposition}
\begin{proof}
  This is a simple verification. See \citep[\autoref*{p1:thm:delta_inv}]{deopurkar12:_compac_hurwit} for a detailed argument in a more general case.
\end{proof}

We analyze $\Crimp(\tw\phi, \Sigma)$ where $\tw\phi$ has degree three and $\tw C$ is smooth. Although tedious, this analysis is fairly straightforward. Observe that there are three possibilities for $\tw\phi$:
\begin{enumerate}
\item \'etale: $\tw C = \Delta \sqcup \Delta \sqcup \Delta \to \Delta$,
\item totally ramified: $\tw C = \spec R[x]/(x^3-t) \to \Delta$.
\item simply ramified: $\tw C = \Delta \sqcup \spec R[x]/(x^2-t) \to \Delta$,
\end{enumerate}
We are interested mainly in the stratification of  $\Crimp(\tw\phi, \Sigma)$ given by the $\mu$ invariant. Denote by $\Crimp(\tw\phi, \Sigma, l)$ be the locally closed subset (with the reduced scheme structure) of $\Crimp(\tw\phi, \Sigma)$ parametrizing crimps with $\mu$ invariant $l$. 

\subsubsection{$\Crimp(\tw\phi, \Sigma, l)$ for $\tw\phi$ \'etale}\label{sec:crimps_triple_etale}

Fix $\tw\phi\from \tw C = \spec \tw S \to \Delta$, a triple cover with $\tw\phi$ \'etale. Fix an isomorphism of $R$-algebras
\[\tw S \cong R \oplus R \oplus R.\] 
The space $\Crimp(\tw\phi, \Sigma, l)$ may be thought of as the parameter space of suitable $R$-subalgebras $S$ of $\tw S$. The next result characterizes the $R$-submodules of $S$ that are also $R$-subalgebras. We often identify $R$ with its diagonal embedding in $S$.
\begin{proposition}\label{thm:module_to_algebra}
  Let $S \subset \tw S$ be an $R$-module such that 
  \begin{equation}\label{eqn:etale_quotient_split}
    R \subset S \text{ and } \tw S/S \cong k[t]/t^m \oplus k[t]/t^n,
  \end{equation}
  with $m \leq n$. Then
  \begin{compactenum}
  \item\label{eqn:contains_tn} $S$ contains $t^n\tw S$.
  \item\label{eqn:one_generator} The quotient $S/\ideal{R, t^n\tw S}$ is an  $R$-submodule of $\tw S/\ideal{R, t^n\tw S}$ generated by the image of one element $t^mf$, for some $f \in \tw S$ nonzero modulo $\ideal{R,t}$.
  \item\label{eqn:subalgebra_condition} $S$ is an $R$-subalgebra if and only if $t^{2m}f^2 \in S$.
  \end{compactenum}
  In particular, if $2m \geq n$, then every $R$-submodule $S$ of $\tw S$ satisfying \eqref{eqn:etale_quotient_split} is an $R$-subalgebra.
\end{proposition}
\begin{proof}
  From \eqref{eqn:etale_quotient_split}, it follows that $S$ is generated as an $R$-module by $1$, $t^mf$ and $t^ng$ for some $f$ and $g$ in $\tw S$ such that $1$, $f$ and $g$ are linearly independent modulo $t$. The first two assertions follow from this observation. For the third, see that $S$ is closed under multiplication if and only if the pairwise products of the generators lie in $S$. By \eqref{eqn:contains_tn}, this is automatic for all products except $t^{2m}f^2$. Finally, if $2m \geq n$ then the condition \eqref{eqn:subalgebra_condition} is vacuous by \eqref{eqn:contains_tn}.
\end{proof}

We can now describe $\Crimp(\tw\phi, \Sigma, l)$.
\begin{proposition}\label{thm:crimps_triple_etale}
  Retain the setup introduced at the beginning of \autoref{sec:crimps_triple_etale}. Let $m$, $n$ be such that
  \[ n + m = b/2 \text{ and } n-m = l.\]
  First, $\Crimp(\tw\phi, \Sigma, l)$ is non-empty only if $m$ and $n$ are non-negative integers. If this numerical condition is satisfied, then we have the following two cases:
  \begin{enumerate}
  \item If $2m \geq n$, then $\Crimp(\tw\phi, \Sigma, l)$ is irreducible of dimension $l$. Its points correspond to $R$-subalgebras of $\tw S$ generated as an $R$-module by $1$, $t^n\tw S$ and $t^mf$ for some $f \in \tw S$ nonzero modulo $\ideal{R,t}$.
  \item If $n > 2m$, then $\Crimp(\tw\phi, \Sigma, l)$ is a disjoint union of three irreducible components of dimension $m$, conjugate under the $\Aut(\tw\phi) = \s_3$ action. Its points correspond to $R$-subalgebras of $\tw S$ generated as an $R$-module by $1$, $t^n\tw S$ and $t^mf$, where $f$ has the form
    \[ f = (1,h,-h) \text{ or } (h,1,-h) \text{ or } (h,-h,1),\]
    with $h \equiv 0 \pmod {t^{n-2m}}$.
  \end{enumerate}
\end{proposition}
\begin{proof}
  Let $\tw C \to C = \spec S \to \Delta$ be a crimp with $\mu$ invariant $l$ and branch divisor $\Sigma$. By \autoref{thm:delta_branch}, we have
  \[ \tw S/S \cong k[t]/t^m \oplus k[t]/t^n.\]
  In particular, $m$ and $n$ must be integers.

 The space $\Crimp(\tw\phi, \Sigma, l)$ may be identified with the space of $R$-modules $S$ satisfying
  \begin{equation}\label{eqn:sandwich}
    R \subset S \subset \tw S \text{ and } \tw S / S \cong k[t]/t^m \oplus k[t]/t^n,
  \end{equation}
  satisfying the additional condition that $S$ be closed under multiplication. Set $F = \tw S/R$. By \autoref{thm:module_to_algebra}~\eqref{eqn:one_generator}, the space of $S$ as in \eqref{eqn:sandwich} is simply the space of submodules of the $(R/t^{n-m}R)$-module $t^mF/t^nF$ generated by one element $t^mf$, where $f$ is nonzero modulo $t$. To specify such a submodule, it suffices to specify the image $\overline f$ of $f$ in $F/t^{n-m}F$, such that it is nonzero modulo $t$. Two such $\overline f$ define the same submodule if and only if they are related by multiplication by a unit of $R/t^{n-m}R$. 

In the case $2m \geq n$, the condition of being closed under multiplication is superfluous. Thus, $\Crimp(\tw\phi, \Sigma, l)$ may be identified with the quotient 
\[ (F/t^{n-m}F)^* / (R/t^{n-m}R)^* \cong ((k[t]/t^{n-m})^{\oplus 2})^* / (k[t]/t^{n-m})^*,\]
where the superscript $*$ denotes elements nonzero modulo $t$. This quotient is simply the jet-scheme of order $(n-m-1)$ jets of $\P^1 = \P_\sub(F/tF)$. In particular, it is irreducible of dimension $(n-m) = l$. 

In the case $n > 2m$, we must check when $S$ is closed under multiplication. It is not too hard to check that after multiplying by a unit of $R/t^{n-m}R$, the element $\overline f$  in $F/t^{n-m}F$ can be represented as the image in $F/t^{n-m}F$ of 
  \begin{equation}\label{eqn:three_choices}
    (1,h,-h) \text{ or } (h,1,-h), \text{ or } (h,-h,1), \text{ for some $h \in R/t^{n-m}R$.}
  \end{equation}
  It is easy to check that $t^{2m}f^2$ lies in the $R$-submodule $S$ of $\tw S$ generated by $1$, $t^n\tw S$ and $t^mf$ if and only if the element $t^{2m}\overline f^2 \in t^mF/t^nF$ lies in $R\ideal{t^m\o f}$, or equivalently, $h \equiv 0 \pmod {t^{n-2m}}$. Hence, the choice of $\overline f$ that gives an $R$-subalgebra $S$ is equivalent to the choice of $h$ from $t^{n-2m}R/t^{n-m}R \cong k^m$. Also, see that different choices of $h$ give different $S$. Hence, $\Crimp(\tw\phi, \Sigma, l)$ is the disjoint union of three irreducible components of dimension $m$ corresponding to the three possibilities in \eqref{eqn:three_choices}. Since the group $\s_3$ acts by permuting the three entries, these three components are conjugate. 
\end{proof}

\subsubsection{$\Crimp(\tw\phi, \Sigma,l)$ for $\tw\phi$ totally ramified}\label{sec:crimps_triple_total}
Fix $\tw\phi\from \tw C = \spec \tw S \to \Delta$, a triple cover with $\tw C$ smooth and $\tw\phi$ totally ramified. Let $\Delta' \to \Delta$ be the triple cover given by $R \to R' = R[s]/(s^3-t)$. Let $\tw S'$ be the normalization of $\tw S \otimes_R R'$ and set $\tw C' = \spec \tw S'$. Fix an isomorphism of $R$ algebras
\[\tw S \cong R[x]/(x^3-t),\]
and an isomorphism of $R'$ algebras
\[\tw S' \cong R' \oplus R' \oplus R',\]
such that the normalization map $\tw S \otimes_R R' \to \tw S'$ is given by 
\[ x \mapsto (s, \zeta s, \zeta^2s),\]
where $\zeta$ is a third root of unity. Identify $R$ with its image in $\tw S$ and $R'$ with its image in $\tw S'$. 

\begin{proposition}\label{thm:triple_tensor_up}
  Let $M \subset \tw S$ be an $R$-submodule of rank three. 
  \begin{compactenum}
  \item  $M$ is spanned by three elements $f_i(x) \in \tw S$, for $i = 1, 2, 3$, having $x$-valuations $v_i$ that are distinct modulo 3. 
  \item Set $M' = M \otimes_R R'$ and identify it with its image in $\tw S'$. Then
    \[\tw S'/M' \cong k[s]/s^{v_1} \oplus k[s]/s^{v_2} \oplus k[s]/s^{v_3}.\]
  \item $M$ contains $R$ and is closed under the multiplication map induced from $\tw S$ if and only if $M'$ contains $R'$ and is closed under the multiplication map induced from $\tw S'$.
\end{compactenum}
\end{proposition}
\begin{proof}
  Take an $R$-basis $\ideal{f_i}$ of $M$ with $f_i(x) = x^{v_i}g_i(x)$, where $g_i(0) \neq 0$ and the $v_i$ are distinct. Then $M' \subset \tw S'$ is spanned by the elements 
  \begin{equation*}\label{eqn:zeta_basis}
    s^{v_i}(g_i(s), \zeta^{v_i}g_i(\zeta s), \zeta^{2v_i}g_i(\zeta^2s)).
\end{equation*}
Since the $v_i$ are distinct and the three elements above are $R'$-linearly independent, the three vectors $(g_i(0), \zeta^{v_i}g_i(0), \zeta^{2v_i}g_i(0))$ must be $k$-linearly independent. It follows that the $v_i$ are distinct modulo $3$ and
\[ \tw S'/M' \cong k[s]/s^{v_1} \oplus k[s]/s^{v_2} \oplus k[s]/s^{v_3}.\]

For the last statement, see that $M$ contains $R$ if and only if the map $M \to \tw S/R$ is zero; $M$ is closed under multiplication if and only if the map $M \otimes_R M \to \tw S/M$ is zero. Both conditions can be checked after the extension $R \to R'$.
\end{proof}

Using \autoref{thm:triple_tensor_up} and our analysis of $R'$-subalgebras of the \'etale extension $R' \to \tw S'$ from \autoref{sec:crimps_triple_etale}, we get a description of $\Crimp(\tw\phi, \Sigma, l)$.
\begin{proposition}\label{thm:crimps_triple_total}
  Retain the setup introduced at the beginning of \autoref{sec:crimps_triple_total}. Let $m$, $n$ be such that 
  \[ n+m = 3b/2 \text{ and } n-m = 3l.\]
  First, $\Crimp(\tw\phi, \Sigma, l)$ is non-empty only if $m$ and $n$ are non-negative integers distinct and nonzero modulo $3$ and $2m \geq n$. If these numerical conditions are satisfied, then $\Crimp(\tw\phi, \Sigma, l)$ is irreducible of dimension $\lfloor l \rfloor$. Its points correspond to $R$-subalgebras of $\tw S$ generated as an $R$-module by $1$, $x^n\tw S$ and $x^mf \in \tw S$, with $f$ of the form
  \[ f = 1 + \sum_{\substack{0 < i < n-m \\ i \equiv m \pmod 3}}a_ix^i,\]
  for some $a_i \in k$. 
\end{proposition}
\begin{proof}
  Let $\tw C \to C = \spec S \to \Delta$ be a crimp with branch divisor given by $\ideal{t^b}$ and $\mu$ invariant $l$. Then $\tw C' \to C\times_\Delta \Delta' \to \Delta'$ is a crimp with branch divisor given by $\ideal{s^{3b}}$ and $\mu$ invariant $3l$. Set $S' = S \otimes_R R'$. Then
  \[ \tw S'/S' \cong k[s]/s^m \oplus k[s]/s^n.\]
  In particular, $m$ and $n$ must be integers. From \autoref{thm:triple_tensor_up}, $0$, $m$ and $n$ are distinct mod $3$. 

  $\Crimp(\tw\phi, \Sigma, l)$ may be identified with the space of $R$-modules $S$ satisfying
  \begin{equation}\label{eqn:total_quotient_split}
    R \subset S \subset \tw S \text{ and } \tw S'/S' \cong k[s]/s^m \oplus k[s]/s^n,
  \end{equation}
  (where $S' = S \otimes_R R'$) with the additional restriction that $S$ be closed under multiplication. Let $S$ be an $R$-submodule of $\tw S$ satisfying \eqref{eqn:total_quotient_split}. From \autoref{thm:triple_tensor_up}, $S$ is generated by $1$, $x^mf$ and $x^ng$ where $f$ and $g$ are nonzero modulo $x$. Then $S$ is determined by the image of $f$ in $\tw S/x^{n-m}\tw S$. For $S$ to be closed under multiplication, the image of $x^{2m}f^2$ in $\tw S/x^n \tw S$ must be an $R$-linear combination of $1$ and $x^mf$. Since all elements of $S$ lying in $R$ have $x$-valuation divisible by three, this is impossible unless $x^{2m}f^2 \equiv 0$ in $\tw S/x^n \tw S$; that is $2m \geq n$. 

  For $2m \geq n$, every $f \in \tw S/x^{n-m}\tw S$ nonzero modulo $\ideal{R,x}$ yields an $S$ satisfying \eqref{eqn:total_quotient_split} closed under multiplication. Two choices $f_1$ and $f_2$ determine the same $S$ if and only if they are related by 
  \[ x^mf_1 = ax^mf_2 + b,\]
  for $a, b \in R$ with $a$ invertible. It is not hard to check that $f$ can be chosen uniquely of the form 
  \[ f = 1 + \sum_{\substack{0 < i < n-m \\ i \equiv m \pmod 3}}  a_ix^i,\]
  for some $a_i \in k$. Therefore,  $\Crimp(\tw\phi, \Sigma, l)$ is  irreducible of dimension $\lfloor (n-m)/3 \rfloor = \lfloor l \rfloor$.
\end{proof}

\subsubsection{$\Crimp(\tw\phi, \Sigma,l)$ for $\tw\phi$ simply ramified}\label{sec:crimps_triple_simple}
Fix $\tw\phi\from \tw C = \spec \tw S \to \Delta$, a triple cover with $\tw C$ smooth and $\tw\phi$ simply ramified. Let $\Delta' \to \Delta$ be the double cover given by $R \to R' = R[s]/(s^2-t)$. Let $\tw S'$ be the normalization of $\tw S \otimes_R R'$ and set $\tw C' = \spec \tw S'$. Set $\tw S_1 = R[x]/(x^2-t)$ and fix an isomorphism of $R$-algebras
\[\tw S \cong \tw S_1 \oplus R,\]
and an isomorphism of $R'$-algebras
\[ \tw S' \cong R' \oplus R' \oplus R',\]
with the normalization map $\tw S \otimes_R R' \to \tw S'$ given by 
\[ (x,0) \mapsto (s, -s, 0) \quad (0,r) \mapsto (0,0,r).\]
Identify $R$ with its image in $\tw S$ and $R'$ with its image in $\tw S'$.

\begin{proposition}\label{thm:double_tensor_up}
  Let $M \subset \tw S$ be an $R$-submodule of rank three containing $R$.
  \begin{compactenum}
  \item  $M$ is spanned by three elements: $1$, $(f_1(x), 0)$ and $(f_2(x),0)$, with the $f_i(x)$ having $x$-valuations $v_i$ that are distinct modulo $2$.
  \item Set $M' = M \otimes_R R'$ and identify it with its image in $\tw S'$. Then
    \[\tw S'/M' \cong k[s]/s^{v_1} \oplus k[s]/s^{v_2}.\]
  \item $M$ is closed under the multiplication map induced from $\tw S$ if and only if $M'$ is closed under the multiplication map induced from $\tw S'$.
\end{compactenum}
\end{proposition}
\begin{proof}
  Identify $\tw S/R$ with $\tw S_1$. Then there is an equivalence between submodules of $\tw S$ containing $R$ and submodules of $\tw S_1$. With this modification, the proof is almost identical to the proof of \autoref{thm:triple_tensor_up}, with $\tw S_1$ playing the role of $\tw S$.
\end{proof}

Using \autoref{thm:double_tensor_up} and our analysis of $R'$ subalgebras of the \'etale extension $R' \to S'$ from \autoref{sec:crimps_triple_etale}, we get a description of $\Crimp(\tw\phi, \Sigma, l)$.
\begin{proposition}\label{thm:crimps_triple_simple}
  Retain the setup introduced at the beginning of \autoref{sec:crimps_triple_simple}. Let $m$, $n$ be such that
  \[ n + m = b \text{ and } n - m = 2l.\]
  First, $\Crimp(\tw\phi, l)$ is non-empty only if $n$ and $m$ are non-negative integers distinct modulo $2$. If these numerical conditions are satisfied, then we have the following two cases:
  \begin{compactenum}
  \item    If $2m \geq n$, then $\Crimp(\tw\phi, \Sigma, l)$ is  irreducible of dimension $\lfloor l\rfloor$. Its points correspond to $R$-subalgebras of $\tw S$ generated as an $R$-module by $1$, $x^n\tw S$ and $(x^mf, 0)$ for $f \in \tw S_1$ of the form
    \[ f = 1 + \sum_{\substack{0 < i < n-m\\ i \text{ odd}}} a_i x^i,\]
  \item If $n > 2m$, then we have two further cases:
    \begin{compactenum}
    \item If $m$ is odd, then $\Crimp(\tw\phi, \Sigma, l)$ is empty.
    \item If $m$ is even, then $\Crimp(\tw\phi, \Sigma, l)$ is irreducible of dimension $m/2$. Its points correspond to $R$-subalgebras of $\tw S$ generated as an $R$-module by $1$, $x^n\tw S$ and $(x^mf, 0)$ for $f \in \tw S_1$ of the form
      \[ f = 1 + \sum_{\substack{n-2m \leq i < n-m\\ i \text{ odd}}} a_i x^i,\]
      for some $a_i \in k$.
    \end{compactenum}
  \end{compactenum}
\end{proposition}
\begin{proof}
  Let $\tw C \to C = \spec S \to \Delta$ be a crimp with branch divisor given by $\ideal{t^b}$ and $\mu$ invariant $l$. Then $\tw C' \to C\times_\Delta \Delta' \to \Delta'$ is a crimp with branch divisor given by $\ideal{s^{2b}}$ and $\mu$ invariant $2l$. Set $S' = S \otimes_R R'$. Then
  \[ \tw S'/S' \cong k[s]/s^m \oplus k[s]/s^n.\]
  In particular, $m$ and $n$ must be integers. From \autoref{thm:triple_tensor_up}, $m$ and $n$ are distinct modulo $2$. 
  
  $\Crimp(\tw\phi, \Sigma, l)$ may be identified with the space of $R$-modules $S$ satisfying
  \begin{equation}\label{eqn:simple_quotient_split}
    R \subset S \subset \tw S \text{ and } \tw S'/S' \cong k[s]/s^m \oplus k[s]/s^n,
  \end{equation}
  (where $S' = S \otimes_R R'$) with the additional condition that $S$ be closed under multiplication. Let $S$ be an $R$-submodule of $\tw S$ satisfying \eqref{eqn:simple_quotient_split}. See that $S$ is determined by the image of $f$ in $\tw S_1/x^{n-m} \tw S_1$. For $S$ to be closed under multiplication $x^{2m}f^2 \in \tw S_1/x^n\tw S_1$ must be an $R$-multiple of $x^mf$ in $S_1/x^n\tw S_1$.

  In the case $2m \geq n$, any $f \in \tw S_1/x^{n-m}\tw S_1$ nonzero modulo $x$ yields an $S$ satisfying \eqref{eqn:simple_quotient_split} closed under multiplication. Two $f_1$ and $f_2$ give the same $S$ if and only if they are related by 
  \[ f_1 = a f_2,\]
  for some unit $a \in R$. It is not hard to check that $f$ can be chosen uniquely of the form
  \[ f = 1 + \sum_{\substack{0 < i < n-m \\ i \text{ odd}}} a_i x^i,\]
  for some $a_i \in k$. Thus,  $\Crimp(\tw\phi, \Sigma, l)$ is irreducible of dimension $\lfloor (n-m)/2 \rfloor = \lfloor l \rfloor$.

  In the case $n > 2m$, the condition for being closed under multiplication is non-vacuous. For this to hold, $x^{2m}f^2 \in \tw S_1/x^n\tw S_1$ must be an $R$-multiple of $x^mf$. Since elements of $R$ have even $x$-valuation, we conclude that $m$ must be even. In this case, $x^{2m}f^{2} \equiv x^mfg \pmod {x^n}$ for some $g \in R$ implies that the image of $f$  in $\tw S_1/x^{n-2m}\tw S_1$ is contained in the image of $R$. Thus the unique choice of $f$ as above must have the form
  \[ f = 1 + \sum_{\substack{n-2m \leq i < n-m \\ i \text{ odd}}} a_i x^i.\]
  Thus, $\Crimp(\tw\phi, \Sigma, l)$ is irreducible of dimension $m/2$.
\end{proof}

\subsection{Dimension counts}\label{sec:dimensions}
We use the results from \autoref{sec:triple_structure} and \autoref{sec:triple_singularities} to count the dimension of important loci in $\st T_{g;1}$. Let $0 \leq l \leq g$ be an integer with $l \equiv g \pmod 2$. Denote by $\st T_{g;1}(l) \subset \st T_{g;1}$ the locally closed locus consisting of $(P; \sigma; \phi \from C \to P)$ where $P \cong \P^1$ and $\phi$ has Maroni invariant $l$. Let $m \leq n$ be such that
\[ n+m = g+2 \text{ and } n-m = l.\]
\begin{proposition}\label{thm:dim_maroni}
  Let $0 \leq l \leq g$ be an integer with $l \equiv g \pmod 2$. 
  Then $\st T_{g;1}(l)$ is irreducible of dimension given by
  \[ \dim \st T_{g;1}(l) = 
  \begin{cases}
    2g + 2 & \text{ if $l = 0$,}\\
    2g + 3 - l & \text{ if $0 < l \leq (g+2)/3$,}\\
    (3g + l)/2 + 1 & \text{ if $(g+2)/3 < l$.}
  \end{cases}
  \]
  In particular, $\st T_{g;1}(l)$ has codimension one in the following two cases: $l = g$ and $l = 2$ (for even $g$). For $2 < l < g$, it has codimension at least two.
\end{proposition}
\begin{proof}
  Let $E = O_{\P^1}(m) \oplus O_{\P^1}(n)$ and set 
  \begin{align*}
    V &= H^0(\Sym^3(E) \otimes \det \dual E)\\ 
    &= H^0\left(O_{\P^1}(2m-n) \oplus O_{\P^1}(m) \oplus O_{\P^1}(n) \oplus O_{\P^1}(2n-m)\right).
  \end{align*}
  Using $n + m = g+2$ and $n - m = l \geq 0$, we get
  \begin{equation}\label{eqn:dim_upstairs}
  \dim V = 
  \begin{cases}
    2(g+2)+4 & \text{ if $2m \geq n$, i.e. $l \leq (g+2)/3$} \\
    3(g+l)/2 + 6 & \text{ if $2m < n$, i.e. $l > (g+2)/3$}
  \end{cases}.
\end{equation}
Using \autoref{thm:structure}, a point $v \in V$ gives a triple cover $\phi_v \from C \to \P^1$ with ${\phi_v}_*O_C/O_{\P^1} = \dual E$. Let $U \subset V \times \P^1$ be the open subset
  \[ U = \{(v, p) \mid p \not \in \br(\phi_v)\}.\]
  Then we have a surjective morphism $U \to \st T_{g;1}(l)$. Hence $\st T_{g;1}(l)$ is irreducible. The dimension of a general fiber of $U \to \st T_{g;1}(l)$ is simply
  \[\dim\Aut E + \dim \Aut (\P^1) = \dim \Aut E + 3.\]
  Hence
  \begin{equation}\label{eqn:dim}
    \dim \st T_{g;1}(l) = \dim U - \dim \Aut E -3 = \dim V - \dim \Aut E -2.
  \end{equation}
  Observe that
  \begin{equation}\label{eqn:dim_fiber}
    \dim \Aut E = \dim\Hom(E, E)
   = \begin{cases}
     4 & \text{ if $l = 0$, i.e. $m = n$}\\
     l+3 & \text{ if $l > 0$, i.e. $m < n$}
   \end{cases}.
 \end{equation}
 By combining \eqref{eqn:dim_upstairs}, \eqref{eqn:dim} and \eqref{eqn:dim_fiber}, we get the desired dimension count.
\end{proof}

Denote by $\st T_{g;1}^\bullet(l) \subset \st T_{g;1}$ the locally closed locus consisting of $(P; \sigma; \phi \from C \to P)$ where $P \cong \P^1$ and $\phi$ has concentrated branching with $\mu$ invariant $l$.
\begin{proposition}\label{thm:dim_mu}
  Let $0 \leq l \leq g$ and $l \equiv g \pmod 2$. Then $\st T_{g;1}^\bullet(l)$ is irreducible of dimension given by
  \[ \dim \st T_{g;1}^\bullet(l) =
  \begin{cases}
    l-1 & \text{ if $l \leq (g+2)/3$}\\
    (g-l)/2 & \text{ if $l > (g+2)/3$}
  \end{cases}.
  \]
  In particular, $\st T_{g;1}^\bullet(l) \subset \st T_{g;1}$ has codimension at least two.
\end{proposition}
\begin{proof}
  Note that for a point $(\phi \from C \to \P^1; \sigma)$ of $\st T^\bullet_{g;1}(l)$, the normalization $\tw C$ is isomorphic to $\P^1 \sqcup \P^1 \sqcup \P^1$. Thus, all the curves parametrized by $\st T_{g;1}^\bullet(l)$ are crimps of the fixed cover
  \[\phi \from \tw C = \P^1 \sqcup \P^1 \sqcup \P^1 \to \P^1.\]
  Let $\bullet \subset \st M_{0;b,1}$ be the closed substack consisting of $(P; \Sigma; \sigma)$ where $\Sigma$ is supported at a point. Then $\bullet$ has only one $k$-point, namely $p \to \bullet$ given by $(\P^1; b \cdot 0; \infty)$. Since $\Aut_p \st M_{0;b,1} = \G_m$, we have 
  \begin{equation}\label{eqn:dim_bullet}
    \dim \bullet = -1.
  \end{equation}
Furthermore, $\st T_{g;1}^\bullet(l)$ is contained in $\bullet \times_{\st M_{0;b,1}} \st T_{g;1}$. 

  Set $\tw C = \P^1 \sqcup \P^1 \sqcup \P^1$ and consider the space $\Crimp(\tw C \to \P^1; b \cdot 0)$. Recall that this is the moduli space of $(\tw C \to C \to \P^1)$, where the branch locus of $C \to \P^1$ is $b \cdot 0$. We have a natural morphism 
  \[\Crimp(\tw C \to \P^1; b \cdot 0) \to p \times_{\st M_{0;b,1}} \st T_{g;1},\]
  given by
  \[ (\tw C \to C \to \P^1) \mapsto (C \to \P^1; \infty).\]
  By \citep[\autoref*{p1:thm:crimp_fiber}]{deopurkar12:_compac_hurwit}, this morphism is finite. Hence, the locus $p \times_{\st M_{0;b,1}} \st T_{g;1}(l)$ has the same dimension as the locus in $\Crimp(\tw C \to \P^1; b \cdot 0)$ of crimps with $\mu$ invariant $l$. From the explicit description of this locus in \autoref{thm:crimps_triple_etale}, we get
  \begin{equation}\label{eqn:dim_crimps}
    \dim (p \times_{\st M_{0;b,1}}\st T^\bullet_{g;1}(l)) =
    \begin{cases}
      l &\text{ if $2m \leq n$, i.e. $l \leq (g+2)/3$,}\\
      (g+2-l)/2  &\text{ if $2m > n$, i.e. $l > (g+2)/3$}.
    \end{cases}
  \end{equation}
  By combining \eqref{eqn:dim_bullet} and \eqref{eqn:dim_crimps}, we get the desired dimension count. 
\end{proof}

Let $b = b_1 + \dots + b_n$ be a partition of $b$ with $b_i \geq 1$ and $n \geq 2$. Denote by $\st M_{0;b,1}(\{b_i\}) \subset \st M_{0;b,1}$ the locally closed locus consisting of $k$-points $(P; \Sigma; \sigma)$ where $P \cong \P^1$ and $\Sigma$ has the form $\Sigma = \sum_i b_i p_i$ for $n$ distinct points $p_1, \dots, p_n \in \P^1$. Set 
\[\st T_{g;1}(\{b_i\}) = \st M_{0;b,1}(\{b_i\}) \times_{\st M_{0;b,1}} \st T_{g;1}.\]
  \begin{proposition}\label{thm:dim_higher_collisions}
    With the above notation, we have
    \[ \dim \st T_{g;1}(\{b_i\}) \leq n - 2 + \sum_i \lfloor b_i/6 \rfloor.\]
    In particular, $\st T_{g;1}(\{b_i\})$ has codimension at least two if $n \leq b-2$.
  \end{proposition}
\begin{proof}
  First of all, see that $\dim \st M_{0;b,1}(\{b_i\}) = n - 2$. Next, we compute the dimensions of the fibers of $\br \from \st T_{g;1}(\{b_i\}) \to \st M_{0;b,1}(\{b_i\})$. Let $p \from \spec k \to \st T_{g;1}(\{b_i\})$ be a point, given by $(\P^1; \sigma; \phi \from C \to \P^1)$ with $\Sigma = \sum_i b_i p_i$. By \citep[\autoref*{p1:thm:crimp_fiber}]{deopurkar12:_compac_hurwit} and \citep[\autoref*{p1:thm:crimp_local}]{deopurkar12:_compac_hurwit}, the dimension of the fiber of $\br$ containing $p$ is simply the dimension of $\prod_i \Crimp(\tw\phi_i, b_i \cdot p_i)$, where $\tw\phi$ is the cover of the disk $\Delta_i$ around $p_i$ obtained by normalizing $C$. From the descriptions of $\Crimp(\tw\phi_i, b_i \cdot p_i)$ in \autoref{thm:crimps_triple_etale}, \autoref{thm:crimps_triple_total} and \autoref{thm:crimps_triple_simple}, we see that
    \[ \dim(\Crimp(\tw\phi_i, b_i \cdot p_i)) \leq \lfloor b_i/6 \rfloor.\]
    The result follows.
\end{proof}

A part of \autoref{intro:flips} follows immediately from the dimension counts. Recall that for a rational map $\beta \from X \dashrightarrow Y$, the exceptional locus $\Exc(\beta) \subset X$ is the closed subset where $\beta$ is not an isomorphism.
\begin{proposition}\label{thm:flips}
  For $2 < l < g$, the rational map $\beta_l \from \o{\sp T}_{g;1}^l \dashrightarrow \o{\sp T}_{g;1}^{l-2}$ is an isomorphism away from codimension two.
\end{proposition}
\begin{proof}
  $\Exc(\beta_l)$ is an open subset of the locus of covers of Maroni invariant $l$. By \autoref{thm:dim_maroni}, $\Exc(\beta_l) \subset \o{\sp T}_{g;1}^l$ has codimension at least two.
  
   $\Exc(\beta^{-1}_l)$ is an open subset of the locus of covers with concentrated branching and $\mu$ invariant $l$. By \autoref{thm:dim_mu}, $\Exc(\beta^{-1}_l) \subset \o{\sp T}_{g;1}^{l-2}$ has codimension at least two.
\end{proof}

\section{The Hyperelliptic contraction}\label{sec:hyperelliptic}
Let $g \geq 2$. In this section, we prove that $\beta_g \from \o{\sp T}_{g;1}^g \dashrightarrow \o{\sp T}_{g;1}^{g-2}$ is a divisorial contraction morphism. The idea is to analyze the exceptional loci $\Exc(\beta_g)$ and $\Exc(\beta_g^{-1})$; the result follows seamlessly from this analysis.

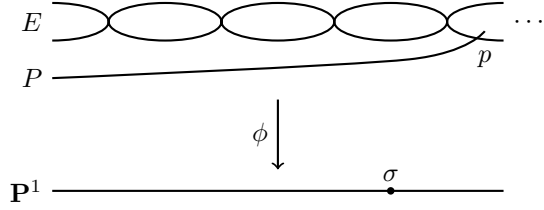
\begin{figure}[ht]
  \centering
  \begin{tikzpicture}[math, thick,scale=.25, baseline=-2.5em]
    \draw 
    (-6,0) ellipse (3 and 1)
    (0,0) ellipse (3 and 1)
    (6,0) ellipse (3 and 1)
    (12,1) arc (90:270:3 and 1) (12,0) node [right] {\cdots}
    +(-1,-2) node {p}
    (-12,-1) arc (-90:90:3 and 1) (-12,0)
    node [left] {E}
    ;
    
    \draw[smooth] plot 
    coordinates {(11,-.5) (7,-2) (-12,-3)}
    node [left] {P}
    ;
    
    \draw (12,-9) -- (-12,-9)
    node [left] {\P^1}
    ;
    \fill (6,-9) circle (.2) node[above] {\sigma};
    
    \path[->, shorten >= .1em, shorten <= .1em]
    (0,-4) edge node [left] {\phi} (0,-8);
  \end{tikzpicture}
  \caption{A generic point of the hyperelliptic divisor.}
  \label{fig:hyperelliptic}
\end{figure}

Set $H = \Exc(\beta_g) \subset \o{\sp T}_{g;1}^{g}$; this is the locus of curves with Maroni invariant $g$. By \autoref{thm:dim_maroni}, $H$ is irreducible of dimension $2g+1$. In other words, it is an irreducible divisor. We call $H$ the \emph{hyperelliptic divisor}. The terminology is justified by the following observation.

\begin{proposition}
  A (geometric) generic point of $H$ corresponds to $(\P^1; \sigma; \phi \from C \to \P^1)$ of the following form (see \autoref{fig:hyperelliptic}):
\begin{itemize}
\item $C = E \cup P$ with $P \cong \P^1$ and $E \cap P = \{p\}$,
\item $E$ is a smooth hyperelliptic curve of genus $g$ and $p \in E$ is a non-Weierstrass point,
\item $\phi \from E \to \P^1$ has degree two and $\phi \from P \to \P^1$ has degree one,
\item $\sigma \not \in \br\phi$.
\end{itemize}
\end{proposition}
\begin{proof}
  Let $(\P^1;\sigma; \phi \from C = E \cup P \to \P^1)$ be such a cover. Then we have 
  \begin{align*}
    h^0(\phi_* O_C \otimes O_{\P^1}(1)) = h^0(\phi^*O_{\P^1}(1)) = 3,
  \end{align*}
  which implies that $\phi_* O_C \cong O_{\P^1} \oplus O_{\P^1}(-1) \oplus O_{\P^1}(-g-1)$. Hence $\phi$ has Maroni invariant $g$. See that the covers $(\P^1; \sigma; \phi \from C = E \cup P \to \P^1)$ as above form a locus of dimension $2g+1$. Since this locus lies in $H$, which is irreducible of the same dimension, we conclude that this locus is dense in $H$.
\end{proof}

\begin{theorem}\label{thm:hyperelliptic}
  The birational map $\beta_g \from \o{\sp T}_{g;1}^g \dashrightarrow \o{\sp T}_{g;1}^{g-2}$ extends to a morphism. The extension contracts the hyperelliptic divisor $H$ to a point. The point $\beta_g(H)$ corresponds to the cover $(\P^1; \sigma; \phi \from C \to \P^1)$, where $\phi$ has concentrated branching and $C$ has a singularity of type $D_{2g+2}$.
\end{theorem}
We first prove a lemma (essentially \citep[Lemma~4.2]{smyth11:_modul_ii}) that gives a simple criterion to check whether an extension defined on $k$-points is in fact a morphism. 
\begin{lemma}\label{thm:pointwise_morphism}
  Let $X$ and $Y$ be algebraic spaces over $k$ with $X$ normal and $Y$ proper. Let $U \subset X$ be a dense open set and $\phi \from U \to Y$ a morphism. Let $\phi' \from X(k) \to Y(k)$ be a function that agrees with the one given by $\phi$ on $U(k)$. Assume that $\phi'$ is ``continuous in one-parameter families'' in the following sense: for every $\Delta$ which is the spectrum of a DVR with residue field $k$, and every morphism $\gamma \from \Delta \to X$ which sends $\Delta^\circ$ to $U$, we have
  \[\phi'(\gamma(0)) = (\phi\circ\gamma) (0),\]
  where the right hand side is the image of $0$ under the unique extension to $\Delta$ of $\phi \circ \gamma \from \Delta^\circ \to Y$. Then $\phi \from U \to Y$ extends to a morphism $\phi \from X \to Y$ that induces $\phi'$ on $k$-points. 
\end{lemma}
\begin{proof}
  Let $\o{\Phi} \subset X \times Y$ be the closure of the graph $\Phi$ of $\phi \from U \to Y$. Denote by $\pi_1 \from \o\Phi \to X$ and $\pi_2 \from \o\Phi \to Y$ the two projections. Then $\pi_1 \from \o\Phi \to X$ is proper and an isomorphism over $U$. We prove that it is an isomorphism. Then the extension of $\phi$ is obtained by composing $\pi_1^{-1} \from X \to \o\Phi$ with the projection $\pi_2 \from \o\Phi \to Y$.

  Since $X$ is normal and $\pi_1 \from \o\Phi \to X$ is proper and dominant, by Zariski's main theorem, it suffices to prove that $\pi_1$ is an injection on $k$-points.  Let $p_1, p_2 \in \o\Phi(k)$ be two points with $\pi_1(p_i) = x \in X(k)$. Since $\Phi$ is dense in $\o\Phi$, we can choose maps $\gamma_i \from \Delta \to \o\Phi$ which send $\Delta^\circ$ to $\Phi$ and $0$ to $p_i$. Since $\phi'$ is continuous in one-parameter families, we get
  \[ \pi_2(p_i) = \pi_2 \circ \gamma_i (0) = (\phi \circ \pi_1 \circ \gamma_i) (0) = \phi'(x).\]
  Since $\pi_1(p_i) = x$, we get $p_1 = p_2 = (x,\phi'(x))$.
\end{proof}

\begin{proof}[Proof of \autoref{thm:hyperelliptic}]
  Consider the exceptional locus $\Exc(\beta_{g}^{-1})$. Let $p \from \spec k \to \Exc(\beta^{-1}_g)$ be a point, given by a cover $(\P^1; \sigma; \phi \from C \to \P^1)$. Then $M(\phi) \leq g-1$ and $\phi$ has concentrated branching with $\mu(\phi) = g$. Without loss of generality, we may take $\sigma = \infty$ and $\br\phi = b\cdot 0$. The normalization $\tw C$ of $C$ is the disjoint union $\P^1 \sqcup \P^1 \sqcup \P^1$. Let $\spec k[x] \subset \P^1$ be the standard neighborhood of $0$. From \autoref{thm:crimps_triple_etale}, we see that, up to permuting the three components of $\tw C$ over $\P^1$, the subalgebra $O_C \subset O_{\tw C}$ is generated locally around $0$ as an $O_{\P^1}$ module by
  \[ 1, \quad x^{g+1}O_{\tw C}, \text{ and } (x, ax^g, -ax^g),\]
  for some $a \in k$. Observe that if $a = 0$, then $M(\phi) = g$, which is not allowed; hence $a \neq 0$. However, two covers given by nonzero $a, a' \in k$ are isomorphic via the morphism induced by the scaling 
  \[(\P^1, \infty, 0) \to (\P^1, \infty, 0), \quad x \mapsto \sqrt[g]{a/a'} x.\]
  We conclude that $\Exc(\beta_g^{-1})$ consists of a single point. Taking $ a = 1$, we see that the map $C \to \P^1$ is given locally by
  \[ k[x] \to k[x,y]/(y^2-x^{2g})(y-x).\]
  In particular, the singularity of $C$ is a $D_{2g+2}$ singularity.
  
  Next, consider the pointwise extension $\beta_g' \from \o{\sp T}_{g;1}^g(k) \to \o{\sp T}_{g;1}^{g-2}(k)$ which agrees with the one induced by $\beta_g$ on the complement of $H = \Exc(\beta_g)$ and sends all the points of $H$ to the unique point of $\Exc(\beta_g^{-1})$. It is clearly continuous in one-parameter families in the sense of \autoref{thm:pointwise_morphism}. Since $\o{\sp T}^g_{g;1}$ is normal and $\o{\sp T}_{g;1}^{g-2}$ proper, we conclude that $\beta_g$ extends to a morphism that contracts $H$ to a point.
\end{proof}

\section{The Maroni contraction}\label{sec:maroni}
Let $g \geq 4$ be even. In this section, we prove that $\beta_2 \from \o{\sp T}_{g;1}^2 \dashrightarrow \o{\sp T}_{g;1}^0$ is a divisorial contraction morphism. The idea is the same as in the case of the hyperelliptic contraction; we first define the extension on $k$-points and then argue that it is a morphism by checking continuity on one-parameter families. The details are a bit more involved as $\Exc(\beta_2^{-1})$ is not merely a point. The pointwise extension is obtained by relating the so-called cross-ratio of a marked unbalanced cover on one side and the so-called principal part of an unbalanced crimp on the other side. We begin by defining these two quantities.

For use throughout this section, set $V = k^{\oplus 3}/k$, where $k$ is diagonally embedded and $\P = \P_\sub V/\s_3$, where $\s_3$ acts on $V$ by permuting the three coordinates. The two dimensional vector space $V$ is to be thought of as the space of functions on $\{1, 2, 3\} \times \spec k$ modulo constant functions.

\subsection{The cross-ratio of a marked unbalanced cover}\label{sec:cross_ratio}
Consider a point $p \from \spec k \to \st T_{g;1}$ given by $(\P^1; \sigma; \phi \from C \to \P^1)$. Set $F = \phi_*O_C/O_{\P^1}$ and assume that
\[ F \cong O_{\P^1}(-m) \oplus O_{\P^1}(-n) \text{ with } 0 < m < n.\]
Define the \emph{cross-ratio of $\phi$ over $\sigma$} as a point of $\P_\sub(F|_\sigma)$ as the line given by 
\[k \cong H^0(F\otimes O_{\P^1}(m)) \into F|_\sigma \otimes O_{\P^1}(m) \cong F|_\sigma.\]
Since the isomorphisms on both sides are canonical up to the choice of a scalar, this line is well defined. An identification $C|_\sigma \isom \{1,2,3\}$ induces an identification $F|_\sigma \isom V$ and lets us treat the cross-ratio as a point of $\P_\sub V$. Let $\chi(p)$ be the image of the cross-ratio in $\P = \P_\sub V/\s_3$. Then $\chi(p)$ is independent of the identification $C|_\sigma \isom \{1,2,3\}$.

The name ``cross-ratio'' comes from the following geometric realization of $\chi(p)$. For simplicity, assume that $C$ is Gorenstein. We have the canonical embedding $C \into \F_M$, where $M = n-m$ and $\F_M = \P\dual F$ is a Hirzebruch surface. Let $\tau \from \P^1 \to \F_M$ be the unique section of negative self-intersection and $P \cong \P^1$  the fiber of $\F_M \to \P^1$ over $\sigma$. On $P$, we have four marked points, namely the three distinct points of $C|_\sigma$ and the point $\tau(\sigma)$. The element $\chi(p)$ is simply the moduli of $(P, C|_\sigma, \tau(\sigma))$. Even if $C$ is not Gorenstein, it is Gorenstein over an open set $U \subset \P^1$ containing $\sigma$. The geometric description of $\chi(p)$ goes through if we consider the restricted embedding $C|_U \into \F_M|_U$.

\subsection{The principal part of an unbalanced crimp}\label{sec:principal_part}
Analogous to the cross ratio, there is a $\P$-valued invariant of a cover with concentrated branching. This is a local invariant, so we consider $\tw\phi \from \tw C \to \Delta$, where $\Delta$ is the disk $\spec k\f{t}$ and $\tw\phi$ an \'etale triple cover. Let $\tw C \to C \stackrel\phi\to \Delta$ be a crimp; set $\tw F = O_{\tw C}/O_\Delta$ and $F = O_C/O_\Delta$ and $Q = O_{\tw C}/O_C = \tw F / F$. Assume that
\[ Q \cong k[t]/t^m \oplus k[t]/t^n \text{ with } 0 < m < n.\]
Then the map $i \from F \to \tw F$ is divisible by $t^m$ and the rank of the induced map
\[ t^{-m}i \from F|_0 \to \tw F|_0\]
is one. Define the \emph{principal part} of the crimp to be the point of $\P_\sub(\tw F|_0)$ given by the image of $t^{-m}i$. 

More explicitly, from \autoref{thm:crimps_triple_etale}, we know that $O_C$ is generated as an $O_\Delta$ module by  $1$, $t^mf$ and $t^nO_{\tw C}$, where $\overline f \in \tw F$ is nonzero modulo $t$. The principal part is simply the line $\langle f(0) \rangle \subset \tw F|_0$. Thus, it partially encodes the moduli of a crimp.

Finally, consider a point $p \from \spec k \to \st T_{g;1}$ given by $(\P^1; \sigma; \phi \from C \to \P^1)$, where $\phi$ has concentrated branching at $0$ with an unbalanced crimp, as above. Identifying $\tw C|_\sigma \isom \{1,2,3\}$ lets us treat the principal part as a point of $\P_\sub V$. Let $\rho(p)$ be the image of the principal part in $\P = \P_\sub V /\s_3$. Then $\rho(p)$ is independent of the identification $\tw C|_\sigma \isom \{1,2,3\}$.

The invariants $\chi(p)$ and $\rho(p)$ are equal in a particular case.
\begin{proposition}\label{thm:chi_rho_equal}
  Let $p = (\P^1; \sigma; \phi \from C \to \P^1)$ be such that $\phi$ has concentrated branching and $M(\phi) = \mu(\phi) > 0$. Then the cross-ratio equals the principal part: \[\chi(p) = \rho(p).\]
\end{proposition}
\begin{proof}
  Let $\tw C = \P^1 \sqcup \P^1 \sqcup \P^1 \to C$ be the normalization. Set $F = \phi_* O_C/O_{\P^1}$ and $\tw F = \tw\phi_* O_{\tw C}/O_{\P^1}$ as usual. Let the splitting type of $F$ and the splitting type of the singularity be $(m,n)$ with $m < n$.  Suppose $\supp\br(\phi) = \{0\}$ and let $O_{\P^1}(-1) \to O_{\P^1}$ be the ideal sheaf of $\{0\}$. The inclusion $i \from F \to \tw F$ factors through
  \[ i' \from F \otimes O_{\P^1}(m) \to \tw F \cong O_{\P^1}^{\oplus 2}.\]
  Clearly, it is an isomorphism away from $0$. Let $f$ be a nonzero global section of $F \otimes O_{\P^1}(m)$. Now, $\chi(p)$ is defined by the image of $f$ in $F \otimes O_{\P^1}(m)|_\sigma = \tw F|_\sigma$ and $\rho(p)$ by the image of $f$ in $\tw F|_0$. Since $\tw F$ is trivial, the two are equal.
\end{proof}

We now relate the cross-ratio and the principal part in the context of the blowing down of a trigonal curve. Let $X$ be a smooth surface, $s \in X$ a point, $\beta \from \Bl_sX \to X$ the blowup and $E$ the exceptional divisor. Let $P \subset X$ be a smooth curve through $s$, $\tw P$ its proper transform and $\tw s = E \cap \tw P$. Let $\tw C \to \Bl_sX$ be a triple cover, \'etale over $\tw P$, and set $\tw F = O_{\tw C}/O_{\Bl_s X}$. Assume that
\[\tw F|_E \cong O_E(-m) \oplus O_E(-n), \text{ with $0 < m < n$}.\]
Let $C \to X$ be the cover obtained by blowing down $\tw C \to \Bl_sX$ along $E$. Then $C|_P \to P$ is a crimp over $s$ with normalization $\tw C|_{\tw P}$. Assume that the singularity of $C|_P \to P$ over $s$ also has the splitting type $(m,n)$. 

\begin{proposition}\label{thm:blowdown_cross_ratio}
  In the above setup, the cross-ratio of $(E; \tw s; \tw C|_E \to E)$ and the 
  principal part of $(\tw C|_{\tw P} \to C|_P \to P)$ are equal.
\end{proposition}
  \begin{proof}
    By \autoref{thm:blow_down_trigonal_curves}, we have $O_C = \beta_* O_{\tw C}$. Denote by $\nu$ the map 
    \[ \nu \from \beta^*O_C|_{\tw P} \to O_{\tw C}|_{\tw P}.\]
    This map expresses $O_{\tw C}|_{\tw P}$ as the normalization of $\beta^*O_C|_{\tw P} = O_C|_P$. Set $F = O_C/O_X$. As the singularity of $C|_P \to P$ over $s$ has splitting type $(m,n)$, there is a nonzero section $f$ of $F$ defined around $s$ such that $\nu(\beta^*f|_{\tw P})$ has valuation $m$ with respect to a uniformizer of $\tw P$ at $\widetilde s$. By shrinking $X$ if necessary, assume that $f$ is defined on all of $X$. As $F = \beta_*\tw F$, we can interpret $f$ as a section of $\tw F$ on $\Bl_sX$. Since $\tw F|_E \cong O_E(-m) \oplus O_E(-n)$, the section $f$ must in fact be the image of a section $f'$ under the inclusion
    \[ \tw F \otimes I_E^m \to \tw F.\]
    Furthermore, since the section $f|_{\tw P}$ has valuation $m$ at $\tw s$, the section $f'|_{\tw P}$ has valuation zero at $\tw s$. Said differently, the restriction of $f'$ to $\tw s$ is nonzero. We see that the cross-ratio of $(E; \tw s; \tw C|_E \to E)$ and the principal part of $(\tw C|_{\tw P} \to C|_P \to P)$ are defined by the line spanned by the image of $f'$ in $\tw F|_{\tw s}$.
  \end{proof}

\subsection{The Maroni contraction}
We are now ready to tackle $\beta_2 \from \o{\sp T}_{g;1}^2 \dashrightarrow \o{\sp T}_{g;1}^0$. The exceptional locus $\Exc(\beta_2)$ is the locus in $\o{\sp T}_{g;1}^2$ consisting of covers with Maroni invariant two. By \autoref{thm:dim_maroni}, this locus is an irreducible divisor. We call it the \emph{Maroni divisor}. 

\begin{theorem}\label{thm:maroni}
  Let $g \geq 4$ be even. The birational map $\beta_2 \from \o{\sp T}_{g;1}^2 \dashrightarrow \o{\sp T}_{g;1}^{0}$ extends to a morphism. The extension contracts the Maroni divisor to $\P \cong \P^1$.
\end{theorem}
\begin{proof}
  Set $g = 2h$. Consider the exceptional locus $\Exc(\beta_2^{-1}) \subset \o{\sp T}_{g;1}^0$. Let $p \from \spec k \to \Exc(\beta_2^{-1})$ be a point, given by $(\P^1; \sigma; \phi \from C \to \P^1)$. Then $M(\phi) = 0$ and $\phi$ has concentrated branching with $\mu(\phi) = 2$. Without loss of generality, we may take $\sigma = \infty$ and $\br\phi = b \cdot 0$. The normalization $\tw C$ of $C$ is the disjoint union $\P^1 \sqcup \P^1 \sqcup \P^1$. Let $\spec k[x] \subset \P^1$ be the standard neighborhood of $0$. From \autoref{thm:crimps_triple_etale}, the subalgebra $O_C \subset O_{\tw C}$ is generated locally around $0$ as an $O_{\P^1}$ module by
  \[ 1,\quad x^{h+1}O_{\tw C} \text{ and } x^{h-1}f,\]
  for some $f \in O_{\tw C}$ whose image in $\tw F = O_{\tw C}/O_{\P^1}$ is nonzero modulo $x$. Clearly $O_C$ is determined by $\o f \in \tw F/x^2\tw F$ and $\o f$ only matters up to multiplication by a unit in $k[x]/x^2$. Let $\o f = f_1 + x f_2$, where $f_i \in \tw F|_0$ with $f_1 \neq 0$. By multiplying by units of $k[x]/x^2$, we see that $O_C$ is determined by a line $\ideal{f_1} \subset \tw F|_0$ and an element $\o f_2$ in the one-dimensional $k$-vector space $\tw F|_0/\ideal{f_1}$. However, if $\o f_2 = 0$, then $M(\phi) = 2$, which is not allowed. On the other hand, two covers given by $f_1 + xf_2$ and $f_1 + axf_2$, for $a \in k^*$, are isomorphic via the map induced by the scaling
  \[ (\P^1, \infty, 0) \to (\P^1, \infty, 0), \quad x \mapsto ax.\]
  The upshot is that $p \in \Exc(\beta_2^{-1})$ is determined by the line $\langle f_1 \rangle \subset \tw F|_0$, or equivalently by the principal part $\rho(p) \in \P$.

  We now define an extension $\beta_2'$ of $\beta_2 \from \o{\sp T}_{g;1}^2 \dashrightarrow \o{\sp T}_{g;1}^0$ on $k$-points. Let $p \in \Exc(\beta_2)$ be a point corresponding to $(\P^1; \sigma; \phi \from C \to \P^1)$ with $M(\phi) = 2$. Let $\beta'_2(p)$ be the unique point of $\Exc(\beta_2^{-1})$ whose principal part equals the cross-ratio $\chi(p)$ as points of $\P$.

  By \autoref{thm:pointwise_morphism}, it suffices to check that $\beta'_2$ is continuous in one-parameter families. Let $\Delta \to \o{\sp T}_{g;1}^2$ be a map sending $\Delta^\circ$ to the complement of $\Exc(\beta_2)$ and $0$ to a point $p \in \Exc(\beta_2)$. By replacing $\Delta$ by a finite cover, assume that the map lifts to $\Delta \to \o{\orb T}_{g;1}^2$ and is given by the family $(P; \sigma; \phi \from C \to P)$ over $\Delta$. From the proof of the valuative criterion for $\o{\orb T}_{g;1}^l$ (\autoref{thm:tg3d}), we know the procedure to modify $(P; \sigma; \phi \from C \to P)$ so that the central fiber lies in $\o{\orb T}_{g;1}^0$. After a sufficiently large base change, it involves blowing up $\sigma(0)$ and blowing down the proper transform of $P|_0$, until the central fiber has Maroni invariant $0$. Throughout this process, the central fiber continues to have $\mu$ invariant $2$. By repeated applications of \autoref{thm:chi_rho_equal} and \autoref{thm:blowdown_cross_ratio}, we conclude that the principal part of the resulting limit equals the cross-ratio of the original central fiber. It follows that $\beta_2'$ is continuous in one-parameter families.
 \end{proof}

 \section{The Picard group}\label{sec:picard}
 In this section, we compute the rational Picard groups and the canonical divisors of $\o{\sp T}_{g;1}^l$ for $0 < l < g$. By \autoref{thm:flips}, in this range, all the models $\o{\sp T}_{g;1}^l$ are isomorphic to one another away from codimension two. Hence their Picard groups are identical. Denoting by $\Pic_\Q$ the rational Picard group $\Pic \otimes \Q$, we have the equality
 \[ \Pic_\Q (\o{\orb T}_{g;1}^l) = \Pic_\Q (\o{\sp T}_{g;1}^l).\]
Moreover, it is clear that the locus of points with nontrivial automorphisms has codimension at least two in $\o{\orb T}_{g;1}^l$. As a result, the coarse space morphism $\o{\orb T}_{g;1}^l \to \o{\sp T}_{g;1}^l$ is an isomorphism in codimension one. Hence, we may transfer divisors from one to the other without worrying about multiplication factors.

 We begin by defining several classes in $\Pic_\Q(\o{\orb T}_{g;1}^l)$. Let $(\orb P; \sigma; \phi \from \orb C \to \orb P)$ be the universal family over $\o{\orb T}_{g;1}^l$. Denote by $\pi_{\orb P} \from \orb P \to \o{\orb T}_{g;1}^l$ and $\pi_{\orb C} \from \orb C \to \o{\orb T}_{g;1}^l$ the projections. When no confusion is likely, we denote both projections by $\pi$. Let $\Sigma \subset \orb P$ be the branch divisor $\Sigma = \br\phi$.

We define a number of elements of $\Pic_\Q(\o T_{g;1}^l)$.
\begin{compactenum}
  \item[$\lambda$: ] Observe that $R^1{\pi_C}_*O_{\orb C}$ is a vector bundle of rank $g$. We define $\lambda$ as the determinant of its dual
    \[ \lambda = \det \dual{(R^1{\pi}_*O_{\orb C})}.\]
  \item[$\delta$: ] The locus of points in $\o{\orb T}_{g;1}^l$ over which $\orb C$ is singular is a divisor. We define $\delta$ to be its class.
  \item[$T$: ] The locus of points in $\o{\orb T}_{g;1}^l$ over which $\phi$ has a point of triple ramification is a divisor. We define $T$ to be its class.
  \item[$\Br^2$: ] Define $\Br^2$ as the pushforward
    \[ \Br^2 = \pi_*(\Sigma^2).\]
  \item[$c_1^2$: ] Define $c_1^2$ as the pushforward
    \[ c_1^2 = \pi_*\left(c_1(\phi_*O_{\orb C})^2[\orb P]\right).\]
  \item[$c_2$: ] Define $c_2$ as the pushforward
    \[ c_2 = \pi_*\left(c_2(\phi_*O_{\orb C})[\orb P]\right).\]
  \item[$\sigma^2$: ] Define $\sigma^2$ as the pushforward
    \[ \sigma^2 = \pi_* \left(\sigma(\o{\orb T}_{g;1}^l)^2\right).\]
  \item[$K$: ] Define $K$ to be the canonical divisor
    \[ K = K_{\o{\orb T}_{g;1}^l} = K_{\o{\sp T}_{g;1}^l}.\]
  \end{compactenum}

  \begin{proposition}\label{thm:picard}
    Let $0 < l < g$ and $l \equiv g \pmod 2$. Then 
    \[\Pic_\Q(\o{\orb T}_{g;1}^l) = \Q\langle c_1^2, c_2 \rangle \cong \Q^{\oplus 2}.\]
  \end{proposition}
    The reason we prefer $c_1^2$ and $c_2$ as a basis is that their cohomological nature allows us to compute intersections with curves very easily.
  \begin{proof}
    We may throw away loci of codimension at least two. Consider the open subset $V_{g;1}$ (resp. $U_{g;1}$) of $\o{\orb T}^l_{g;1}$ consisting of $(P; \sigma; \phi \from C \to P)$ where $\br(\phi)$ has at least $(b-1)$ (resp. $b$) points in its support. By \autoref{thm:dim_higher_collisions}, the complement of $V_{g;1}$ has codimension at least two. Hence 
    \[ \Pic_\Q(\o{\orb T}_{g;1}^l) = \Pic_\Q(V_{g;1}).\]
    On the other hand, it is easy to see that the complement of $U_{g;1}$ in $V_{g;1}$ consists of two irreducible divisors, namely $T$ (the divisor where $\phi$ has a triple ramification point) and $\delta$ (the divisor where $C$ is singular). We thus have an exact sequence
    \begin{equation}\label{eq:pic_seq}
 \Q\langle T, \delta \rangle \to \Pic_\Q(V_{g;1}) \to \Pic_\Q(U_{g;1}) \to 0.
\end{equation}
We claim that $\Pic_\Q (U_{g;1}) = 0$. Indeed, consider the moduli space $U_g$ of unmarked trigonal curves, namely
    \[ U_g = \{(P; \phi \from C \to P)\},\]
    where $P \cong \P^1$, $C$ is a smooth and connected curve of genus $g$ and $\phi$ is simply branched. From \citep[Proposition~12.1]{Stankova-Frenkel00:_Modul_Of_Trigon_Curves} or \citep[Theorem~1.1]{Bolognesi09:_Stack_Of_Trigon_Curves}, we know that $\Pic_\Q(U_g) = 0$. Let $\phi \from \orb C \to \orb P$ be the universal object over $U_g$. Then $\orb P \to U_{g}$ is a conic bundle and $U_{g;1} \subset \orb P$ is the complement of $\br(\phi)$. We conclude that $\Pic_\Q (U_{g;1}) = 0$. From the sequence \eqref{eq:pic_seq}, we conclude that the dimension of $\Pic_\Q(\o{\orb T}_{g;1}^l)$ is at most two. It is easy to verify (for example, by intersecting with test curves) that $c_1^2$ and $c_2$ are linearly independent.
  \end{proof}
  
  We now express all the divisors described above in terms of $c_1^2$ and $c_2$. We are particularly interested in the expressions for $\lambda$, $\delta$ and $K$. As in the proof of \autoref{thm:picard}, we may throw away loci of codimension at least two. Accordingly, let $W_{0;b,1} \subset \st M_{0;b,1}$ be the open locus consisting of points $(P; \sigma; \Sigma)$, where $P \cong \P^1$ and $\supp\Sigma$ has at least $(b-1)$ points. Then $W_{0;b,1} \subset \st M_{0;b,1}$ has complement of codimension at least two. Likewise, let $V_{g;1}$ be the preimage of $W_{0;b,1}$ in $\o{\orb T}^l_{g;1}$. Then $V_{g;1} \subset \o{\orb T}_{g;1}^l$ also has complement of codimension at least two.

\begin{proposition}\label{thm:pullback_disc}
  Denote by $D \subset W_{0;b,1}$ the divisor consisting of points where $\Sigma$ is not reduced. Then
  \[ \br^* D = 3 T + \delta.\]
\end{proposition}
\begin{proof}
  Consider a point $w$ of $D$ given by $(P; \Sigma; \sigma)$. Then $\Sigma = 2p + p_3 + \dots + p_{b}$ for distinct points $p, p_3, \dots, p_{b} \in P$. Consider a point $v$ of $V_{g;1}$ over $(P; \Sigma; \sigma)$ given by a triple cover $\phi \from C \to P$. Then $\phi$ either has a triple ramification point or a node over $p$. Hence, 
  \[\supp \br^* D = \supp T\ \cup\ \supp \delta.\]
  We must verify the multiplicities. One way to compute the multiplicities is to look at the morphism $\Def_\phi \to \Def_{(P; \Sigma)}$. Let $U \to \P^1$ be a neighborhood of $p$ and $\phi_U \from C|_U \to U$ the restriction of $\phi$; it suffices to look at $\Def_{\phi_U} \to \Def_{U, \Sigma|_U}$. In fact, we may even restrict to an \'etale neighborhood.

  Take the case where $v \in T$. Then, \'etale locally around $p$, the cover $\phi$ has the from 
  \[\phi \equiv \spec k[x,y]/(x-y^3) \to \spec k[x].\] 
  The branch divisor is given by $\Sigma = x^2$. A versal deformation of the subscheme $\Sigma \subset \spec k[x]$ can be given over $R = k[u,v]$ as the family 
  \[\spec R[x]/(x^2+ux+v).\]
  The divisor $D \subset \spec R$ corresponding to non-reduced $\Sigma$ is given by $\ideal{u^2-4v}$. A versal deformation of $\phi$ can be given over $S = k[a,b]$ as the family 
  \[\spec S[x,y]/(x-y^3-ay-b) \to \spec S[x].\]
  The divisor $T \subset \spec S$ corresponding to covers with a triple ramification point is given by $\ideal{a}$. On the other hand, the branch divisor in $\spec S[x]$ is $\ideal{27(x-b)^2 + 4a^3}$. Thus, under the induced map $\spec S \to \spec R$, the pullback of $\ideal{u^2-4v}$ is $\ideal{a^3}$. In other words, $T$ appears with multiplicity three in $\br^*D$.

  The case where $v \in \delta$ is similar. \'Etale locally around $p$, the cover $\phi$ has the form 
  \[\phi \equiv \spec k[x] \sqcup \spec k[x,y]/(x^2-y^2) \to \spec k[x].\]
  A versal deformation of $\phi$ can be given over $S' = k[c]$ as the family
  \[ \spec S'[x] \sqcup \spec S'[x,y]/(x^2-y^2-c) \to \spec S'[x].\]
  The divisor $\delta \subset \spec S'$ corresponding to singular covers is given by $\ideal{c}$. On the other hand, the branch divisor in $\spec S'[x]$ is $\ideal{x^2-c}$. Thus, under the induced map $\spec S \to \spec R$, the pullback of $\ideal{u^2-4v}$ is $\ideal{c}$. In other words, $\delta$ appears with multiplicity one in $\br^*D$.
\end{proof}
\begin{proof}[Another proof of \autoref{thm:pullback_disc} over the complex numbers]
  Assume that $k = \C$. In this case, we can compute the multiplicities of $T$ and $\delta$ in $\br^*D$ by a beautiful topological argument due to Joe Harris. Take a point $w \in D$ and let $v \in V_{g;1}$ be a point over $w$. The idea is to compute the multiplicity of the component of $\br^*D$ containing $v$ by analyzing the monodromy of $V_{g;1} \to W_{0;b,1}$ around $w$ near $v$. Concretely, this can be done by lifting a (real) loop around $w$, beginning at a point near $v$.

  Let $w \equiv 2p + p_3 + \dots + p_b$. Take a nearby point $w' \equiv p_1 + p_2 + p_3 + \dots p_b$, where $p_1 \neq p_2$. Take a loop $\gamma$ in $W_{0;b,1}$ based at $w'$ that exchanges $p_1$ and $p_2$ in the standard way:
    
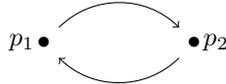
\[
    \begin{tikzpicture}[math]
      \draw (-1,0) node (p1) {\bullet}  node [left] {p_1};
      \draw (1,0) node (p2)  {\bullet}  node [right] {p_2};
      \path [->, bend left=45] 
      (p1) edge (p2)
      (p2) edge (p1);
    \end{tikzpicture}
    \]
    Set $\pi_1 = \pi_1(\P^1\setminus\{p_1, \dots, p_b\})$. Take a point $v'$ over $w'$ near $v$. Then $v'$ is determined by some monodromy data $\pi_1 \to \s_3$. Choose a basepoint $O$ in $\P^1$ and represent the generators of $\pi_1$ by the standard circuits around $p_i$:
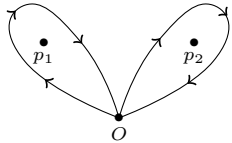
    \[
    \begin{tikzpicture}[math, map, scale=.5]
      \draw (-2,0) node (p2)  {\bullet}  node [below] {p_1};
      \draw (2,0) node (p2)  {\bullet}  node [below] {p_2};
      \draw (0,-2) node (O) {\bullet} node [below] {O};
      \draw      [smooth, tension=1]
      [postaction={decorate, decoration={markings,
          mark=at position 0.25 with \arrow[thick]{>};,
          mark=at position 0.5 with \arrow[thick]{>};,
          mark=at position 0.75 with \arrow[thick]{>};
        }}]
      plot coordinates {(O) (-2.5,-.5) (-2.5, 1) (-1,0) (O)};
      
      \draw      [smooth, tension=1]
      [postaction={decorate, decoration={markings,
          mark=at position 0.25 with \arrow[thick]{<};,
          mark=at position 0.5 with \arrow[thick]{<};,
          mark=at position 0.75 with \arrow[thick]{<};
        }}]
      plot coordinates {(O) (2.5,-.5) (2.5, 1) (1,0) (O)};
    \end{tikzpicture}
    \]
    Express the monodromy $\pi_1 \to \s_3$ by the $b$-tuple $(\sigma_1, \dots, \sigma_b)$ where $\sigma_i$ is the image of the standard loop around $p_i$ described above. Observe that the lift of $\gamma$ which starts at $v' \equiv (\sigma_1, \sigma_2, \dots)$ ends at $v'' \equiv (\sigma_2, \sigma_2^{-1}\sigma_1\sigma_2, \dots)$.    The order of the operation $v' \mapsto v''$ is the multiplicity of the component of $\br^*D$ containing $v$. For $v \in \delta$, we can take
    \[ \sigma_1 = \sigma_2 = (12).\]
    Then $v' = v''$, and hence $\delta$ appears with multiplicity one in $\br^*D$. For $v \in T$, we can take
    \[ \sigma_1 = (12), \quad \sigma_2 = (23).\]
    Then $v' \mapsto v''$ has order three; the cycle is given by
    \[ (12), (23) \mapsto (23), (13) \mapsto (13) (12) \mapsto (12), (23).\]
    Hence $T$ appears with multiplicity three in $\br^*D$.
  \end{proof}

  \begin{proposition}\label{thm:divisor_relations}
    Let $l$ be such that $0 < l < g$ and $l \equiv g \pmod 2$. Then the following relations hold in $\Pic_\Q(\o{\orb T}_{g;1}^l)$:
    \begin{align*}
    \Br^2 &= 4c_1^2,\\
    \lambda &= \frac{g+1}{2(g+2)}c_1^2 - c_2,\\
    T &= 3c_2,\\
    \delta &= \frac{4g+6}{(g+2)}c_1^2 - 9 c_2,\\
    \sigma^2 &= -\frac{1}{(g+2)^2}c_1^2.\\
  \end{align*}
  \end{proposition}
  \begin{proof}
    The relations follow from straightforward Chern class calculations. The first few are scattered throughout  \citep{Stankova-Frenkel00:_Modul_Of_Trigon_Curves}. We present the details for completeness.

    Take a one-parameter family of triple covers 
  \[C \stackrel\phi\to P \stackrel\pi\to B,\]
  where $B$ is a smooth projective curve, $P \to B$ a $\P^1$ bundle and $C \to B$ a family of connected curves of genus $g$. Let $\Sigma = \br\phi \subset P$. Set $E = \dual{(\phi_* O_C/O_P)}$. Then $c_1(\phi_*O_C) = -c_1(E)$ and $c_2(\phi_*O_C) = c_2(E)$. Choose a (possibly rational) class $\zeta$ on $P$ which has degree one on the fibers of $\pi$ and satisfies $\zeta^2 = 0$. In the calculations that follow, we omit writing pullbacks or push-forwards where they are clear by context. We use $[\ ]$ to denote the class of a divisor in the Chow ring. Chern classes are understood to be applied to the fundamental class.

  Since $\Sigma = \br\phi$ is the zero locus a section of ${(\det\dual{\phi_*O_C})}^{\otimes 2}$, we have 
  \[[\Sigma] = 2c_1(E).\]
  This gives the first relation.

  Using that $\Sigma$ has relative degree $b = 2(g+2)$ over $B$ and $\zeta^2 = 0$, we get
   \begin{equation}\label{eqn:zs}
     [\Sigma] \cdot \zeta = \frac{c_1^2(E)}{g+2} \text{ and } c_1(E) \cdot \zeta = \frac{c_1^2(E)}{2(g+2)}.
   \end{equation}

 From Grothendieck--Riemann--Roch applied to $\pi \from P \to B$, we get
 \begin{align*}
   \ch(R\pi_*O_C) &= g + \lambda \\
   &= \pi_*\left(\ch(\phi_*O_C)\cdot\td(P/B)\right)\\
   &= \pi_*\left(\left(1-c_1(E)+\frac{c_1^2(E)-2c_2(E)}{2}\right) \cdot (1-\zeta)\right).\\
 \end{align*}
Combining with \eqref{eqn:zs}, we get the second relation
 \begin{equation*}\label{eqn:lambda}
   \begin{split}
     \lambda &= \frac{c_1^2(E)}{2} - c_2(E) + \Sigma \cdot \zeta\\
     &= \frac{g+1}{2(g+2)}c_1^2 - c_2.
 \end{split}
\end{equation*}
 
We have an embedding $C \into \P E$ over $P$ that exhibits $C$ as the zero locus of a section $\alpha$ of the line bundle $L =  O_{\P E}(3)\otimes\dual{\det E}$.  Set $\xi = c_1(O_{\P E}(1))$. Then $\xi$ satisfies the equation 
    \[\xi^2 - c_1(E)\xi + c_2(E) = 0.\]
    Let $J_3L$ be the bundle of order three jets of $L$ along the fibers of $\pi_E \from \P E \to P$. More precisely,
    \[ J_3 L = {\pi_1}_* \left({\pi_2}^*L \otimes O_{3\Delta}\right),\]
    where $\pi_i$ are the two projections $\P E \times_P \P E \to P$ and $\Delta \subset \P E \times_P \P E$ the diagonal divisor. Then the locus $T \subset B$ of points $b \in B$ over which $C|_b \to P|_b$ has a triple ramification point is simply the image in $B$ of the zero locus of the section $J_3\alpha$ of $J_3L$ induced from the section $\alpha$ of $L$. In particular,
    \[ \deg T = c_3(J_3L).\]
    Since $J_3 L = L + L \otimes \Omega_{\P E/ P} + L \otimes \Omega^{\otimes2}_{\P E/P}$ in the Grothendieck group of sheaves on $\P E$ and $c_1(\Omega_{\P E/P}) = -2\xi+c_1(E)$, we get the third relation
    \begin{equation*}
      \begin{split}
        \deg T &= {\pi_E}_*(3\xi-c_1(E))\cdot\xi\cdot (-\xi + c_1(E))\\
        &= 3c_2(E) = 3c_2.
      \end{split}
    \end{equation*}

    Our next goal is to compute the divisor $\br^*D$. This is simply the branch divisor of $\Sigma \to B$. The ramification divisor of $\Sigma \to B$ is given by $\omega_{\Sigma/B}$. By adjunction, we have
    \begin{align*}
      \omega_{\Sigma/B} &= (\omega_{P/B} + c_1(\Sigma))|_\Sigma \\
      &= (-2\zeta + 2c_1(E)) \cdot 2c_1(E)\\
      &= \frac{4g+6}{g+2}c_1^2(E),
    \end{align*}
    and hence $\br^*D = \frac{4g+6}{g+2}c_1^2$.    Using \autoref{thm:pullback_disc}, we get the third relation
    \begin{align*}
      \delta &= \br^*D - 3T \\
      &= \frac{4g+6}{g+2}c_1^2 - 9c_2.
    \end{align*}
    
    For the last relation, assume furthermore that we have a section $\sigma \from B \to P$ disjoint from $\Sigma$. Abusing notation, also denote by $\sigma$ the image $\sigma(B) \subset P$. Now, $[\sigma]$ and $\frac{c_1(E)}{g+2}$ are two divisor classes on $P$ that have degree one on the fibers and their product is zero. The last relation follows.
  \end{proof}

  \begin{proposition}\label{thm:canonical}
    Let $0 < l < g$ and $l \equiv g \pmod 2$. The canonical divisor of 
    $\o{\orb T}_{g;1}^l$ (or $\o{\sp T}_{g;1}^l$) is given by
    \begin{align*}
      K &= -\frac {2(g+3)(2g+3)}{(g+2)^2} c_1^2 + 6c_2 \\
      &= \frac{2}{(g+2)(g-3)}\left(3(2g+3)(g-1) \lambda - (g^2-3)\delta\right).
    \end{align*}
  \end{proposition}
  \begin{proof}
    Retain the notation introduced just before \autoref{thm:pullback_disc}.
    We restrict to the open subset $V_{g;1}$. Let $K_V$ (resp. $K_W$) be the canonical divisor of $V_{g;1}$ (resp. $W_{0;b,1}$). Since the finite morphism $V_{g;1} \to W_{0;b,1}$ is \'etale except over $D \subset W_{0;b,1}$ and $\br^* D = 3T + \delta$, we have
    \begin{equation}\label{eqn:branch_rh}
      K_V = \br^* K_W + 2T.
    \end{equation}
    We first compute $K_W$ in terms of $D$. Let $w$ be a point of $W_{0;b,1}$ corresponding to the data $(P; \Sigma; \sigma)$. We have the following canonical identification of the tangent space to $W_{0;b,1}$ at $w$:
    \[ T_w W_{0;b,1} = \Hom(I_\Sigma/I_\Sigma^2, O_\Sigma) \oplus \Hom(I_\sigma/I_\sigma^2, O_\sigma).\]
    With this, the relation between $D$ and $K_W$ follows from an easy test-curve calculation, as follows. Take a one parameter family $(\pi \from P \to B; \sigma; \Sigma)$ giving a map $f \from B \to W_{0;b,1}$, where $B$ is a smooth projective curve. Set $\zeta = \frac{[\sigma]}2 + \frac{[\Sigma]}{2b}$. Then $\zeta$ has degree one on the fibers of $\pi$ and satisfies $\zeta^2 = 0$. Hence $\omega_{P/B} = -2\zeta$. By adjunction
    \begin{align*}
      \deg \omega_{\Sigma/B} &= (- 2\zeta + [\Sigma]) \cdot [\Sigma]
      = \frac{b-1}{b}\cdot \Sigma^2.
    \end{align*}
    Therefore, 
    \[ \deg D = \deg \omega_{\Sigma/B} = \frac{b-1}{b}\cdot \Sigma^2.\]
    On the other hand, we have
    \[ f^* T_{W_{0;b,1}} = \pi_*\sh Hom_\Sigma(I_\Sigma/I_\Sigma^2, O_\Sigma)  \oplus \pi_* \sh Hom_\sigma(I_\sigma/I_\sigma^2, O_\sigma).\]
    Observe that
    \begin{align*}
      \deg \pi_* \sh Hom_\Sigma(I_\Sigma/I_\Sigma^2, O_\Sigma) &= \deg \left(\sh Hom_\Sigma(I_\Sigma/I_\Sigma^2, O_\Sigma)\right) - \frac{\deg \omega_{\Sigma/B}}{2} \\
      &= \Sigma^2 - \frac{\deg D}{2} = \frac{b+1}{2(b-1)}\deg D.
    \end{align*}
    Using $\sigma^2 = -\Sigma^2/b^2 = -\deg D/(b(b-1))$, we get
    \begin{align*}
      -\deg K_W &= \deg f^* T_{W_{0;b,1}} \\
      &= \deg \pi_* \sh Hom_\Sigma(I_\Sigma/I_\Sigma^2, O_\Sigma) + \deg \pi_*\sh Hom_\sigma(I_\sigma/I_\sigma^2, O_\sigma) \\
      &= \frac{b+1}{2(b-1)}\deg D + \sigma^2
      = \frac{b+2}{2b} \cdot \deg D.
    \end{align*}
    Using \eqref{eqn:branch_rh} and \autoref{thm:pullback_disc}, we get
    \[ K_V = -\frac{b+2}{2b} \cdot (3T+\delta) + 2T.\]
    Substituting $T$ and $\delta$ from \autoref{thm:divisor_relations} and using $b = 2g+4$ yields the result.
  \end{proof}
  
  \section{The ample cones}\label{sec:flip_ample}
  In this section, we identify the cone of ample divisors on $\o{\sp T}_{g;1}^l$. We retain the notation, especially the list of divisors, introduced at the beginning of \autoref{sec:picard}. 

  Define the divisor $D_l$ by the formula
  \[ D_l = (4c_2 - c_1^2)+ \left(\frac{2l}{b}\right)^2c_1^2,\]
  where $b = 2g+4$, as usual. Recall that a $\Q$-Cartier divisor on a space is \emph{nef} if it intersects non-negatively with all complete curves in that space. The bulk of the section is devoted to proving that certain divisors on $\o{\sp T}_{g;1}^l$ are nef.
  \begin{restatable}{theorem}{thmflipnef}
    \label{thm:flip_nef}
    Let $0 < l < g$ and $l \equiv g \pmod 2$. A divisor is nef on $\o{\sp T}_{g;1}^l$ if and only if it is a non-negative linear combination of $D_l$ and $D_{l+2}$.
  \end{restatable}
  The proof is the content of \autoref{sec:positivity_generic} and \autoref{sec:positivity_crimp}. Using the Nakai--Moishezon criterion for ampleness, we then deduce projectivity.
    \begin{restatable}{theorem}{thmflipprojectivity}
      \label{thm:flip_projectivity}
      Let $0 < l < g$ and $l \equiv g \pmod 2$. A divisor is ample on $\o{\sp T}_{g;1}^l$ if and only if it is a positive linear combination of $D_l$ and $D_{l+2}$. In particular, the algebraic space $\o{\sp T}_{g;1}^l$ is a projective scheme.
    \end{restatable}
    The proof is the content of \autoref{sec:flip_projectivity}. In terms of the more standard generators $\lambda$ and $\delta$, the divisor $D_l$ admits the following expression:
    \[ \left(\frac{g-3}{2}\right)D_l = \left((7g+6)\lambda -g \delta\right) + \frac{l^2}{g+2}\cdot \left(9\lambda - \delta\right).\] 
    
  \subsection{Positivity for families with $\Br^2 \geq 0$}\label{sec:positivity_generic}
  We first prove that $D_l$ and $D_{l+2}$ are non-negative on one-parameter families in $\o{\sp T}_{g;1}^l$ with non-negative $\Br^2$. 
  \begin{proposition}\label{thm:generic_positivity}
    Let $B$ be a smooth projective curve and $\pi \from P \to B$ a $\P^1$ bundle with a section $\sigma \from B \to P$. Let $\phi \from C \to P$ be a triple cover, \'etale over $\sigma$. Assume that on the generic point of $B$, we have
    \[ \phi_*O_C/O_P \cong O_P(-m) \oplus O_P(-n),\]
    where $m \leq n$ are positive integers. Assume, furthermore, that $\br(\phi)^2 \geq 0$. Set $c_i = c_i(\phi_*O_C)$. Then
    \[\left(4c_2-c_1^2 + \left(\frac{n-m}{n+m}\right)^2c_1^2\right)[P] \geq 0.\]
  \end{proposition}
  \begin{corollary}\label{thm:generic_positivity_l}
    Assume that $0 < l < g$ and $l \equiv g \pmod 2$. Let $B$ be a projective curve and $f \from B \to \o{\sp T}_{g;1}^l$ a morphism such that $f^*(\Br^2) \geq 0$. Then \[f^*D_{l+2}\geq f^*D_{l} \geq 0.\]
  \end{corollary}
  \begin{proof}
    By replacing $B$ by a finite cover and normalizing, we may assume that we have a map $B \to \o{\orb T}_{g;1}^l$ and $B$ is smooth. Then \autoref{thm:generic_positivity} applies. Since $\Br^2 = 4c_1^2$, and $\Br^2 \geq 0$, we have $c_1^2 \geq 0$. Hence $f^*D_{l+2} \geq f^*D_l$. Putting $n+m=b/2$ in the conclusion of \autoref{thm:generic_positivity} and using $l \leq n-m$ (as the curves are $l$-balanced) we conclude that $f^*D_l \geq 0$.
  \end{proof}
  
  We give two proofs of \autoref{thm:generic_positivity}. We freely use the divisor relations from \autoref{thm:divisor_relations}.
  \subsubsection{First proof}
  The underlying tool in the first proof is the following result, coupled with an amusing ``balancing trick'' (\autoref{balancing_lemma}).
  \begin{proposition}\label{thm:bogomolov}
    Let $B$ be a smooth projective curve, $P \to B$ a $\P^1$ bundle, and $E$ a vector bundle of rank $r$ on $P$. If the restriction of $E$ to any fiber of $P \to B$ is balanced, then the class $\pi_*(2rc_2(E) - (r-1)c_1^2(E))$ on $B$ is non-negative.
  \end{proposition}
  Recall that $E$ is balanced if $E \cong O(d)^{\oplus r}$ for some $d$. The result is a special case of a result of Moriwaki \citep[Theorem~A]{moriwaki98:_relat_bogom}. Stankova-Frenkel  proves the particular case ($r=2$) that we need \citep[Proof of Proposition~12.2]{Stankova-Frenkel00:_Modul_Of_Trigon_Curves}. Nevertheless, here is another proof.
  \begin{proof}
     Observe that $2rc_2(E) - (r-1)c_1^2(E)$ is unaffected if $E$ is replaced by $E \otimes L$ for any line bundle $L$. Thus, possibly after replacing $B$ by a finite cover, assume that $\det E \cong O_P$. We must conclude that $\pi_*c_2(E) \geq 0$. 

     That $E$ is generically balanced and $\det E = O_P$ forces $E_b \cong O_{P_b}^{\oplus r}$ for a generic $b \in B$. Then $\pi_* E$ is a vector bundle of rank $r$ on $B$. Since $\pi^*\pi_* E \to E$ is generically an isomorphism and $\det E = O_P$, it follows that $c_1(\pi_* E) \leq 0$. Since $R^1\pi_*E$ is supported on finitely many points, we evidently have $c_1(R^1\pi_*E) \geq 0$. Therefore, $c_1(R\pi_* E) \leq 0$. But Grothendieck--Riemann--Roch shows that $c_1(R\pi_* E) = -\pi_*c_2(E)$.
  \end{proof}
  
  The following lemma lets us cook up a balanced vector bundle from a possibly unbalanced cover. We will use the lemma for $L = O_C$. It is formulated more generally because the proof is by induction.
  \begin{lemma}\label{balancing_lemma}
    Let $\phi \from C \to \P^1$ be a triple cover, $s \in \P^1$ a point over which $\phi$ is \'etale, $\{t_1,t_2,t_3\}$ an ordering of $\phi^{-1}(s)$ and $L$ a line bundle on $C$. Assume that
    \[ \phi_*L\cong O_{\P^1} \oplus O_{\P^1}(-m) \oplus O_{\P^1}(-n),\]
    for some positive integers $m \leq n$. Then there is an effective divisor $D$ of degree $n-m$ supported on $\{t_1,t_2,t_3\}$ such that
    \[ \phi_*(L(D)) \cong O_{\P^1} \oplus O_{\P^1}(-m) \oplus O_{\P^1}(-m).\]
  \end{lemma}
  \begin{proof}
    The proof is by induction on $n-m$. If $n = m$, take $D = 0$.

    Let $n-m \geq 1$. Set $O_C(1) =\phi^*O_{\P^1}(1)$. Consider the chain of inclusions 
    \[L \into L(t_1) \into L(t_1+t_2) \into L(t_1+t_2+t_3) = L \otimes O_C(1),\] 
    and the induced inclusions on global sections
    \[
    \begin{split}
      H^0(L \otimes O_C(m-1)) \stackrel{\iota_1}\into H^0(L(t_1) \otimes O_C(m-1))\stackrel{\iota_2}\into H^0(L (t_1+t_2) \otimes O_C(m-1)) \\
      \quad \stackrel{\iota_3}\into H^0(L(t_1+t_2+t_3) \otimes O_C(m-1)) = H^0(L \otimes O_C(m)).
    \end{split}
    \]
    From $\phi_*L = O_{\P^1} \oplus O_{\P^1}(-m) \oplus O_{\P^1}(-n)$, and $n > m$, we see that
    \[ h^0(L \otimes O_C(m)) - h^0(L \otimes O_C(m-1)) = 2.\]
    Therefore, at least one of the $\iota_j$ is an isomorphism. For such a $j$, the inclusion
    \[ H^0(L \otimes O_C(m-1)) \into H^0(L(t_j) \otimes O_C(m-1))\]
    must be an isomorphism. This isomorphism implies that
    \[ \phi_*L(t_j) \cong O_{\P^1} \oplus O_{\P^1}(-m) \oplus O_{\P^1}(-n+1).\]
    The induction step is thus complete.
  \end{proof}
  \begin{remark}\label{support_of_balancing_divisor}
    In \autoref{balancing_lemma}, the support of $D$ cannot be all of $\{t_1,t_2,t_3\}$. Otherwise, $\phi_*(L(D))$ will contain $\phi_*L \otimes O_{\P^1}(1)$ and hence also $O_{\P^1}(1)$, contradicting the conclusion of the lemma.
  \end{remark}

    \begin{proof}[First proof of \autoref{thm:generic_positivity}]
      After a  base change if necessary, assume that we have three distinct sections $\tau_i \from B \to C$ lying over $\sigma$, for $i=1,2,3$. For a generic $b \in B$, we have
      \[\phi_*O_{C_b} \cong O_{P_b} \oplus O_{P_b}(-m) \oplus O_{P_b}(-n).\]
      By \autoref{balancing_lemma}, there is an effective linear combination $D$ of the sections $\tau_i(B)$ of degree $n-m$ such that
      \[ \phi_*O_{C_b}(D) = O_{P_b} \oplus O_{P_b}(-m) \oplus O_{P_b}(-m).\]
      Consider the map $O_P \to \phi_*O_C(D)$ adjoint to the map $O_C \stackrel{D}\to O_C(D)$. The map of vector bundles $O_P \to \phi_*O_C(D)$ on $P$ is fiberwise nonzero, since $D$ does not contain any fiber of $C \to P$, by \autoref{support_of_balancing_divisor}. Define $V$ as the cokernel 
    \[ 0 \to O_P \to \phi_*O_C(D) \to V \to 0.\]
    Then $V$ is locally free of rank 2. The restriction of $V$ to $P_b$ is balanced. Hence, by \autoref{thm:bogomolov}, we have
    \begin{equation}\label{c2c1V}
      4c_2(V) - c_1^2(V) \geq 0.
    \end{equation}
    We compute the Chern classes of $V$ in terms of those of $\phi_*O_C$. Throughout, we abbreviate $c_i(\phi_*O_C)$ by $c_i$, and omit writing pullback or pushforward symbols where they are clear by context. Note that $c_1 \cdot \sigma = 0$ since $\phi$ is \'etale over $\sigma$.

    Say $D = a \tau_1 + b \tau_2$, where $a$ and $b$ are positive integers with $a+b = n-m$. We have the exact sequences
    \begin{align*}
      0 \to O_C \to&\ O_C(a\tau_1) \to O_{a\tau_1}(a\tau_1) \to 0, \text{ and } \\
      0 \to O_C(a\tau_1) \to&\ O_C(D) \to O_{b\tau_2} (D) \to 0.
    \end{align*}
    Here $a\tau_1$ denotes the subscheme of $C$ defined by the $a$th power of the ideal of $\tau_1$, and likewise for $b\tau_2$. Pushing forward along $\phi$, we get a relation in the Grothendieck group of $P$:
    \[ \phi_*O_C(D) = \phi_*O_C + \phi_*O_{a\tau_1}(a\tau_1) + \phi_*O_{b\tau_2}(D).\]
     Since $\phi$ is \'etale over $\sigma$, and $\tau_1$, $\tau_2$ are disjoint, we have $\phi_*O_{a\tau_1}(a\tau_1) \cong O_{a\sigma}(a\sigma)$ and $\phi_*O_{b\tau_2}(D) \cong O_{b\sigma}(b\sigma)$. Therefore, 
    \begin{align*}
      c(V) &= c(\phi_*O_C(D)) \\
      &= c(\phi_*O_C) \cdot (1+a\sigma) \cdot (1+b\sigma) \\
      &= 1 + (c_1 + (n-m) \sigma) + (ab\cdot\sigma^2 + c_2).
    \end{align*}
    Using \eqref{c2c1V}, we obtain
    \begin{equation}
      4c_2 + ab\cdot\sigma^2 - c_1^2 - (n-m)^2\sigma^2 \geq 0.
    \end{equation}
    Since $ab \geq 0$ and $\sigma^2 = \Br^2/b^2 \geq 0$, we conclude that
    \[4c_2 - c_1^2 + \left(\frac{n-m}{n+m}\right)^2c_1^2 \geq 0.\]
  \end{proof}

  \subsubsection{Second proof}\label{sec:generic_positivity_second_proof}
  The second proof resembles the proof of positivity in the $\Br^2 < 0$ case. Although it is somewhat less elegant, it is more transparent. The underlying idea is to use, in some form, the morphism to $\P^1$ given by the cross-ratio.
  \begin{proof}[Second proof of \autoref{thm:generic_positivity}]
    Assume, possibly after a finite base change, that we have three disjoint sections $\tau_i \from B \to C$ lying over $\sigma$, for $i = 1,2,3$. For a generic $b$, we have
    \[\phi_*O_{C_b} \cong O_{P_b} \oplus O_{P_b}(-m) \oplus O_{P_b}(-n).\]
    If $m = n$, then by \autoref{thm:bogomolov}, we get $4c_2 - c_1^2 \geq 0$. Since $\Br^2 = 4c_1^2\geq 0$, this implies the desired result. Henceforth, assume that $m < n$.
    
    Set $F = \phi_*O_C/O_P$, and consider the map
    \[ \chi \from \pi_*(F \otimes O_P(m\sigma)) \to \pi_*(F|_{\sigma} \otimes O_P(m\sigma)).\]
    Note that $\pi_*(F \otimes O_P(m\sigma))$ is a line bundle on $B$. Since $\phi^{-1}\sigma = \tau_1 \sqcup \tau_2 \sqcup \tau_3$, the bundle $F|_\sigma$ is trivial. Clearly, over the points $b \in B$ where $F \cong O(-m) \oplus O(-n)$, the map $\chi$ is injective. Hence, there is a map $p \from O_\sigma^{\oplus 2} = F|_{\sigma} \to O_\sigma$ such that the induced map
    \[p \circ \chi \from \pi_*(F \otimes O_P(m\sigma)) \to \pi_*(O_\sigma(m\sigma))\]
    is an isomorphism at the generic point of $B$.

    Denote by $(n-m)\sigma$ the scheme defined by the $(n-m)$th power of the ideal of $\sigma$. Consider the diagram of vector bundles of rank $(n-m-1)$ on $B$:
    \begin{equation}\label{eqn:square}
    \begin{tikzpicture}
      \matrix (m) [matrix of math nodes, column sep = 1.5em, row sep = 2em]{
        \pi_*(F \otimes O_P(m)) \otimes \pi_*O_P((n-m-1)\sigma)&
        \pi_*O_{\sigma}(m\sigma) \otimes \pi_*O_P((n-m-1)\sigma)\\
        \pi_*(F \otimes O_P((n-1)\sigma)) & \pi_*O_{(n-m)\sigma}((n-1)\sigma)\\
    };
    \path [->] 
    (m-1-1) edge (m-1-2)
    (m-1-1) edge (m-2-1)
    (m-1-2) edge (m-2-2)
    (m-2-1) edge (m-2-2);
    \end{tikzpicture}.
  \end{equation}
  The top and the left maps are clear. The bottom one is the composition
    \[ \pi_*(F \otimes O_P((n-1)\sigma)) \to \pi_*(F|_{(n-m)\sigma}\otimes O_P((n-1)\sigma) \stackrel{p}\to \pi_*(O_{(n-m)\sigma}((n-1)\sigma)).\]
    The one on the right is induced by the map of rings
    \[ O_\sigma \to O_{(n-m)\sigma}\]
    dual to the projection $(n-m)\sigma \to \sigma$.  The map on the right is an isomorphism; the rest are isomorphisms generically on $B$. In particular, we conclude that
    \begin{equation}\label{eqn:generic_degree_inequality}
      \deg\pi_*O_{(n-m)\sigma}((n-1)\sigma)
      \geq
      \deg \pi_*(F \otimes O_P((n-1)\sigma)) .
    \end{equation}
    We compute both sides in terms of $c_i(\phi_*O_C) = c_i(F)$, henceforth abbreviated by $c_i$. The left side is easy from the exact sequences
    \[ 0 \to O_\sigma((m+i)\sigma) \to O_{(n-m-i)\sigma}((n-1)\sigma) \to O_{(n-m-i-1)\sigma}((n-1)\sigma) \to 0,\]
    for $0 \leq i \leq n-m-1$. So we have
    \begin{align}\label{left_side}
      \begin{split}
        \deg\pi_*O_{(n-m)\sigma}((n-1)\sigma) &= \left((n-1)\sigma + (n-2)\sigma + \dots + m\sigma\right) \cdot \sigma \\
        &= \left(\frac{(n+m-1)(n-m)}{2}\right)\sigma^2 \\
        &= - \left(\frac{(n+m-1)(n-m)}{2(n+m)^2}\right)c_1^2.
      \end{split}
    \end{align}
    For the right side, apply Grothendieck--Riemann--Roch, keeping in mind
    \begin{align*}
      \omega_\pi \cdot \sigma = -\sigma^2, \quad \sigma^2 = - \frac{c_1^2}{2(n+m)},\quad \omega_\pi \cdot c_1 = - \frac{c_1^2}{(n+m)},  \text{ and } c_1\cdot \sigma = \omega_\pi^2 = 0.
    \end{align*}
    The result is
    \begin{align}\label{right_side}
      \begin{split}
      \ch&(R\pi_*(F \otimes O_P((n-1) \sigma))) \\
      &= \pi_*\left(\ch F \cdot \ch O_P((n-1)\sigma) \cdot \td_{P/B}\right)\\
      &= \pi_*\left(\left(2+c_1+\frac{c_1^2}{2}-c2\right)\cdot\left(1+(n-1)\sigma+\frac{(n-1)^2\sigma^2}{2}\right) \cdot \left(1 - \frac{\omega_\pi}{2}\right)\right),\\
      &\text{so that}\\
      c_1&(R\pi_*(F \otimes O_P((n-1)\sigma))) \\
      &= \left(\frac{(n+m)^2 + (n+m) - 2(n-1)^2 - 2(n-1)}{2(n+m)^2}\right)c_1^2 - c_2 \\
      &= -\left(\frac{(n-m)(n+m-1)-2mn}{2(n+m)^2}\right)c_1^2 - c_2.
    \end{split}
  \end{align}
  See that $R^1\pi_*(F \otimes O_P((n-1)\sigma))$ is supported on finitely many points of $B$. Hence
    \[ c_1(\pi_*(F \otimes O_P((n-1)\sigma))) \geq c_1(R\pi_*(F \otimes O_P((n-1)\sigma))).\]
    Combining with \eqref{eqn:generic_degree_inequality}, we arrive at
    \[ c_1 (\pi_*O_P((n-1)\sigma)|_{(n-m)\sigma}) \geq c_1(R\pi_*(F \otimes O_P((n-1)\sigma))). \]
    Substituting the left side from \eqref{left_side} and the right side from \eqref{right_side}, we get
    \begin{align*}
      &- \left(\frac{(n+m-1)(n-m)}{2(n+m)^2}\right)c_1^2 
      \geq -\left(\frac{(n-m)(n+m-1)-2mn}{2(n+m)^2}\right)c_1^2 - c_2 \\
      &\implies 
      4c_2 - \left(\frac{4mn}{(n+m)^2}\right)c_1^2 =
      (4c_2 - c_1^2) + \left(\frac{n-m}{n+m}\right)^2c_1^2 \geq 0. \\
    \end{align*}
    The proof is thus complete.    
  \end{proof}
  \begin{remark}\label{rem:generic_zero}
    Let us examine the proof to determine when equality holds. This is the case if and only if the map $\pi_*(F \otimes O_P((n-1)\sigma)) \to \pi_*O_{(n-m)\sigma}((n-1)\sigma)$ is an isomorphism and $R^1\pi_*(F \otimes O_P((n-1)\sigma)) = 0$. Then the splitting type of $F_b$ is $(m,n)$ for all $b \in B$. Therefore, all the maps in \eqref{eqn:square} are isomorphisms. In particular, we have an isomorphism $\pi_*(F \otimes O_P(m)) \cong O_\sigma(m\sigma)$. Hence the map 
    \[ \pi_*(F \otimes O_P(m)) \to \pi_*(F|_\sigma \otimes O_P(m)),\]
    which defines the cross-ratio, is equivalent to a (nonzero) global section of 
    \[ \sh Hom_B(\pi_*(F \otimes O_P(m)), \pi_*(F|_\sigma \otimes O_P(m))) \cong O_B^{\oplus 2}.\]
    The upshot is that $C \to P$ is a family of triple covers with a constant Maroni invariant $l = n - m$ and a constant cross-ratio.

    Retracing the steps, it is easy to see that for such a family $C \to P$ with a constant Maroni invariant $l = n - m$ and a constant cross-ratio, equality holds; that is, the pullback of $D_l$ is zero.
  \end{remark}

  \subsection{Positivity for families with $\Br^2 < 0$}\label{sec:positivity_crimp}
  Having taken care of families with $\Br^2 \geq 0$, we now consider families with $\Br^2 < 0$. 
  \begin{proposition}\label{thm:crimp_positivity}
    Let $B$ be a smooth projective curve and $\pi \from P \to B$ a $\P^1$ bundle with a section $\sigma \from B \to P$. Assume that $\Br^2 < 0$ and let $\zeta$ be the unique section of $\pi$ of negative self-intersection. Let $\phi \from C \to P$ be a triple cover \'etale away from $\zeta$. Assume that the splitting type of the singularity of $C \to P$ over $\zeta$ is $(m,n)$ over a generic point of $B$, where $m < n$ are positive integers. Then
    \[\left(4c_2-c_1^2 + \left(\frac{n-m}{n+m}\right)^2c_1^2\right)[P] \geq 0.\]
  \end{proposition}
  \begin{corollary}\label{thm:crimp_positivity_l}
    Assume that $0 < l < g$ and $l \equiv g \pmod 2$. Let $B$ be a projective curve and  $f \from B \to \o{\sp T}_{g;1}^l$ a morphism such that $f^*\Br^2 < 0$. Then
    \[ f^*D_{l} \geq f^*D_{l+2} \geq 0.\]  
  \end{corollary}
  \begin{proof}
    Since $4c_1^2 = \Br^2$ and $\Br^2 < 0$, we have $c_1^2 < 0$. Therefore $f^*D_l \geq f^*D_{l+2}$. Hence, it suffices to prove that $f^*D_{l+2} \geq 0$.
    
    By replacing $B$ by a finite cover and normalizing if necessary, we may assume that $f$ lifts to a map $f \from B \to \o{\orb T}_{g;1}^l$ and $B$ is smooth. Say $f$ is given by $(P; \sigma; \phi \from C \to P)$. Since $4c_1^2 = \Br^2 < 0$, the branch divisor of $\phi$ must be supported on the section $\zeta$ of $P \to B$ of negative self-intersection. Thus, \autoref{thm:crimp_positivity} applies. Since our family consists of $l$-balanced covers, we have $n-m > l$, and hence $n - m \geq l+2$ because $n-m \equiv l \equiv g \pmod 2$. Using $n+m=b/2$, $c_1^2 < 0$ and $n-m \geq l+2$ in the conclusion of \autoref{thm:crimp_positivity}, we conclude that
    \[ f^*D_{l+2} \geq 0.\]
  \end{proof}
  
  \begin{proof}[Proof of \autoref{thm:crimp_positivity}]
    By making a base change if necessary, assume that we have three sections $\tau_i \from B \to C$ over $\sigma$, for $i=1,2,3$. Let $\tw C \to C$ be the normalization and $\tw\phi \from \tw C \to P$ the corresponding map. By \cite{teissier80:_resol_i}, the fibers are $\tw C \to B$ are normalizations of the fibers of $C \to B$. Therefore, $\tw C$ is the disjoint union of three copies of $P$, each containing one section $\tau_i$. Set $F = \phi_*O_C/O_P$, and $E = \tw\phi_*O_{\tw C}/O_P$. Note that $E \cong O_{P}^{\oplus 2}$. The inclusion $\phi_*O_C \into \tw\phi_*{O_{\tw C}}$ induces an inclusion $F \into E$, which is an isomorphism away from $\zeta$. We think of $F$ as a subsheaf of $E$ via this inclusion.
    
    Since the generic spitting type of the singularity of the fibers of $C \to B$ over $\zeta$ is $(m,n)$, we have the inclusions
    \[ I_\zeta^n \cdot E \subset F \subset I_\zeta^m \cdot E.\]
    Let $\overline F = F/\left(I_\zeta^n \cdot E\right)$ and $\overline E = \left(I_\zeta^m/I_\zeta^n\right) \cdot E$. Both $\overline F$ and $\overline E$ are supported on $\zeta$, are $\pi$-flat, and their $\pi$-fibers have lengths $(n-m)$ and $2(n-m)$ respectively. Pushing forward $\overline F \to \overline E$, we get 
    \[ i \from \pi_*\overline F \to \pi_* \overline E,\]
    a map of locally free sheaves on $B$ of rank $(n-m)$ and $2(n-m)$, respectively. The target $\pi_*\overline E$ is isomorphic to $\pi_*(I_\zeta^m/I_\zeta^n) \otimes (O_B^{\oplus 2})$.

    We examine $i$ explicitly over a point $b \in B$ where the splitting type of the singularity is $(m,n)$. Let $x$ be a local coordinate for $P_b$ near $\zeta(b)$. Since the singularity of $C_b \to P_b$ is of type $(m,n)$, the subalgebra $O_{C_b} \subset O_{\tw C_b} = O_{P}^{\oplus 3}$ is generated as an $O_{\P^1_b}$ module, locally around $x$, by $\langle 1, x^mf, x^nO_{P_b}^{\oplus 3}\rangle$, where the image of $f$ in $E_b$ is nonzero modulo $x$. Therefore,
    \[ \overline F_b = k\langle x^mf, x^{m+1}f, \dots, x^{n-1}f \rangle.\]
    Since the image of $f$ in $E|_{\zeta(b)}$ is nonzero, it is nonzero in one of the projections $p \from O_{\zeta(b)}^{\oplus 2} = E|_{\zeta(b)} \to O_{\zeta(b)}$. It follows that the composite $j_b = p \circ i_b$ gives an isomorphism
    \[ j_b \from \overline F_b \isom I_{\zeta(b)}^m/I_{\zeta(b)}^n.\]
    Consequently, the composition $j = p \circ i$ is an isomorphism on the generic fiber:
    \[ j \from \pi_*\overline F \to \pi_*(I_\zeta^m/I_\zeta^n).\]
    We conclude that
    \begin{equation}\label{eqn:crimp_degree_inequality}
      \deg \pi_*(I_\zeta^m/I_\zeta^n)
      \geq
     \deg \pi_*\overline F.
    \end{equation}
    We compute both sides in terms of $c_i(\phi_*O_C)$, abbreviated henceforth by $c_i$. Since $P\to B$ is a $\P^1$ bundle with two disjoint sections $\sigma$ and $\zeta$ of positive and negative self-intersection, respectively, we have 
    \[\omega_{P/B} = -\sigma - \zeta.\]
    By Grothendieck--Riemann--Roch,
    \begin{align*}
      \ch(\pi_*(I_\zeta^m/I_\zeta^n)) &= \pi_*(\ch (I_\zeta^m/I_\zeta^n) \cdot \td_{P/B})\\
      &= \pi_*\left((\ch I_\zeta^m - \ch I_\zeta^n) \cdot \td_{P/B} \right)\\
      &= \pi_*\left(\left((n-m)\zeta+\frac{(m^2-n^2)\zeta^2}{2} \right)\cdot\left(1+\frac{\sigma+\zeta}{2}\right)\right);\\
      c_1(\pi_*\overline E) &= \left(\frac{m^2-n^2+n-m}{2(n+m)^2}\right)c_1^2.
    \end{align*}
    The last equality uses $\zeta^2 = -\sigma^2 = c_1^2/(n+m)^2$.
    
    Similarly, using $c_1 = -(m+n) \zeta$ and Grothendieck--Riemann--Roch,
    \begin{align*}
      \ch(\pi_*\overline F) &= \pi_*(\ch \overline F \cdot \td_{P/B}) \\
      &= \pi_*((\ch F - 2\ch I_\zeta^n) \cdot \td_{P/B}) \\
      &= \pi_*\left( \left(c_1+2n\zeta+\frac{c_1^2}{2}-c_2-n^2\zeta^2\right) \cdot \left(1+ \frac{\zeta+\sigma}{2}\right) \right);\\
      c_1(\pi_*\overline F) &= \left(\frac{m^2-n^2+2mn+n-m}{2(m+n)^2}\right)c_1^2 - c_2.
    \end{align*}
    Substituting into \eqref{eqn:crimp_degree_inequality}, we get
    \begin{align*}
      4c_2 - \left(\frac{4mn}{(m+n)^2}\right)c_1^2 = 4c_2 -c_1 + \left(\frac{n-m}{m+n}\right)^2c_1^2 \geq 0.
      \end{align*}
    \end{proof}
\begin{remark}\label{rem:crimp_zero}
  Let us examine the proof to determine when equality holds. This is the case if and only if the map $\o F \to I^m_\zeta/I^n_\zeta$ is an isomorphism. Then fiber $C_b \to P_b$ has a singularity of splitting type $(m,n)$ for all $b \in B$. Furthermore, the map $\o F|_\zeta \to \o E|_\zeta$, which defines the principal part, is equivalent to a (nonzero) global section of 
  \[ \sh Hom (\o F|_\zeta, \o E|_\zeta) \cong O_\zeta^{\oplus 2}.\]
  The upshot is that the family $C \to P$ has singularities of a constant $\mu$ invariant $l = n -m $ and a constant principal part.

  Retracing the above steps, it is easy to see that for such a family $C \to P$ of covers with concentrated branching with singularities of a constant $\mu$ invariant $l = n - m$ and a constant principal part, equality holds; that is, the pullback of $D_l$ is zero.
\end{remark}

    We have essentially finished the proof of \autoref{thm:flip_nef}; it is now a matter of collecting the pieces. We recall the statement for the convenience of the reader.
    \thmflipnef*
    \begin{proof}
      By \autoref{thm:generic_positivity_l} and \autoref{thm:crimp_positivity_l} we conclude that $D_l$ and $D_{l+2}$ are non-negative on any complete curve in $\o{\sp T}_{g;1}^l$. Hence, every non-negative linear combination of $D_l$ and $D_{l+2}$ is nef.

      It remains to show that $D_l$ and $D_{l+2}$ are indeed the edges of the nef cone. For that, it suffices exhibit curves on which $D_l$ is zero and $D_{l+2}$ is positive, and vice-versa. \autoref{rem:generic_zero} and \autoref{rem:crimp_zero} tell us how to construct such curves. We explain the constructions briefly. Let $m < n$ be such that $n+m = g+2$ and $n-m = l$.

      For the edge $\langle D_l \rangle$, we construct a family of covers with constant Maroni invariant $l$ and constant cross-ratio. Such a family can be constructed, for example, as follows. Let $C \to \P^2$ be a connected, generically \'etale triple cover with 
      \[O_C/O_{\P^2} \cong O_{\P^2}(-m) \oplus O_{\P^2}(-n).\]
      Let $p \in \P^2$ be a point over which $C \to \P^2$ is \'etale. Let $P = \Bl_p\P^2$ and set $C' = C \times_{\P^2} \Bl_p\P^2$. Then $P \to \P^1$ is a $\P^1$ bundle with a section $\sigma$ given by the exceptional divisor. The family $(P; \sigma; C' \to P)$ gives a curve in $\o{\sp T}_{g;1}^l$. The pullback of $D_l$ to this curve is zero and the pullback of $D_{l+2}$ is positive.
      
      For the edge $\langle D_{l+2}\rangle$, we construct a family of covers with concentrated branching having $\mu$ invariant $l+2$ and constant principal part. To construct such a family, consider the two-parameter family of sub-algebras $S(a,b)$ of $O_{\P^1}^{\oplus 3}$, where $S(a,b)$ is generated locally around $0$ as an $O_{\P^1}$ module by
      \[ 1, (x^{m-1} + ax^{n-1} + bx^n, 0, 0) ,\text{ and } x^{n+1}O_{\P^1}^{\oplus 3}.\]
      Via the scaling $x \mapsto tx$, we have an isomorphism \[S(a,b) \isom S(t^{n-m}a, t^{n-m+1}b).\]
      The resulting family on $(k^{\oplus 2} \setminus 0)/\G_m$ gives a curve in $\o{\sp T}_{g;1}^l$. The pullback of $D_{l+2}$ to this curve is zero and the pullback of $D_l$ is positive.
    \end{proof}
    
    \subsection{Projectivity}\label{sec:flip_projectivity}
    In this section, we prove that divisors in the interior of the nef cone of $\o{\sp T}_{g;1}^l$ are indeed ample. This would follow from Kleiman's criterion if we knew that $\o{\sp T}_{g;1}^l$ is a scheme. However, it is a priori only an algebraic space. Kleiman's criterion can fail for algebraic spaces, as pointed out by Koll\'ar \citep[\S~VI, Exercise~2.19.3]{kollar96:_ration}. However, Nakai--Moishezon's criterion still holds.
\begin{theorem}[(Nakai--Moishezon criterion for ampleness)]\citep[Theorem~3.11]{kollar90:_projec}
  Let $X$ be an algebraic space proper over an algebraically closed field and $H$ a Cartier divisor on $X$. Then $H$ is ample if and only if for every irreducible closed subspace $Y \subset X$ of dimension $n$, the number $H^n \cdot Y$ is positive.
\end{theorem}

To deduce that divisors in the interior of the nef cone are ample, we need a mild extension of \citep[Lemma~4.12]{fedorchuk11:_ample}
\begin{lemma}\label{thm:perturbation}
  Let $X$ be an algebraic space proper over an algebraically closed field. Suppose $X$ satisfies the following: for every irreducible subspace $Y \subset X$, there is a finite surjective map $Z \to Y$ and a Cartier divisor $D$ on $X$ whose pullback to $Z$ is numerically equivalent to a nonzero and effective divisor. Then any Cartier divisor in the interior of the nef cone of $X$ is ample.
\end{lemma}
\begin{proof}
  The proof follows the proof of \citep[Lemma~4.12]{fedorchuk11:_ample} almost verbatim. In what follows, ``divisor'' means a $\Q$-Cartier divisor. Let $H$ be a divisor in the interior of the nef cone. By the Nakai--Moishezon criterion, it suffices to prove that for every $n$-dimensional closed subspace $Y \subset X$, the number $H^n \cdot Y$ is positive. We induct on $n$; the case $n = 0$ is trivial.

Let $Z \to Y$ be as in the hypothesis and denote by $f$ the finite map $Z \to X$. Say
\[ D \cdot [Z] \equiv \sum_i a_i [Z_i],\]
where $a_i > 0$ and $Z_i \subset Z$ are reduced and irreducible divisors. Let $f_i \from Z_i \to X$ be the restriction of $f$. Then $f_i$ is also a finite map.

  Since $H$ is in the interior of the nef cone, for a sufficiently small $\epsilon > 0$, the divisor $H - \epsilon D$ is nef. Since $H$ and $H-\epsilon D$ are nef, we have
  \[ H^{n-1}(H-\epsilon D) \cdot [Z] \geq 0.\]
  Therefore,
  \begin{align*}
    (\deg f) \cdot H^n \cdot [Y] = H^n \cdot [Z] &\geq \epsilon H^{n-1} D \cdot [Z] \\
    &= \epsilon \sum_i a_i H^{n-1} \cdot [Z_i] \\
    &= \epsilon \sum_i a_i (\deg f_i) \cdot H^{n-1}[f_i(Z_i)] > 0,
  \end{align*}
  where the last inequality is by the induction hypothesis. The induction step is complete.
\end{proof}

We now have the tools to prove \autoref{thm:flip_projectivity}. We recall the statement for the convenience of the reader.
\thmflipprojectivity*
\begin{proof}
  We check that $X = \o{\sp T}^l_{g;1}$ satisfies the hypothesis of \autoref{thm:perturbation}. Let $Y \subset X$ be an irreducible closed subspace. Choose a finite surjective map $Z \to Y$ such that $Z$ is a normal scheme and $Z \to X$ lifts to $Z \to \o{\orb T}_{g;1}^l$, given by a family $(\pi \from P \to Z; \sigma; \phi \from C \to P)$. Set $\Sigma = \br(\phi) \subset P$. Then $\Sigma$ is a $\pi$-flat divisor of degree $b$, disjoint from $\sigma$. By passing to a finite cover of $Z$ if necessary, assume that we have three disjoint sections $\tau_1, \tau_2, \tau_3 \from Z \to C$ over $\sigma$ and sections $\sigma_1, \dots, \sigma_b \from Z \to P$ such that
  \[ \Sigma = \sigma_1(Z) + \dots + \sigma_b(Z)\]
  as divisors on $P$.  Observe that all the $\sigma_i$ are disjoint from $\sigma$, and hence are linearly equivalent to each other. Furthermore, $\pi_*[\sigma_i^2] = -\pi_*[\sigma^2]$.

  We exhibit a divisor class on $X$ whose pullback to $Z$ is nonzero and effective. 
\begin{asparadesc}
\item[Case 1: $\sigma_i$ are not all coincident.] Without loss of generality, $\sigma_1 \neq \sigma_2$ at a generic point of $Z$. If $\sigma_1(z) \neq \sigma_2(z)$ for all $z \in Z$, then we have three disjoint sections $\sigma, \sigma_1$ and $\sigma_2$ of the $\P^1$ bundle $P \to Z$. Hence $P \to Z$ is trivial and $\Sigma \subset P$ is a constant family. In other words, the map $Z \to \o{\orb T}_{g;1}^l$ lies in a geometric fiber of $\br \from \st H^3 \to \st M$. By \citep[\autoref*{p1:thm:lambda_ample_on_fibers}]{deopurkar12:_compac_hurwit}, the pullback of $-\lambda$ is ample on $Z$. In particular, some multiple of $-\lambda$ pulls back to a nonzero and effective divisor.
  
  If $\sigma_1(z) = \sigma_2(z)$ for some $z \in Z$, then $\pi_*(\sigma_1 \cdot \sigma_2)$ is a nonzero and effective divisor on $Z$. Since $\sigma_1 \sim \sigma_2 \sim -\sigma$, the divisor $\pi_*(\sigma_1 \cdot \sigma_2)$ is equivalent to the pullback of the divisor $-\sigma^2$ on $\o{\sp T}_{g;1}^l$.

\item[Case 2: $\sigma_i$ are all coincident.] Say $\sigma_i = \zeta$ for $i=1,\dots, b$. In this case, we have a family of covers with concentrated branching. We begin as in the proof of \autoref{thm:crimp_positivity}.

 Let the splitting type of the singularity over a generic $z \in Z$ be $(m,n)$, with $m < n$. Let $\tw C \to C$ be the normalization. By \citep{teissier80:_resol_i}, the fibers of $\tw C \to Z$ are the normalizations of the corresponding fibers of $C \to Z$. In particular, $\tw C \cong P \sqcup P \sqcup P$. Set 
  \[ F = \phi_*O_C/ O_P \text{ and } E = \tw\phi_*O_{\tw C}/ O_P \cong O_P^{\oplus 2}.\]
  We have inclusions
  \[ I_\zeta^n E \subset F \subset I_\zeta^m E.\]
  Set
  \[ \o F = F/I^n_\zeta E \text{ and } \o E = I^m_\zeta E/I^n_\zeta E.\]
  Then we have an induced map $\o F \to \o E$. Also, see that $\o E \cong (I^m_\zeta/I^n_\zeta)^{\oplus 2}$. Since the generic splitting type of the singularity is $(m, n)$, there is a projection $E \to I^m_\zeta/I^n_\zeta$ such that $\o F \to I^m_\zeta/I^n_\zeta$ is an isomorphism over the generic point of $Z$. Suppose $\o F \to I^m_\zeta/I^n_\zeta$ is not an isomorphism over all of $Z$. Then we get a map of line bundles on $Z$:
  \[ \det\pi_*\o F \to \det \pi_* (I^m_\zeta/I^n_\zeta)\]
  whose vanishing locus is a nonzero effective divisor. The class of this vanishing locus can be expressed as a pullback of $c_1^2$ and $c_2$; in fact, it is precisely $D_{n-m}$, as computed in the proof of \autoref{thm:crimp_positivity}.

  We are thus left with the case where $\o F \to I_\zeta^m/I_\zeta^n$ is an isomorphism. In this case, $\o F$ and $\o E$ are vector bundles on $(n-m)\zeta$ of rank one and two respectively. The map $\o F \to \o E$ is just a global section $f$ of $\o V = \sh Hom_{(n-m)\zeta}(\o F, \o E) \cong O_{(n-m)\zeta}^{\oplus 2}$. The restriction $f|_\zeta$ is a section of $O_\zeta^{\oplus 2}$, and hence it is constant. We are thus dealing with a family of covers with concentrated branching, a constant $\mu$ invariant and a constant principal part.

Set $G = \Isom_Z((P, \zeta, \sigma), (\P^1, 0, \infty))$. Then $G \to Z$ is a $\G_m$ torsor. We have a canonical isomorphism 
  \[ (P_G, \zeta_G, \sigma_G) \isom (\P^1 \times G, 0 \times G, \infty \times G).\]
  Let $x$ be a uniformizer of $\P^1$ around $0$. Then $\o V_G \cong {(k[x]/x^{n-m})}^{\oplus 2} \otimes_k O_G$. We can interpret the global section $f_G$ of $\o V_G$ as a $\G_m$ equivariant map $\phi_G \from G \to (k[x]/x^{n-m})^{\oplus 2}$, where $\G_m$ acts on the $k$ vector space $(k[x]/x^{n-m})^{\oplus 2}$ by $t \from x^i \mapsto t^ix^i$. Since the restriction of $f_G$ to $\zeta$ is constant, $\phi_G$ has the form $\phi_G = (c+\psi_G)$, where $c \in k^{\oplus 2}$ is a constant and $\psi_G \from G \to (xk[x]/x^{n-m})^{\oplus 2}$. Explicitly, over a point $b \in G$, the subalgebra $O_{C_b} \subset O_{\tw C_b}$ is generated as an $O_{P_b}$ module, locally around $0$, by 
  \[ 1, x^n O_{\tw C_b} \text{ and } x^m (c + \psi_G(b)).\]
  Since the Maroni invariant of the resulting cover is less than $n-m$, we conclude that $\psi_G(b) \neq 0$ for any $b \in G$. Furthermore, since the map $Z \to \o{\sp T}_{g;1}^l$ is quasi-finite, so is the map $\psi_G \from G \to (xk[x]/x^{n-m})^{\oplus 2}$. We thus have a finite map
  \[ \psi \from Z \to [(xk[x]/x^{n-m})^{\oplus 2} \setminus 0 / \G_m],\]
  where the right side is a weighted projective stack. We conclude that $\psi^*O(1)$ is ample. But $\psi^*O(-1)$ is the line bundle associated to $G \to Z$, which is $P \setminus \sigma \to Z$. Therefore,
  \begin{align*}
    c_1(\psi^* O(1)) &= c_1(\zeta^* O_P(-\zeta)) \\
    &= \pi_*(-\zeta^2) = \pi_*(\sigma^2).
  \end{align*}
  In particular, the pullback to $Z$ of some multiple of the divisor $\sigma^2$ on $\o{\sp T}_{g;1}^l$ is effective.
\end{asparadesc}
\end{proof}

As a result of the positivity of the interior of the cone spanned by $D_l$ and $D_{l+2}$, we can easily deduce the following.
\begin{proposition}\label{thm:flip_K}
  Let $0 < l < g$ and $l \equiv g \pmod 2$. Consider the map $\beta_l \from \o{\sp T}_{g;1}^l \dashrightarrow \o{\sp T}_{g;1}^{l-2}$. Then $\Exc(\beta_l)$ is covered by $K$-negative curves. If $l > 0$, then $\Exc(\beta_l^{-1})$ is covered by $K$-positive curves.
\end{proposition}
\begin{proof}
  $\Exc(\beta_l)$ is the locus of $(P; \sigma; \phi \from C \to P)$ where $\phi$ has Maroni invariant $l$. This locus is covered by curves $S$ in which the cross-ratio is constant. Similarly $\Exc(\beta_l^{-1})$ is the locus of $(P; \sigma; \phi \from C \to P)$ where $\phi$ has concentrated branching and $\mu$ invariant $l+2$. For $l > 0$, this locus is covered by curves $T$ in which the principal part is constant. By \autoref{rem:generic_zero} (resp. \autoref{rem:crimp_zero}), the divisor $D_l$ (resp. $D_{l+2}$) is zero on such $S$ (resp. $T$). Since $D_{l+2}$ and $K$ are on the opposite sides of the line spanned by $D_{l}$ in $\Pic_\Q(\o T_{g;1}^l)$, the claim follows.
\end{proof}

\section{The final model}\label{sec:flip_final}
In this section, we prove that for even $g$, the final model $\o{\sp T}_{g;1}^0$ is Fano and for odd $g$, the final model $\o{\sp T}_{g;1}^1$ is a Fano fibration over $\P^1$.

\subsection{The case of even $g$}\label{sec:final_even}
Let $g = 2(h-1)$, where $h \geq 1$. Fix an identification $\P^1 = \proj k[X,Y]$. Set $0 = [0:1]$ and $\infty = [1:0]$. Let $G = \Aut(\P^1, \infty)$; this is the group of affine linear transformations 
\[\mu_{\alpha,\beta} \from (X,Y) \mapsto (X+\beta Y,\alpha Y),\]
where $\alpha \in k^*$ and $\beta \in k$. Let $\Lambda$ be the two dimensional $k$ vector space $\Lambda = \ker (\tr \from k^{\oplus 3} \to k)$, where $\tr$ is the sum of the three coordinates. $\Lambda$ should be thought of as the space of traceless functions on $\{1,2,3\} \times \spec k$. Set $\Gamma = \Lambda \otimes_k O_{\P^1}(h)$ and
\begin{align*}
  V &= H^0\left(\Sym^3\Gamma \otimes_{\P^1} \det \dual \Gamma \right)\\
  &= (\Sym^3(\Lambda) \otimes_k \det \dual \Lambda) \otimes_k H^0(O_{\P^1}(h)).
\end{align*}
By the structure theorem of triple covers (\autoref{thm:structure}), balanced triple covers of $\P^1$ of arithmetic genus $g$ correspond precisely to elements of $V$. Note that $V$ admits a natural action of $\Gl(\Lambda) \times G$. Indeed, $\Gl(\Lambda) = \Gl_2$ acts naturally on the first factor in $V$, whereas $G$ acts on the second factor by
\[ \mu_{\alpha,\beta} \from p(X,Y) \mapsto p\circ\mu_{\alpha,\beta}^{-1}(X,Y).\]

Let $v_1, v_2 \in V$ be two points and $C_1 \to \P^1$ and $C_2 \to \P^1$ the corresponding balanced triple covers. Using the point $\infty \in \P^1$ as the additional marked point, treat them as marked triple covers $(C_i \to \P^1; \infty)$. We observe that these two marked covers are isomorphic if and only if $v_1$ and $v_2$ are related by the action of $\Gl_2 \times G$. Thus, we might expect $\o{\orb T}_{g;1}^0$ to be the quotient $[V/\Gl_2 \times G]$. This is not quite true since not all elements of $V$ give an element of $\o{\orb T}_{g;1}^0$. Firstly, the cover must be \'etale over $\infty$ and secondly, it must not have $\mu$ invariant $0$. Thus, we expect
\[ \o{\orb T}_{g;1}^0 = [U / \Gl_2\times G],\]
for a suitable open $U \subset V$. In what follows, we prove that this is indeed the case. Along the way, we also simplify the presentation $[U / \Gl_2\times G]$.

The first step is to exhibit a morphism $\o{\orb T}_{g;1}^0 \to [V/\Gl_2 \times G]$. Let $S$ be a scheme and $S \to \o{\orb T}_{g;1}^0$ a morphism given by $(P; \sigma; \phi \from C \to P)$. Let $E = \dual{(\phi_* O_C/O_P)}$. By the structure theorem for triple covers (\autoref{thm:structure}), the cover $C \to P$ gives a global section
$v$ of $\Sym^3(E) \otimes \det \dual E$. Set 
\[T = \Isom_S(\sigma^* E, \Lambda) \times_S \Isom_S((P, \sigma), (\P^1, \infty)).\]
Then $T \to S$ is a $\Gl_2 \times G$ torsor. Over $T$, we have canonical identifications 
\[ \sigma_T^* E_T \isom \Lambda \otimes_k O_T \text{ and } (P_T, \sigma_T) \isom (\P^1 \times T, \infty \times T).\]
Since $C \to P$ is a family of balanced triple covers, $E$ is fiberwise isomorphic to $\Gamma$. On $T$, the isomorphism $\sigma^*E_T \isom \Lambda \otimes_k O_T$ gives a canonical isomorphism
\[ E_T \isom  \Gamma_T.\]
Thus, we may treat $v_T$ as a global section of $\Sym^3\Gamma_T \otimes \det \dual \Gamma_T$, or equivalently as a map $T \to V$. By construction, this map is $\Gl_2 \times G$ equivariant. We thus have a morphism
\begin{equation}\label{eqn:final_even_map}
  q \from \o{\orb T}_{g;1}^0 \to [ V / \Gl_2 \times G].
\end{equation}
\begin{proposition}\label{thm:final_even_map}
  The morphism $q$ in \eqref{eqn:final_even_map} is representable and an injection on $k$-points.
\end{proposition}
Recall that $k$-points of $[X/H]$ are just orbits of the action of $H(k)$ on $X(k)$.
\begin{proof}
  Let $p \from \spec k \to \o{\orb T}_{g;1}^0$ be a point. For representbility, we must show that the map $\Aut_p \to \Aut_{q(p)}$ is injective. Say $p$ is given by $(P; \sigma; \phi \from C \to P)$. Set $E = \dual{(\phi_*O_C/O_P)}$ and pick identifications $E|_\sigma \cong \Lambda$ and $(P, \sigma) \cong (\P^1, \infty)$. Consider an element $\psi \in \Aut_p$. Then $\psi$ consists of $(\psi_1, \psi_2)$, where $\psi_1 \from (\P^1;\infty) \to (\P^1; \infty)$ is an automorphism and $\psi_2 \from C \to C$ is an automorphism over $\psi_1$. To understand the image of $\psi$ in $\Aut_{q(p)}$, consider the map on algebras
  \[\psi_2^\# \from \psi_1^*\phi_* O_C \to \phi_*O_C\]
  dual to $\psi_2 \from C \to C$. The map $\psi_2^\#$ induces a map $\alpha \from E \to \psi_1^* E$. Then the image of $\psi = (\psi_1,\psi_2)$ is just $(\alpha|_\sigma, \psi_1)$.

  Suppose that $(\alpha|_\sigma, \psi_1) = \id$. Then $\psi_1 = \id$. Furthermore, the fact that $E$ is balanced and $\alpha|_\sigma = \id$ implies that $\alpha = \id$. From the sequence
  \[ 0 \to O_P \to \phi_* O_C \to \dual E \to 0, \]
  it follows that $\psi_2 = \id$. Thus $\Aut_p \to \Aut_{q(p)}$ is injective.
  
  It is clear that $q$ is injective on $k$-points---two sections $v_1, v_2$ in the same orbit of $\Gl_2 \times G$ clearly give isomorphic marked covers.
\end{proof}

\begin{theorem}\label{thm:final_even}
  Let $g = 2(h-1)$, where $h \geq 1$. Then
  \[\o{\orb T}_{g;1}^0 \cong [(\A^{2g+3}\setminus 0) / (\s_3 \times \G_m)],\]
  where $\G_m$ acts by weights 
  \[1,2,\dots,h, 1,2, \dots, h, 1,2, \dots, h, 2,3 \dots, h.\]
  The space $\o{\sp T}_{g;1}^0$ is the quotient of the weighted projective space 
  \[\P\left(1,\dots,h, 1, \dots, h, 1, \dots, h, 2, \dots, h\right)\]
  by an action of $\s_3$. In particular, $\o {\sp T}_{g;1}^0$ is a unirational, Fano variety.
\end{theorem}
We use the following lemma in the proof to simplify a group action.
\begin{lemma}\label{thm:simplify_action}
  Let $X$ be a normal variety over $k$ with the action of a connected algebraic group $H$. Let $X' \subset X$ be a reduced and irreducible subvariety and $H' \subset H$ a subgroup such that the action of $H'$ restricts to an action on $X'$. If $[X'/H'] \to [X/H]$ is a bijection on $k$-points, then it is an isomorphism.
\end{lemma}
\begin{proof}
  The map $[X'/H'] \to [X/H]$ is clearly representable. Set $Y =  X \times_{[X/H]}[X'/H']$. It suffices to prove that $Y \to X$ is an isomorphism. We have the diagram
  \[ 
  \begin{tikzpicture}
    \matrix (m) [matrix of math nodes, row sep=2em, column sep=2em]{
      X' \times H & X' \\
      Y  & {[X'/H']} \\
      X  & {[X/H]}\\
    };
    \path [->, map]
    (m-1-1) edge (m-1-2)
    (m-1-1) edge (m-2-1)
    (m-1-2) edge (m-2-2)
    (m-2-1) edge (m-2-2)
    (m-2-1) edge (m-3-1)
    (m-2-2) edge (m-3-2)
    (m-3-1) edge (m-3-2);    
    \cartsquare{(m-1-1)}{(m-2-2)};
    \cartsquare{(m-2-1)}{(m-3-2)};
  \end{tikzpicture}.
  \]
  The smooth morphism $Y \to [X'/H']$ shows that $Y$ is reduced and the surjective morphism $X' \times H \to Y$ shows that it is irreducible. 
  Furthermore, $Y \to X$ induces a bijection on $k$-points. The quasi-finite map $Y \to X$ can be factored as $Y \into \o Y \to X$, where the first is a dense open inclusion and the second a finite map. Since $X$ is normal, Zariski's main theorem implies that $\o Y \to X$ is an isomorphism. But then $Y \into \o Y$ is a bijection on $k$-points. It follows that $Y = \o Y$.
\end{proof}

\begin{proof}[Proof of \autoref{thm:final_even}]
  Retain the notation at the beginning of \autoref{sec:final_even}. Let $U \subset V$ be the open subset consisting of $v \in V$ whose associated triple cover is \'etale over $\infty$. Then $U$ is invariant under the $\Gl_2 \times G$ action and $q \from \o{\orb T}_{g;1}^0 \to [V/\Gl_2 \times G]$ lands in $[U/\Gl_2 \times G]$.

  We now simplify the presentation $[U/\Gl_2 \times G]$. Choose coordinates on $\Lambda$: say $s = \o{(1,-1,0)}$ and $t = \o{(0,1,-1)}$. Write points of $V$ explicitly as
  \[ v_{s,t} = (a s^3 + b s^2t + c st^2 + d t^3) \otimes (s^* \wedge t^*),\]
  where $a, b, c, d \in H^0(O_{\P^1}(h))$. Let $a = \sum a_i X^{h-i}Y^i$, where $a_i \in k$, and similarly for $b, c$ and $d$. Let $W \subset U$ be the closed subvariety defined by
  \[ v_{s,t}(1,0) = st(t+s) \otimes (s^*\wedge t^*) \text{ and } 2a_1 +2d_1-b_1-c_1 = 0.\]
  The first equation specifies $v$ over $\infty$; the second is the result of imposing the condition that the branch divisor of the triple cover given by $v$ be centered around $0$. Explicitly, the points of $W$ have the form
\begin{align*}
  &\left(\sum_{i=1}^h a_{h-i} X^iY^{h-i}\right) s^3 \otimes (s^* \wedge t^*) \\
  +&\left(X^h + \sum_{i=1}^h b_{h-i} X^iY^{h-i}\right) s^2t \otimes (s^* \wedge t^*) \\
   +&\left(X^h + \sum_{i=1}^h c_{h-i} X^iY^{h-i}\right) st^2 \otimes (s^* \wedge t^*) \\
   +&\left(\sum_{i=1}^h d_{h-i} X^iY^{h-i}\right) t^3 \otimes (s^* \wedge t^*). 
\end{align*}
where $2a_{1} + 2d_{1} -b_{1} - c_{1} = 0$. The action of $\s_3 \subset \Gl_2$ by permuting the three coordinates of $\Lambda \subset k^3$ and the action of $\G_m \subset G$ by scaling $(X,Y) \mapsto (X,tY)$ restrict to actions on $W$.

\begin{claim}
  The map $i \from [W / \s_3 \times \G_m] \to [U / \Gl_2 \times G]$ is an isomorphism
\end{claim}
\begin{proof}
  By \autoref{thm:simplify_action}, it suffices to check that it is a bijection on $k$-points.

  We first check that $i$ is injective on $k$-points. Said differently, we want to check that two points of $W$ that are related by the action of $\Gl_2 \times G$ are in fact related by the action of $\s_3 \times \G_m$.
Let $w_1, w_2 \in W$  and $\psi = (\psi_1, \psi_2) \in \Gl_2 \times G$ be such that $w_1 = \psi w_2$. Then the linear isomorphism $\psi_1 \from k\langle s, t\rangle \to k\langle s, t \rangle$ takes $st(s+t) \otimes s^*\wedge t^*$ to itself. It is easy to check that it must lie in the $\s_3 \subset \Gl_2$. Secondly, observe that for $\psi_2:(X,Y) \mapsto (X+\beta Y,\alpha Y)$, we have
\begin{equation}\label{eqn:centered_branching}
\psi_2^{-1}(2a_{1}+2d_{1}-b_{1}-c_{1}) = \alpha(2a_{1}+2d_{1}-c_{1}-d_{1}) - h\beta.
\end{equation}
By the second defining condition for $W$, we must have $\beta = 0$. In other words, $\psi \in \s_3 \times \G_m$.

We now check that $i$ is surjective on $k$-points. Said differently, we want to check that every $\Gl_2\times G$-orbit of $U$ has a representative in $W$. Take a point $v \in U$.  Since $v$ describes a cover \'etale over $\infty$, the homogeneous cubic in $v_{s,t}(1,0)$ has distinct roots. By a suitable linear automorphism of $k\langle s,t \rangle$, the element $v_{s,t}(1,0)$ can thus be brought into the form $st(s+t)\otimes s^*\wedge t^*$. Secondly, \eqref{eqn:centered_branching} shows that we can make $2a_{1}+2d_{1}-b_{1}-c_{1} = 0$ after a suitable translation $(X,Y) \mapsto (X+\beta Y, Y)$. 
\end{proof}

Returning to the main proof, we have a morphism 
\[q \from \o{\orb T}_{g;1}^0 \to [U/\Gl_2 \times G] = [W / \s_3 \times \G_m],\]
which is representable and injective on $k$-points. Denote by $0 \in W$ the point corresponding to $a_i = b_i = c_i = d_i = 0$ for all $i = 1, \dots, h$. Its corresponding cover has concentrated branching over $0 \in \P^1$ and $\mu$ invariant zero. Hence $q$ factors through
\begin{equation}\label{eqn:final_iso}
  q \from \o{\orb T}_{g;1}^0 \to [(W \setminus 0) / (\s_3 \times \G_m)].
\end{equation}
The right side $[(W \setminus 0) / (\s_3 \times \G_m)]$ is smooth and proper over $k$. Indeed, it is a weighted projective stack modulo an action of $\s_3$. Thus, the morphism $q$ in \eqref{eqn:final_iso} is a representable, proper morphism between two smooth stacks of the same dimension which is an injection on $k$-points. By Zariski's main theorem, it must be an isomorphism.

Finally, note that 
\[W \cong \A^{2g+3} = \A(a_1, \dots, a_h, b_1, \dots, b_h, c_1, \dots, c_h, d_2, \dots, d_h);\]
the $d_1$ can be dropped owing to the condition $2a_1+2d_1-b_1-c_1 = 0$. The $\G_m$ acts by weight $i$ on $a_i, b_i, c_i, d_i$, as desired.
\end{proof}

\subsection{The case of odd $g$}\label{sec:final_odd}
Let $g = 2h -1$. In this case, we do not have quite as explicit a description of the final model as in the case of even $g$. Nevertheless, we prove that it is a Fano fibration over $\P^1$.

The morphism to $\P^1$ is defined by the cross-ratio as in \autoref{sec:cross_ratio}. We quickly recall the construction. Set $V = k^{\oplus 3}/k$, to be thought of as the space of functions on $\{1,2,3\} \times \spec k$ modulo the constant functions. Then there is an action of $\s_3$ on $V$ and an induced action of $\s_3$ on $\P_{\sub} V$. The cross-ratio map 
\[\chi \from \o {\orb T}_{g;1}^1 \to [\P_\sub V/\s_3]\]
is defined as follows. Let $S \to \o{\orb T}_{g;1}^1$ be a morphism given by  $(\pi \from P \to S; \sigma; \phi \from C \to P)$. Set $F = \phi_*O_C/O_P$. Since $F$ is fiberwise isomorphic to $O(-h) \oplus O(-h-1)$, we see that $\pi_* (F \otimes O_P(h\sigma))$ is a line bundle on $S$. Consider the map
\begin{equation}\label{eqn:cross_ratio_map}
  \pi_*(F \otimes O_P(h\sigma)) \otimes \sigma^* O_P(-h\sigma) \to \sigma^*F.
\end{equation}
It is easy to see that this remains injective at every point of $S$. Moreover, passing to $Z = \Isom_S(C|_\sigma, \{1,2,3\})$, we have an identification $\sigma_Z^*F_Z \isom V \otimes_k O_Z$. Hence \eqref{eqn:cross_ratio_map} yields a map $Z \to \P_{\sub}V$, which is by construction $\s_3$ equivariant. We thus get a map $S \to [\P_\sub V/\s_3]$.

The map $\chi \from \o{\orb T}_{g;1}^1 \to [\P_\sub V/\s_3]$ induces a map on the coarse spaces:
\begin{equation}\label{eqn:cross_ratio_coarse}
  \chi \from \o T_{g;1}^1 \to \P_\sub V/\s_3 \cong \P^1.
\end{equation}
\begin{theorem}\label{thm:final_odd}
  Consider the cross-ratio map $\chi \from \o{\sp T}_{g;1}^1 \to \P^1$ as in \eqref{eqn:cross_ratio_coarse}. Then,
  \begin{compactenum}
  \item $\chi^*O_{\P^1}(1) = \frac{3D_1}{2}$;
  \item fibers of $\chi$ are Fano varieties of Picard rank one.
  \end{compactenum}
\end{theorem}
\begin{proof}
  We use the setup introduced above, assuming furthermore that $S$ is a smooth curve. On $Z$, we have
  \[
  \pi_*(F\otimes O_P(h\sigma)) \otimes \sigma^* O_P(-h\sigma) \to \sigma^*F = V \otimes_k O_Z.
  \]
  For the first relation, note that $\P_\sub V \to \P_\sub V/\s_3 = \P^1$ is a degree six cover. Hence, it suffices to prove that the pullback of $O_{\P_\sub V}(1)$ to $Z$ has class $D_1/4$. But the class of this pullback is just
  \begin{equation}\label{eqn:two_summands}
    -c_1(\pi_*F \otimes O_P(h\sigma)) - c_1(\sigma^* O_P(-h\sigma)).
  \end{equation}
  Since $R^1\pi_*(F \otimes O_P(h\sigma)) = 0$, the first summand in \eqref{eqn:two_summands} is a simple calculation using Grothendieck--Riemann--Roch (see \eqref{right_side} in \autoref{sec:generic_positivity_second_proof} for this calculation in a more general case). The result is
  \[c_1(\pi_*F \otimes O_P(h\sigma)) = \frac{h^2}{(2h+1)^2}c_1^2 - c_2.\]
  The second summand in \eqref{eqn:two_summands} is simply $-h\pi_*[\sigma^2]$. Using \autoref{thm:divisor_relations}, we get
  \begin{align*}
    -c_1(\pi_*F \otimes O_P(h\sigma)) - c_1(\sigma^* O_P(-h\sigma)) &= 
    c_2-\frac{h^2+h}{(2h+1)^2}c_1^2 \\
    &= \frac{D_1}{4}.
  \end{align*}
  The first relation is thus proved.

  Finally, set $K = K_{\o{\sp T}_{g;1}^1}$ and let $Y \subset \o {\sp T}_{g;1}^1$ be a fiber of $\chi$. Then all curves in $Y$ have intersection number zero with $D_1$; it follows that they are (numerically) rational multiples of each other. Hence the Picard rank of $Y$ is one.

  Since $\ideal{D_1}$ and $\ideal{D_3}$ bound the ample cone of $\o {\sp T}_{g;1}$ and $D_1|_Y \cong 0$, the ray $\ideal{D_3}$ must be positive on $Y$. But $D_3$ and $K$ are on the opposite sides of the line spanned by $D_1$; hence $K|_Y = K_Y$ is anti-ample.
\end{proof}

Finally, we collect the pieces together for the chamber decomposition.   Recall that the Mori chamber $\Mor(\beta)$ of a birational contraction $\beta \from X \dashrightarrow Y$ is the cone spanned by the pullback of the  ample cone of $Y$ and the exceptional divisors of $\beta$. If the birational map $X \dashrightarrow Y$ is clear, we abuse notation and call $\Mor(\beta)$ the Mori chamber of $Y$. 

\begin{theorem}\label{thm:mori_cones}
  Let $0 \leq l < g$ be such that $l \equiv g \pmod 2$.
  \begin{compactenum}
  \item\label{eqn:generic} For $l > 0$, the interior of the cone $\ideal{D_l, D_{l+2}}$ is the Mori chamber of the model $\o T_{g;1}^l$.
  \item\label{eqn:last_even} For even $g$, the cone $\ideal{D_0, D_2}$ is the Mori chamber of the model $\o T_{g;1}^0$.

  \item\label{eqn:edges} For even (resp. odd) $g$, the ray $\ideal{D_0}$ (resp. $\ideal{D_1}$) is an edge of the effective cone.
  
  \end{compactenum}
\end{theorem}
\begin{proof}
  Since the $\o T_{g;1}^l$ are isomorphic to each other away from codimension two for $0 < l < g$, the Mori chamber of $\o T_{g;1}^l$ is simply its ample cone. Hence, \eqref{eqn:generic} follows from \autoref{thm:flip_projectivity}.

  In the case of $l = 0$, consider the Maroni contraction $\o T_{g;1}^2 \to \o T_{g;1}^0$. It is easy to check that the pullback of the ample ray of $\o T_{g;1}^0$ is the ray $\ideal{D_2}$ and the class of the Maroni divisor is a positive multiple of $D_0$ (in fact $D_0/4$); both statements follow from a simple test curve calculation, which we omit. Hence $\ideal{D_0, D_2}$ is the Mori chamber associated to $\o T_{g;1}^2 \to \o T_{g;1}^0$.

  For the last statement, first consider the case of odd $g$. By \autoref{thm:final_odd}, some positive multiple of $D_1$ is the pullback of $O_{\P^1}(1)$ along the cross-ratio map. Hence the ray $\ideal{D_1}$ must be an edge of the effective cone.

  For even $g$, some positive multiple of $D_0$ is the class of the Maroni divisor. Since the Maroni divisor is the exceptional locus of the birational morphism $\o T_{g;1}^2 \to \o T_{g;1}^0$, it follows that the ray $\ideal{D_0}$ must be the edge of the effective cone.
\end{proof}


\section{The case of a marked ramified fiber}\label{sec:ramified_fiber}
The spaces of weighted admissible covers and the spaces of $l$-balanced covers together provide a beautiful picture of the birational geometry of the space of trigonal curves with a marked unramified fiber, that is, a fiber of type $(1,1,1)$. It is natural to ask if a similar picture holds for the spaces with marked fibers of type $(1,2)$ or $(3)$. Recall that the picture consists of two parts: the first is a sequence of divisorial contractions given by spaces of weighted admissible covers, and the second is a sequence of flips given by the spaces of $l$-balanced covers. 

The recipe for constructing the first sequence of divisorial contractions is virtually the same. Let $r = 1, 2,$ or $3$ and let $\st T_{g;1/r}$ be the open and closed substack of $\st H^3 \times_{\st M}\st M_{0;b,1}$ parametrizing $(\phi \from \orb C \to \orb P; \orb P \to P; \sigma)$, where $\orb C$ is connected and $\Aut_\sigma(\orb P) = \mu_r$. The $r$ determines the ramification type of the fiber over $\sigma$ of the induced cover of the coarse spaces $C \to P$. Setting 
\[\o{\orb T}_{g;1/r}(\epsilon) = \st T_{g;1/r} \times_{\st M_{0;b,1}} \o {\orb M}_{0;b,1}(\epsilon),\]
we generalize \autoref{intro:first_seq} for all ramification types.

In this section, we indicate how to generalize \autoref{intro:second_seq}, namely how to construct the analogues of the spaces of $l$-balanced covers. As before, there is an interplay between the global geometry of marked triple covers measured by a refined Maroni invariant and the local geometry of trigonal singularities measured by a generalized $\mu$ invariant. The section is further divided as follows. In \autoref{sec:teardrop}, we recall some facts about the orbi-curve $(\orb P; \sigma)$ with $\Aut_\sigma\orb P = \mu_r$, which is the base in our families of triple covers. In \autoref{sec:ramified_invariants}, we define the Maroni invariant for a cover of $(\orb P; \sigma)$ and relate it to the classical geometry of the induced cover on the coarse spaces. In this section, we also recall the generalized $\mu$ invariant. In \autoref{sec:ramified_l-balanced}, we define $l$-balanced covers and prove the main theorem. The proof is by a formal reduction to the case $r=1$; there is little extra work.

\subsection{The teardrop curve $\orb P$}\label{sec:teardrop}
Consider the orbi-curve $\orb P$ with coarse space $\rho \from \orb P \to \P^1$, where the local picture of $\rho$ over $\infty \in \P^1$ is given by
\begin{equation}\label{eqn:quotient}
 [\spec k[v]/\mu_r] \to \spec k[x],
\end{equation}
where $\mu_r$ acts by $v \mapsto \zeta v$ and $x = v^r$. The curve $\orb P$ can also be described as the root stack
\[ \orb P = \P^1(\sqrt[r]{\infty})\]
or as a weighted projective stack
\[ \orb P = [\A^2\setminus 0/\G_m],\]
where $\G_m$ acts by weights $1$ and $r$. The name ``teardrop curve'' is inspired by the picture for $k = \C$, where $\orb P$ is imagined to be a `pinching' of the Riemann sphere to make it have conformal angle $2\pi/r$ at $\infty$ (see \autoref{fig:teardrop}).

\begin{figure}[t]
\centering
\begin{tikzpicture}[math]
  \draw [smooth, tension=1, yscale=.8]
  plot coordinates {(0,0) (-1,-2) (0,-3) (1,-2) (0,0)};
  \draw (0,0) [above] node {\sphericalangle\ 2\pi/r};
\end{tikzpicture}
\caption{The ``teardrop curve'' $\P^1(\sqrt[r]\infty)$ over $\C$}
\label{fig:teardrop}
\end{figure}
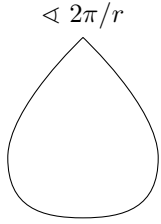

Let $\xi \subset \orb P$ be the reduced preimage of $\infty \in \P^1$. In the explicit description of $\orb P$ in \eqref{eqn:quotient}, this is the closed substack defined by $v = 0$. Let $L$ be the dual of the ideal sheaf of $\xi$ in $\orb P$ and set
\[ O_{\orb P}(d) = L^{\otimes dr},\]
for $d \in \frac{1}{r}\Z$. Thus, $L = O_{\orb P}(1/r)$.
\begin{proposition}\label{thm:teardrop_basics}
  With the notation above,
  \begin{compactenum}
  \item $\Pic(\orb P)$ is generated by $O_{\orb P}(1/r)$. 
  \item The degree of $O_{\orb P}(d)$ is $d$, for every $d \in \frac{1}{r}\Z$.
  \item The canonical sheaf (which is also the dualizing sheaf) of $\orb P$ is $O_{\orb P}(-1-1/r)$.
  \item We have 
    \[\rho^* O_{\P^1}(n) = O_{\orb P}(n), \text{ for $n \in \Z$},\]
    and
    \[ \rho_* O_{\orb P}(d) = O_{\P^1}(\lfloor d\rfloor), \text{ for $d \in \frac{1}{r}\Z$}.\]
  \item \label{part:vb_split}
Every vector bundle on $\orb P$ is isomorphic to a direct sum of line bundles.    
  \end{compactenum}
  \begin{proof}
    All of these statements are easy to see, except possibly the last one. The proof for the case of $\P^1$ sketched by Hartshorne \citep[V.2, Exercise~2.6]{hartshorne77:_algeb} works verbatim. We present the details for lack of a reference.
    
    Let $E$ be a vector bundle on $\orb P$. Since the degree of subsheaves of $\rho_*E$ is bounded above, the degree of subsheaves of $E$ is also bounded above. Let $L \subset E$ be a line bundle of maximum degree. Then the quotient $E' = E/L$ is locally free. We claim that $\Ext^1(E', L) = 0$. Then $E = E' \oplus L$, and the statement follows by induction on the rank. 

    By duality, $\Ext^1(E', L) = \dual{\Hom(L(1+1/r), E')}$. Since we have an inclusion $\Hom(L(1+1/r), E') \subset \Hom(L(1/r), E')$, proving that the latter vanishes implies that the former vanishes. On one hand, we have the exact sequence
    \[ \Hom(L(1/r), E) \to \Hom(L(1/r), E') \to \Ext^1(L(1/r), L) = \dual{\Hom(L, L(-1))} = 0.\]
    On the other hand, we know that $\Hom(L(1/r), E) = 0$ by the maximality of $\deg L$. We conclude that $\Hom(L(1/r), E') = 0$.
  \end{proof}
\end{proposition}

\subsection{The refined Maroni  and $\mu$ invariants}\label{sec:ramified_invariants}
\subsubsection{The refined Maroni invariant}\label{sec:refined_maroni}
Denote by $\orb P$ the teardrop curve $\P^1(\sqrt[r]\infty)$ as in \autoref{sec:teardrop}. \autoref{thm:teardrop_basics}~\eqref{part:vb_split} gives us a way to define the Maroni invariant for triple covers of $\orb P$.
\begin{definition}
 Let $\phi \from \orb C \to \orb P$ be a triple cover. Set $F = \phi_*O_{\orb C}/O_{\orb P}$. Then we have
\[ F \cong O_{\orb P}(-m) \oplus O_{\orb P}(-n),\]
for some $m, n \in \frac{1}{r}\Z$. Define the Maroni invariant of $\phi$ to be the difference
\[ M(\phi) = |m-n|.\]
Note that the Maroni invariant lies in $\frac{1}{r}\Z$.
\end{definition}

The refined Maroni invariant can be read off from the usual Maroni invariant of a new cover associated to $\phi \from \orb C \to \orb P$. We now explain this procedure. Choose a point $p \in \orb P$, away from $\xi$ and define $\psi \from \tw{P} \to \orb P$ to be the cyclic cover of degree $r$ branched over $p$. Explicitly, $\tw{P}$ is given by
\begin{equation}\label{eqn:cyclic_cover}
\tw{P} = \spec_{\orb P}\left( \bigoplus_{i = 0}^{r-1} O_{\orb P}(-i/r) \right),
\end{equation}
where the ring structure is given by a section of $O_{\orb P}(1)$ vanishing at $p$. It is easy to see that $\tw P \cong \P^1$. Set
\[ \tw C = \orb C \times_{\orb P} \tw P \]
with the induced map $\tw \phi \from \tw C \to \tw P$.
\begin{proposition}\label{thm:refined_maroni_cyclic_cover}
  With the above notation, we have $M(\phi) = M(\tw\phi)/r$.
\end{proposition}
\begin{proof}
  By construction, we have $\psi^*O_{\orb P}(d) = O_{\tw P}(dr)$. Since $\tw\phi_* O_{\tw C} = \psi^* \phi_* O_{\orb C}$, the statement follows.
\end{proof}

At first sight, the refined Maroni invariant seems to be an artifact of the stacky way of keeping track of ramification. However, it can be described purely in terms of the geometry of the cover of the coarse spaces. We now describe this connection.

As before, let $\rho \from \orb P = \P^1(\sqrt[r]\infty) \to \P^1$ be the teardrop curve. Consider a $k$-point of $\st T_{g;1/r}$ given by $(\orb P \to \P^1; \infty; \orb C \to \orb P)$ and let $C \to \P^1$ be the induced cover on the coarse spaces. To lighten notation, we treat $O_{\orb C}$ and $O_C$ as sheaves on $\orb P$ and $P$ respectively, omitting the pushforward symbols. Then $O_{C} = \rho_* O_{\orb C}$. Therefore, if we have the splitting
\[ O_{\orb C} \cong O_{\orb P} \oplus O_{\orb P}(-m) \oplus O_{\orb P}(-n),\]
then we deduce the splitting
\[ O_C \cong O_{\P^1} \oplus O_{\P^1}(\lfloor-m\rfloor) \oplus O_{\P^1}(\lfloor -n \rfloor).\]
Thus, the splitting type of $\orb C \to \orb P$ determines the splitting type of $C \to \P^1$. For $r > 1$, however, it carries a bit more information. Let us understand what sort of extra information is contained in this refinement. To that end, observe that the data of the splitting type of $C \to \P^1$ is equivalent to the data of the sequence $\ideal{h^0(O_C(l)) \mid l \in \Z}$. 

First consider the case $r = 3$. In this case, the map $C \to \P^1$ is totally ramified over $\infty$. Denoting by $x \in C$ the unique point over $\infty \in \P^1$, we have $O_C(1) \cong O_C(3x)$.  Therefore, the data of the splitting type of $C \to \P^1$ is the data of the sequence $\ideal{h^0(O_C(3lx)) \mid l \in \Z}$. On the other hand, the data of the splitting type of $\orb C \to \orb P$ is precisely the data of the sequence $\ideal{h^0(O_C(lx)) \mid l \in \Z}$. Thus the refined Maroni invariant encodes the so-called \emph{Weierstrass gap sequence} of the point $x$ on $C$.

Now consider the case $r = 2$. In this case, the map $C \to \P^1$ has ramification type $(2,1)$ over $\infty$. Let the preimage of $\infty$ be $2x+y$, where $x,y \in \P^1$. As before, the data of the splitting type of $C \to \P^1$ is the data of the sequence $\ideal{h^0(O_C(l(2x+y)) \mid l \in \Z}$. On the other hand, the splitting type of $\orb C \to \orb P$ encodes, in addition, the data of $h^0(O_C(l(2x+y) - x))$ and $h^0(O_C(l(2x+y) + x))$, for $l \in \Z$.

\subsubsection{The $\mu$ invariant}
  Let $\phi \from \orb C \to \orb P$ be a triple cover, \'etale except possibly over a point $p \in \orb P$ different from $\xi$. In this case, we say that $\phi$ has \emph{concentrated branching}. Define the $\mu$ invariant of $\phi$ to be the $\mu$ invariant of the singularity of $\orb C \to \orb P$ over $p$ as per the generalization in \autoref{sec:triple_singularities}. We recall this generalization in the current context. Let $\Delta = \spec \compl{O_{\orb P, p}}$ be the formal disk around $p$ and set $C = \orb C \times_{\orb P} \Delta$. Let $\tw C \to C$ be the normalization. Then $\tw C \to \Delta$ is not necessarily \'etale. We choose a cover $\Delta' \to \Delta$ of degree $d$ such that the normalization of $C' = C \times_{\Delta'} \Delta$ is \'etale over $\Delta'$. Then, by definition, we have
\begin{equation}\label{eqn:refined_mu}
 \mu(C \to \Delta) = \frac{1}{r} \mu(C' \to \Delta').
\end{equation}
Note that the $\mu$ invariant lies in $\frac{1}{r}\Z$.

As in the case of the refined Maroni invariant, the $\mu$ invariant can be read off from that of the modified cover $\tw \phi \from \tw C \to \tw P$. We recall the procedure. Define the cyclic cover $\tw P \to \orb P$ ramified only over $p$, as in \eqref{eqn:cyclic_cover}. Set 
\[\tw C = \orb C \times_{\orb P} \tw P\]
with the induced map $\tw \phi \from \tw C \to \tw P$. Let $q \in \tw P$ be the unique point over $p \in \orb P$. See that $\tw \phi \from \tw C \to \tw P$ has concentrated branching over $q$.
\begin{proposition}\label{thm:refined_mu_cyclic_cover}
  With the above notation, we have $\mu(\phi) = \frac{1}{r}\mu(\tw\phi)$.
\end{proposition}
\begin{proof}
  This follows immediately from \eqref{eqn:refined_mu}.
\end{proof}

\subsection{The stack of $l$-balanced covers}\label{sec:ramified_l-balanced}
Having defined the Maroni invariant and the $\mu$ invariant for covers of the teardrop curve, we are ready to formulate and prove the analogue of \autoref{thm:tg3d}. We begin by defining $l$-balanced covers. The definition follows \autoref{def:l-balanced} almost verbatim.

\begin{definition}\label{def:refined_l-balanced}
  Let $l \in \frac{1}{r}\Z$ be non-negative and $\orb P \cong \P^1(\sqrt[r]\infty)$ the teardrop curve with the stacky point $\xi$ as in \autoref{sec:teardrop}. Let $\phi \from \orb C \to \orb P$ be a triple cover, \'etale over $\xi$.  We say that $\phi$ is \emph{$l$-balanced} if the following two conditions are satisfied.
  \begin{compactenum}
  \item The Maroni invariant of $\phi$ is at most $l$.
  \item If $\phi$ has concentrated branching, then its $\mu$ invariant is greater than $l$.
  \end{compactenum}
\end{definition}

We can reduce \autoref{def:refined_l-balanced} to the case of an unramified fiber, namely \autoref{def:l-balanced}, by looking at a modified cover $\tw C \to \tw P$. Let $p \in \orb P$ be a point contained in $\br(\phi)$ and $\tw P \to \orb P$ the cyclic cover of degree $r$ branched over $p$ as in \eqref{eqn:cyclic_cover}. Set 
\[ \tw C = C \times_{\orb P}\tw P,\]
with the induced map $\tw \phi \from \tw C \to \tw P$.
\begin{proposition}\label{thm:l-balanced_cyclic_cover}
  With the above notation, the cover $\phi$ is $l$-balanced if and only if the cover $\tw\phi$ is $rl$-balanced in the sense of \autoref{def:l-balanced}.
\end{proposition}
\begin{proof}
  Combine \autoref{thm:refined_maroni_cyclic_cover} and \autoref{thm:refined_mu_cyclic_cover}.
\end{proof}

Recall that $\st T_{g;1/r} \subset \st H^3$ is the open and closed substack whose geometric points parametrize covers $(\orb P \to P; \sigma_1; \phi \from \orb C \to \orb P)$, where $\Aut_{\sigma_1}(\orb P) = \mu_r$ and $\orb C$ is a connected curve of genus $g$.
\begin{definition}
  Define $\o{\orb T}_{g;1/r}^l$ to be the category whose objects over a scheme $S$ are
  \[ \o{\orb T}_{g;1/r}^l(S) = \{(P \to S; \orb P \to P; \sigma; \phi \from C \to P)\},\]
  such that
  \begin{compactenum}
  \item $(P \to S; \orb P \to P; \sigma; \phi \from \orb C \to \orb P)$ is an object of $\st T_{g;1/r}(S)$;
  \item $P \to S$ is smooth, that is, a $\P^1$ bundle;
  \item For all geometric points $s \to S$, the cover $\phi_s \from \orb C_s \to \orb P_s$ is $l$-balanced.
  \end{compactenum}
\end{definition}
See that $\o{\orb T}_{g;1/r}^l \subset \st T_{g;1/r}$ is an open substack.

\begin{theorem}\label{thm:ram_tg3d}
   $\o{\orb T}_{g;1/r}^l$ is an irreducible Deligne--Mumford stack, smooth and proper over $k$.
\end{theorem}
By the jugglery of modifying a cover $\orb C \to \orb P$ to get a cover $\tw C \to \tw P$ used many times in \autoref{sec:ramified_invariants}, the major steps in the proof can be reduced to the analogous steps in the case of $r = 1$. As a result, little hard work goes into the proof of \autoref{thm:ram_tg3d}.
\begin{proof}
  We divide the proof into steps.
  \begin{asparaenum}
  \item[\textbf{That $\o{\orb T}_{g;1/r}^l$ is irreducible, smooth, and of finite type.}] This part is identical to that in the proof of \autoref{thm:tg3d}. We skip the details.      
    \item[\textbf{That $\o{\orb T}_{g;1/r}^l$ is separated.}]
      We use the valuative criterion. Let $\Delta = \spec R$ be the spectrum of a DVR, with special point $0$, generic point $\eta$ and residue field $k$. Consider two morphisms $\Delta \to \o{\orb T}_{g;1/r}^l$ given by $(\orb P_i \to P_i \to \Delta; \sigma_i; \phi_i \from \orb C_i \to \orb P_i)$ for $i=1,2$. Let $\psi_\eta$ be an isomorphism of this data over $\eta$. We must show that $\psi_\eta$ extends to an isomorphism over all of $\Delta$. We may replace $\Delta$ by a finite cover, if we so desire. 
      
      Let $\Sigma_i \subset P_i$ be the branch divisor of $\phi_i$. By passing to a cover of $\Delta$ if necessary, assume that we have sections $p_i \from \Delta \to \Sigma_i$ which agree over $\eta$, that is $\psi_\eta^P \circ p_1|_\eta = p_2|_\eta$. Denote by $\xi_i \subset \orb P_i$ the reduced preimage of $\sigma_i$ and consider the cyclic triple cover $\tw P_i \to \orb P_i$ defined by
      \[ \tw P_i = \spec_{\orb P_i}\left( \bigoplus_{j=0}^{r-1} O_{\orb P_i}(-\xi_i) \right),\]
      where the ring structure is given by a section of $O_{\orb P_i}(r\xi_i) \cong O_{\orb P_i}(p_i)$ vanishing along $p_i$. Set
      \[ \tw C_i = \orb C_i \times_{\orb P_i} \tw P_i,\]
      with the induced map $\tw\phi_i \from \tw C_i \to \tw P_i$. The reduced preimage $\tw \sigma_i$ of $\sigma_i$ gives a section $\tw \sigma_i \from \Delta \to \tw P_i$. By \autoref{thm:l-balanced_cyclic_cover}, $(\tw P_i; \tw\sigma_i; \tw\phi_i \from \tw C_i \to \tw P_i)$ is a family of $rl$-balanced covers, for $i = 1,2$. We have an isomorphism $\tw \psi_\eta$ of this data over $\eta$. By the separatedness in \autoref{thm:tg3d}, $\tw \psi_\eta$ extends to an isomorphism $\tw\psi$ over all of $\Delta$. By descent, we conclude that $\psi_\eta$ extends to an isomorphism over all of $\Delta$.

    \item[\textbf{That $\o{\orb T}_{g;1/r}^l$ is Deligne--Mumford.}]
      Since we are in characteristic zero, it suffices to prove that a $k$-point $(\orb P \to \P^1; \infty; \phi \from \orb C \to \orb P)$ of $\o{\orb T}_{g;1/r}^l$ has finitely many automorphisms. We have a morphism of algebraic groups
      \[ \tau \from \Aut(\orb P \to \P^1; \infty; \phi \from \orb C \to \orb P) \to \Aut(\P^1),\]
      where the group on the left is proper because $\o{\orb T}_{g;1/r}^l$ is separated and the group on the right is affine. It is clear that $\ker \tau$ is finite. Hence the group on the left is finite.

    \item[\textbf{That $\o{\orb T}_{g;1/r}^l$ is proper.}]
      Let $\Delta = \spec R$ be as before. Let $(\orb P_\eta \to P_\eta; \sigma_\eta; \phi_\eta \from \orb C_\eta \to \orb P_\eta)$ be an object of $\o{\orb T}_{g;1/r}^l$ over $\eta$. We need to show that, possibly after a finite base change, it extends to an object of $\o{\orb T}_{g;1/r}^l$ over $\Delta$.

       By replacing $\Delta$ by a finite cover if necessary, assume that we have a section $p \from \eta \to \br\phi_\eta$. Define the cyclic cover $\tw P_\eta \to \orb P_\eta$ of degree $r$ branched over $p$, as before. Set
      \[\tw C_\eta = \orb C_\eta \times_{\orb P_\eta} \tw P_\eta,\]
      with the induced map $\tw \phi \from \tw C_\eta \to \tw P_\eta$ and the sections $\tw \sigma_\eta \from \eta \to \tw P_\eta$ and $\tw p_\eta \from \eta \to \tw P_\eta$ given by the reduced preimages of $\sigma_\eta$ and $p_\eta$, respectively. Then $(\tw P_\eta; \tw \sigma_\eta; \tw \phi_\eta \from \tw C_\eta \to \tw P_\eta)$ is a family of $rl$-balanced covers. By the properness in \autoref{thm:tg3d}, it extends to a family of $rl$-balanced covers $(\tw P; \tw \sigma; \tw \phi \from \tw C \to \tw P)$ over $\Delta$. The idea is to descend $\tw C \to \tw P$ down to $\orb C \to \orb P$, extending $\orb C_\eta \to \orb P_\eta$.

      We first extend $(\orb P_\eta, \sigma_\eta, p_\eta)$ over $\eta$ to $(\orb P, \sigma, p)$ over $\Delta$ so that $\tw P$ is the cyclic cover of degree $r$ of $\orb P$ branched over $p$. For this, note that $P_\eta \to \eta$ is a $\P^1$-bundle. Possibly after replacing $\Delta$ by a ramified cover of degree $r$, identify $(P_\eta, \sigma_\eta, p_\eta)$ with $(\P^1_\eta, \infty, 0)$ via an isomorphism that induces an isomorphism $(\tw P, \tw\sigma, p) \isom (\tw \P^1_\Delta, \infty, 0)$, where the latter $\tw \P^1 \cong \P^1$ covers the former $\P^1$ by $[X:Y] \mapsto [X^r:Y^r]$. Set $P = \P^1_\Delta$ with two sections $\sigma$ and $p$ given by $\infty$ and $0$, respectively. Then $(P, \sigma, p)$ is an extension of $(P_\eta, \sigma_\eta, p_\eta)$. Setting $\orb P = P(\sqrt[r]\sigma)$, we get an extension of $\orb P_\eta$. Note that the covering $\tw P \to P$ factors through $\tw P \to \orb P$, extending $\tw P_\eta \to \orb P_\eta$, and exhibiting $\tw P \to \orb P$ as a cyclic triple cover of degree $r$ branched over $p$.

      Let $\Sigma \subset P$ be the unique flat extension of the divisor $\br(\phi_\eta) \subset P_\eta$. Then the preimage of $\Sigma$ in $\tw P$ is $\br\tw\phi$. Since the divisor $\br\tw\phi$ is disjoint from $\tw \sigma$, the divisor $\Sigma$ is disjoint from $\sigma$. Thus $(P; \Sigma; \sigma)$ is an object of $\st M_{0;b,1}(\Delta)$. By the properness of $\st T_{g;1/r} \to \st M_{0;b,1}$, we have a unique extension $\phi \from \orb C \to \orb P$. It remains to prove that the central fiber is $l$-balanced.

      We claim that we have an isomorphism
      \begin{equation}\label{eqn:descent}
        \tw C \isom \orb C \times_{\orb P} \tw P \text{ over $\tw P$}.
      \end{equation}
      Indeed, by construction, we have such an isomorphism over $\eta$. Note that $\tw C \to \tw P$ and $\orb C \times_{\orb P}\tw P$ are covers of $\tw P$, isomorphic over $\eta$, and they have the same branch divisor. By the separatedness of $\st T_{g;1/r} \to \st M_{0;b,1}$, we conclude that the isomorphism $\tw C_\eta \to \orb C_\eta \times_{\orb P_\eta} \tw P_\eta$ extends over $\Delta$, yielding \eqref{eqn:descent}.

      Finally, since the central fiber of $\tw \phi \from \tw C \to \tw P$ is $rl$-balanced, we conclude that the central fiber of $\phi \from \orb C \to \orb P$ is $l$-balanced using \autoref{thm:l-balanced_cyclic_cover}.
  \end{asparaenum}
\end{proof}

\section*{Acknowledgements}
  This work is a part of my PhD thesis. I am deeply grateful to my adviser Joe Harris for his invaluable insight and immense generosity. My heartfelt thanks to Maksym Fedorchuk for inspiring this project and providing guidance at all stages. I thank Dan Abramovich, Brendan Hassett, Anand Patel, David Smyth and Ravi Vakil for valuable suggestions and conversations.

\bibliographystyle{abbrvnat}
\bibliography{CommonMath}

\begin{thebibliography}{24}
\providecommand{\natexlab}[1]{#1}
\providecommand{\url}[1]{\texttt{#1}}
\expandafter\ifx\csname urlstyle\endcsname\relax
  \providecommand{\doi}[1]{doi: #1}\else
  \providecommand{\doi}{doi: \begingroup \urlstyle{rm}\Url}\fi

\bibitem[Abramovich et~al.(2003)Abramovich, Corti, and Vistoli]{acv:03}
D.~Abramovich, A.~Corti, and A.~Vistoli.
\newblock Twisted bundles and admissible covers.
\newblock \emph{Comm. Algebra}, 31\penalty0 (8):\penalty0 3547--3618, 2003.
\newblock ISSN 0092-7872.
\newblock \doi{10.1081/AGB-120022434}.
\newblock URL \url{http://dx.doi.org/10.1081/AGB-120022434}.

\bibitem[Bolognesi and Vistoli(2012)]{Bolognesi09:_Stack_Of_Trigon_Curves}
M.~Bolognesi and A.~Vistoli.
\newblock Stacks of trigonal curves.
\newblock \emph{Trans. Amer. Math. Soc.}, Feb. 2012.
\newblock URL
  \url{http://www.ams.org/journals/tran/0000-000-00/S0002-9947-2012-05370-0/S0%
002-9947-2012-05370-0.pdf}.

\bibitem[Chen(2008)]{chen08:_moris_konts_m_bbb_p}
D.~Chen.
\newblock Mori's program for the {K}ontsevich moduli space {$\overline{\mathscr
  M}\sb {0,0}(\Bbb P\sp 3,3)$}.
\newblock \emph{Int. Math. Res. Not. IMRN}, pages Art. ID rnn 067, 17, 2008.
\newblock ISSN 1073-7928.
\newblock \doi{10.1093/imrn/rnn016}.
\newblock URL \url{http://dx.doi.org/10.1093/imrn/rnn016}.

\bibitem[Deligne(2000)]{deligne00:_letter_gan_gross_savin}
P.~Deligne.
\newblock Letter to {Gan}, {Gross} and {Savin}.
\newblock Available at \url{http://www.math.leidenuniv.nl/~jbrakenh/lowrank/},
  November 2000.

\bibitem[Deligne(2006)]{deligne06:_letter_edixh}
P.~Deligne.
\newblock Letter to {Edixhoven}.
\newblock Available at \url{http://www.math.leidenuniv.nl/~jbrakenh/lowrank/},
  June 2006.

\bibitem[Deopurkar(2012)]{deopurkar12:_compac_hurwit}
A.~Deopurkar.
\newblock Compactifications of hurwitz spaces.
\newblock \emph{Preprint}, 2012.

\bibitem[{Fedorchuk}(2010)]{fedorchuk10:_modul_hyp}
M.~{Fedorchuk}.
\newblock Moduli spaces of hyperelliptic curves with {A} and {D} singularities.
\newblock \emph{ArXiv e-prints}, July 2010.

\bibitem[Fedorchuk and Smyth(2011)]{fedorchuk11:_ample}
M.~Fedorchuk and D.~I. Smyth.
\newblock Ample divisors on moduli spaces of pointed rational curves.
\newblock \emph{J. Algebraic Geom.}, 20\penalty0 (4):\penalty0 599--629, 2011.
\newblock ISSN 1056-3911.
\newblock \doi{10.1090/S1056-3911-2011-00547-X}.
\newblock URL \url{http://dx.doi.org/10.1090/S1056-3911-2011-00547-X}.

\bibitem[Gan et~al.(2002)Gan, Gross, and Savin]{gan02:_fourier_g}
W.~T. Gan, B.~Gross, and G.~Savin.
\newblock Fourier coefficients of modular forms on {$G_2$}.
\newblock \emph{Duke Math. J.}, 115\penalty0 (1):\penalty0 105--169, 2002.
\newblock ISSN 0012-7094.
\newblock \doi{10.1215/S0012-7094-02-11514-2}.
\newblock URL \url{http://dx.doi.org/10.1215/S0012-7094-02-11514-2}.

\bibitem[Harris and
  Mumford(1982)]{Harris82:_Kodair_Dimen_Of_Modul_Space_Of_Curves}
J.~Harris and D.~Mumford.
\newblock On the {K}odaira dimension of the moduli space of curves.
\newblock \emph{Invent. Math.}, 67\penalty0 (1):\penalty0 23--88, 1982.
\newblock ISSN 0020-9910.
\newblock \doi{10.1007/Bf01393371}.
\newblock URL \url{Http://Dx.Doi.Org/10.1007/Bf01393371}.

\bibitem[Hartshorne(1977)]{hartshorne77:_algeb}
R.~Hartshorne.
\newblock \emph{Algebraic geometry}.
\newblock Springer-Verlag, New York, 1977.
\newblock ISBN 0-387-90244-9.
\newblock Graduate Texts in Mathematics, No. 52.

\bibitem[Hassett(2003)]{hassett03:_modul}
B.~Hassett.
\newblock Moduli spaces of weighted pointed stable curves.
\newblock \emph{Adv. Math.}, 173\penalty0 (2):\penalty0 316--352, 2003.
\newblock ISSN 0001-8708.
\newblock \doi{10.1016/S0001-8708(02)00058-0}.
\newblock URL \url{http://dx.doi.org/10.1016/S0001-8708(02)00058-0}.

\bibitem[Horrocks(1964)]{horrocks64:_vector}
G.~Horrocks.
\newblock Vector bundles on the punctured spectrum of a local ring.
\newblock \emph{Proc. London Math. Soc. (3)}, 14:\penalty0 689--713, 1964.
\newblock ISSN 0024-6115.

\bibitem[Keel and Mori(1997)]{keel97:_quotien}
S.~Keel and S.~Mori.
\newblock Quotients by groupoids.
\newblock \emph{Ann. of Math. (2)}, 145\penalty0 (1):\penalty0 193--213, 1997.
\newblock ISSN 0003-486X.
\newblock \doi{10.2307/2951828}.
\newblock URL \url{http://dx.doi.org/10.2307/2951828}.

\bibitem[Koll{\'a}r(1990)]{kollar90:_projec}
J.~Koll{\'a}r.
\newblock Projectivity of complete moduli.
\newblock \emph{J. Differential Geom.}, 32\penalty0 (1):\penalty0 235--268,
  1990.
\newblock ISSN 0022-040X.
\newblock URL
  \url{http://projecteuclid.org/getRecord?id=euclid.jdg/1214445046}.

\bibitem[Koll{\'a}r(1996)]{kollar96:_ration}
J.~Koll{\'a}r.
\newblock \emph{Rational curves on algebraic varieties}, volume~32 of
  \emph{Ergebnisse der Mathematik und ihrer Grenzgebiete. 3. Folge. A Series of
  Modern Surveys in Mathematics [Results in Mathematics and Related Areas. 3rd
  Series. A Series of Modern Surveys in Mathematics]}.
\newblock Springer-Verlag, Berlin, 1996.
\newblock ISBN 3-540-60168-6.

\bibitem[Laumon and Moret-Bailly(2000)]{laumon00:_champ}
G.~Laumon and L.~Moret-Bailly.
\newblock \emph{Champs alg\'ebriques}, volume~39 of \emph{Ergebnisse der
  Mathematik und ihrer Grenzgebiete. 3. Folge. A Series of Modern Surveys in
  Mathematics [Results in Mathematics and Related Areas. 3rd Series. A Series
  of Modern Surveys in Mathematics]}.
\newblock Springer-Verlag, Berlin, 2000.
\newblock ISBN 3-540-65761-4.

\bibitem[Miranda(1985)]{miranda85:_tripl_cover_algeb_geomet}
R.~Miranda.
\newblock Triple covers in algebraic geometry.
\newblock \emph{American Journal of Mathematics}, 107\penalty0 (5):\penalty0
  pp. 1123--1158, 1985.
\newblock ISSN 00029327.
\newblock URL \url{http://www.jstor.org/stable/2374349}.

\bibitem[Moriwaki(1998)]{moriwaki98:_relat_bogom}
A.~Moriwaki.
\newblock Relative {B}ogomolov's inequality and the cone of positive divisors
  on the moduli space of stable curves.
\newblock \emph{J. Amer. Math. Soc.}, 11\penalty0 (3):\penalty0 569--600, 1998.
\newblock ISSN 0894-0347.
\newblock \doi{10.1090/S0894-0347-98-00261-6}.
\newblock URL \url{http://dx.doi.org/10.1090/S0894-0347-98-00261-6}.

\bibitem[Poonen(2008)]{poonen08}
B.~Poonen.
\newblock The moduli space of commutative algebras of finite rank.
\newblock \emph{J. Eur. Math. Soc. (JEMS)}, 10\penalty0 (3):\penalty0 817--836,
  2008.
\newblock ISSN 1435-9855.
\newblock \doi{10.4171/JEMS/131}.
\newblock URL \url{http://dx.doi.org/10.4171/JEMS/131}.

\bibitem[Smyth(2011{\natexlab{a}})]{smyth11:_modul_i}
D.~I. Smyth.
\newblock Modular compactifications of the space of pointed elliptic curves
  {I}.
\newblock \emph{Compos. Math.}, 147\penalty0 (3):\penalty0 877--913,
  2011{\natexlab{a}}.
\newblock ISSN 0010-437X.
\newblock \doi{10.1112/S0010437X10005014}.
\newblock URL \url{http://dx.doi.org/10.1112/S0010437X10005014}.

\bibitem[Smyth(2011{\natexlab{b}})]{smyth11:_modul_ii}
D.~I. Smyth.
\newblock Modular compactifications of the space of pointed elliptic curves
  {II}.
\newblock \emph{Compositio Mathematica}, 147\penalty0 (06):\penalty0
  1843--1884, 2011{\natexlab{b}}.
\newblock \doi{10.1112/S0010437X11005549}.
\newblock URL \url{http://dx.doi.org/10.1017/S0010437X11005549}.

\bibitem[Stankova-Frenkel(2000)]{Stankova-Frenkel00:_Modul_Of_Trigon_Curves}
Z.~E. Stankova-Frenkel.
\newblock Moduli of trigonal curves.
\newblock \emph{J. Algebraic Geom.}, 9\penalty0 (4):\penalty0 607--662, 2000.
\newblock ISSN 1056-3911.

\bibitem[Teissier(1980)]{teissier80:_resol_i}
B.~Teissier.
\newblock R\'esolution simultan\'ee, {I}.
\newblock In \emph{S\'eminaire sur les Singularit\'es des Surfaces}, volume 777
  of \emph{Lecture Notes in Mathematics}, pages 71--81. Springer, Berlin, 1980.

\end{thebibliography}

\end{document}